\documentclass[italian,english]{amsart}

\usepackage[T1]{fontenc}
\usepackage[utf8]{inputenc}

\usepackage{babel}
\usepackage{prettyref}
\usepackage{enumitem}
\usepackage{amssymb}
\usepackage[pdfborder={0 0 0},colorlinks=false]{hyperref}
\usepackage{tikz}

\numberwithin{equation}{section}
\numberwithin{figure}{section}

\theoremstyle{plain}
\newtheorem{thm}{Theorem}[section]
\newtheorem{prop}[thm]{Proposition}
\newtheorem{lem}[thm]{Lemma}
\newtheorem{cor}[thm]{Corollary}
\makeatletter
\newcommand\nt@name{Theorem}
\newtheorem*{nt@thm}{\nt@name}
\newenvironment{namedthm}[1][Theorem]
{\renewcommand\nt@name{#1}
  \begin{nt@thm}}
  {\end{nt@thm}}
\makeatother

\theoremstyle{definition}
\newtheorem{defn}[thm]{Definition}
\newtheorem{example}[thm]{Example}
\newtheorem{fact}[thm]{Fact}
\newtheorem{question}[thm]{Question}

\theoremstyle{remark}
\newtheorem{rem}[thm]{Remark}
\newtheorem{notation}[thm]{Notation}

\newrefformat{def}{Definition~\ref{#1}}
\newrefformat{prop}{Proposition~\ref{#1}}
\newrefformat{fact}{Fact~\ref{#1}}
\newrefformat{cor}{Corollary~\ref{#1}}
\newrefformat{lem}{Lemma~\ref{#1}}
\newrefformat{rem}{Remark~\ref{#1}}
\newrefformat{sub}{Subsection~\ref{#1}}
\newrefformat{subsec}{Subsection~\ref{#1}}
\newrefformat{nota}{Notation~\ref{#1}}
\newrefformat{ex}{Example~\ref{#1}}
\newrefformat{que}{Question~\ref{#1}}

\DeclareFontFamily{OMX}{MnSymbolE}{}
\DeclareSymbolFont{MnLargeSymbols}{OMX}{MnSymbolE}{m}{n}
\SetSymbolFont{MnLargeSymbols}{bold}{OMX}{MnSymbolE}{b}{n}
\DeclareFontShape{OMX}{MnSymbolE}{m}{n}{
    <-6>  MnSymbolE5
   <6-7>  MnSymbolE6
   <7-8>  MnSymbolE7
   <8-9>  MnSymbolE8
   <9-10> MnSymbolE9
  <10-12> MnSymbolE10
  <12->   MnSymbolE12
}{}
\DeclareFontShape{OMX}{MnSymbolE}{b}{n}{
    <-6>  MnSymbolE-Bold5
   <6-7>  MnSymbolE-Bold6
   <7-8>  MnSymbolE-Bold7
   <8-9>  MnSymbolE-Bold8
   <9-10> MnSymbolE-Bold9
  <10-12> MnSymbolE-Bold10
  <12->   MnSymbolE-Bold12
}{}

\makeatletter
\let\llangle\@undefined
\let\rrangle\@undefined
\makeatother

\DeclareMathDelimiter{\llangle}{\mathopen}{MnLargeSymbols}{'164}{MnLargeSymbols}{'164}
\DeclareMathDelimiter{\rrangle}{\mathclose}{MnLargeSymbols}{'171}{MnLargeSymbols}{'171}

\newcommand{\no}{\mathbf{No}}
\newcommand{\on}{\mathbf{On}}
\newcommand{\J}{\mathbb{J}}
\newcommand{\li}{\mathbb{L}}
\newcommand{\M}{\mathfrak{M}}
\newcommand{\N}{\mathbb{N}}
\newcommand{\Q}{\mathbb{Q}}
\newcommand{\R}{\mathbb{R}}
\newcommand{\T}{\R^{*}\M}
\newcommand{\Z}{\mathbb{Z}}
\newcommand{\nin}{\notin}
\newcommand{\eps}{\varepsilon}
\DeclareMathOperator{\supp}{Supp}
\DeclareMathOperator{\term}{Term}
\newcommand{\bracket}[1]{\R\llangle#1\rrangle}

\newcommand{\m}{\mathfrak{m}}
\newcommand{\n}{\mathfrak{n}}
\renewcommand{\o}{\mathfrak{o}}

\newcommand{\vell}{\log^{\uparrow}}
\newcommand{\vless}{\prec}
\newcommand{\vleq}{\preceq}
\newcommand{\vgreater}{\succ}
\newcommand{\vgeq}{\succeq}
\newcommand{\veq}{\asymp}

\newcommand{\lequal}{\asymp^{{\scriptscriptstyle L}}}

\newcommand{\suchthat}{\,:\,}

\newcommand{\tree}[3]{\langle#1,#2,#3\rangle}
\DeclareMathOperator{\admtree}{A}
\DeclareMathOperator{\admtreeinf}{A^{\circ}}
\newcommand{\ctree}{\overline{c}}
\DeclareMathOperator{\treeroot}{R}
\DeclareMathOperator{\size}{size}

\DeclareMathOperator{\LM}{LM}
\DeclareMathOperator{\LT}{LT}

\DeclareMathOperator{\erank}{ER_{\Delta}}

\newcommand{\remove}[2]{#1\setminus#2}
\newcommand{\replace}[3]{#1[#2/#3]}
\newcommand{\prune}[2]{#1^{#2}}
\renewcommand{\big}{\uparrow}
\newcommand{\bigeq}{\uparrow=}
\renewcommand{\small}{\downarrow}
\newcommand{\sml}{\mathrm{small}}

\title{Transseries as germs of surreal functions}

\author{Alessandro Berarducci}

\thanks{A.B.\ was partially supported by PRIN 2012
  ``\foreignlanguage{italian}{Logica, Modelli e Insiemi}'' and by
  \foreignlanguage{italian}{Progetto di Ricerca d'Ateneo} 2015
  ``\foreignlanguage{italian}{Connessioni fra dinamica olomorfa,
    teoria ergodica e logica matematica nei sistemi dinamici}''. }

\address{\foreignlanguage{italian}{Università di Pisa, Dipartimento di
    Matematica, Largo Bruno Pontecorvo 5, 56127 Pisa, PI,} Italy}

\email{alessandro.berarducci@unipi.it}

\author{Vincenzo Mantova}

\address{School of Mathematics, University of Leeds, Leeds LS2 9JT,
  United Kingdom}

\thanks{V.M.\ was supported by ERC AdG ``Diophantine Problems'' 267273
  and partially supported by the research group INdAM GNSAGA.}

\email{V.L.Mantova@leeds.ac.uk}

\subjclass[2010]{03C64, 16W60, 04A10, 26A12, 13N15.}

\keywords{surreal numbers, transseries, composition}

\date{March 6th, 2017. Revised on September 20th, 2017.}

\begin{document}

\begin{abstract}
  We show that Écalle's transseries and their variants (LE and
  EL-series) can be interpreted as functions from positive infinite
  surreal numbers to surreal numbers. The same holds for a much larger
  class of formal series, here called omega-series. Omega-series are
  the smallest subfield of the surreal numbers containing the reals,
  the ordinal omega, and closed under the exp and log functions and
  all possible infinite sums.  They form a proper class, can be
  composed and differentiated, and are surreal analytic. The surreal
  numbers themselves can be interpreted as a large field of
  transseries containing the omega-series, but, unlike omega-series,
  they lack a composition operator compatible with the derivation
  introduced by the authors in an earlier paper.
\end{abstract}

\maketitle

\tableofcontents{}

\section{Introduction}

Fields of transseries are an important tool in asymptotic analysis and
played a crucial role in Écalle's approach to the problem of Dulac
\cite{Dulac1923, Ecalle1992}. They appear in various versions, see for
instance \cite{Dahn1987, Dries1997, VanderHoeven1997, Kuhlmann2000,
  DriesMM2001, Schmeling2001, Kuhlmann2005, VanderHoeven2006,
  VanderHoeven2009} and the bibliography therein. In
\cite{Berarducci2015} we proved that Conway's field $\no$ of surreal
numbers \cite{Conway1976} admits the structure of a field of
transseries (in the sense of \cite{Schmeling2001}) and a compatible
derivation (in fact more than one). We also proved the existence of
``integrals'', in the sense of anti-derivatives, for the ``simplest''
surreal derivation on $\no$. This makes $\no$ into a Liouville closed
H-field in the sense of \cite{Aschenbrenner2002}.  We recall that an
H-field is an ordered differential field with some compatibility
properties between the derivation $\partial$ and the order; in
particular if $f$ is greater than any constant, then $\partial f>0$.
A basic example is the field of rational functions $\R(x)$, ordered by
$x>\R$, with constant field $\R=\ker\partial$ and $\partial x=1$.  The
notion of H-field arises as an attempt to axiomatize some of the
properties of Hardy fields, where a Hardy field is a field of germs at
$+\infty$ of eventually $C^{1}$-functions $f:\R\to\R$ closed under
derivation. Such fields have been studied since the 70's, see for
instance \cite{Bourbaki1976, Rosenlicht1983, Rosenlicht1983a,
  Rosenlicht1987}.  Any o-minimal structure on the reals gives rise to
an H-field, namely the field of germs at $+\infty$ of its definable
unary functions.  In \cite{Aschenbrenner2015a} van den Dries,
Aschenbrenner and van der Hoeven proved that, with the ``simplest''
derivation $\partial$ introduced in \cite{Berarducci2015}, the
surreals are a universal H-field; more precisely, every H-field with
``small derivations'' and constant field $\R$ embeds in $\no$ as a
differential field.  Moreover, they proved that $(\no,\partial)$
satisfies the complete first order theory of the
logarithmic-exponential series of \cite{Dries1997, DriesMM2001} and
therefore, by the model completeness of the theory
\cite{Aschenbrenner2015a}, it admits solutions to all the differential
equations that can be solved in a bigger model.

Another approach to derivation and integration on the surreal numbers
was taken by Costin, Ehrlich and Friedman \cite{Costin2015} in a more
analytic vein, possibly suitable for asymptotic analysis, namely they
consider derivatives and definite integrals of functions, rather than
derivatives of ``numbers'' (elements of $\no$).

This paper is a first attempt to reconcile the algebraic and the
analytic approach to surreal derivation and integration through a
notion of composition. The special session on surreal numbers at the
joint AMS-MAA meeting in Seattle (6-9 Jan. 2016) was a timely occasion
to discuss these developments and some of the results of this paper
were presented during that meeting.

To discuss our contribution in more detail, we need some definitions.
We recall that in $\no$, as in any Hahn field, there is a formal
notion of summability, and one can associate to each summable family
$(x_{i})_{i\in I}$ its ``sum'' $\sum_{i\in I}x_{i}\in\no$. We can thus
define the field of \textbf{omega-series} $\bracket{\omega}$ as the
smallest subfield of $\no$ containing $\R(\omega)$ and closed under
$\exp$, $\log$ and sums of summable families. Here $\omega$ is the
first infinite ordinal and plays the role of a formal variable with
derivative $1$. It turns out that $\bracket{\omega}$ is a very big
exponential field (in fact a proper class) properly containing an
isomorphic copy of the logarithmic-exponential series of
\cite{Dries1997, DriesMM2001} (LE-series) and their variants, such as
the exponential-logarithmic series of \cite{Kuhlmann2000,
  Kuhlmann2012d} (EL-series). More precisely, we can isolate two
subfields $\R((\omega))^{LE}\subset\R((\omega))^{EL}$ of
$\bracket{\omega}$ which are isomorphic to the LE and EL-series
respectively. The field $\R((\omega))^{LE}$ is a countable union
$\bigcup_{n\in\N}X_{n}\subseteq\no$, where $X_{0}:=\R(\omega)$ and
$X_{n+1}$ is the set of all sums of summable sequences of elements in
$X_{n}\cup\exp(X_{n})\cup\log(X_{n})$. In other words, a surreal number
is a LE-series if it can be obtained from $\R(\omega)$ by finitely
many applications of $\sum,\exp,\log$ (\prettyref{thm:LE}).  This
remarkably simple characterization of the LE-series, which should be
compared with the original definition, is made possible by working
inside the surreals, with its notion of summability and exponential
structure. The EL-series admit a similar characterization
(\prettyref{prop:characterization-EL}).

We show that each omega-series $f\in\bracket{\omega}$, hence in
particular each LE or EL-series, can be interpreted as a function from
positive infinite surreal numbers to surreal numbers
(\prettyref{cor:substitution-omega-series}).  The idea is simply to
substitute $\omega$ with a positive infinite surreal and evaluate the
resulting expression, but the proof of summability
(\prettyref{lem:summability}) is rather long and technical and it is
carried out in \prettyref{sec:induction}. Similar problems were
tackled in \cite{Schmeling2001} and in some of the cited works by van
der Hoeven, although not in the context of surreal numbers. We shall
borrow from those papers the idea of isolating the contributions
coming from different ``trees'', but with enough differences to
warrant an independent treatment. This will give rise to a natural
composition operator $\circ:\bracket{\omega}\times\no^{>\R}\to\no$
(\prettyref{thm:composition-omega-series}) which restricts to a
composition
$\circ:\bracket{\omega}\times\bracket{\omega}^{>\R}\to\bracket{\omega}$
extending the usual composition of ordinary power series. Formally, we
define a \textbf{composition} on $\bracket{\omega}$ to be a function
$\circ:\bracket{\omega}\times\no^{>\R}\to\no$ satisfying the following
conditions for all $f,g\in\bracket{\omega}$ and $x\in\no^{>\R}$:
\begin{enumerate}
\item if $f=\sum_{i<\alpha}r_{i}e^{\gamma_{i}}$, then
  $f\circ x=\left(\sum_{i<\alpha}r_{i}e^{\gamma_{i}}\right)\circ
  x=\sum_{i<\alpha}r_{i}e^{\gamma_{i}\circ x}$;
\item $f\circ g\in\bracket{\omega}$ and
  $(f\circ g)\circ x=f\circ(g\circ x)$;
\item $f\circ\omega=f$, $\omega\circ x=x$.
\end{enumerate}
We then prove the following.

\begin{namedthm}[\prettyref{thm:composition-omega-series}]
  There is a (unique) composition
  $\circ:\bracket{\omega}\times\no^{>\R}\to\no$.
\end{namedthm}

In the last part of the paper we study the interaction between the
derivation $\partial:\no\to\no$ introduced in \cite{Berarducci2015}
and the composition on $\bracket{\omega}$. Let us recall that in
\cite{Berarducci2015} we proved the existence of several ``surreal
derivations'' $\partial:\no\to\no$ and we studied in detail the
``simplest'' such derivation \cite[Def.\ 6.21]{Berarducci2015}.  It is
easy to see that all surreal derivations coincide on the subfield
$\bracket{\omega}$, so the latter admits a unique surreal derivation
$\partial:\bracket{\omega}\to\bracket{\omega}$. The derivation
$\partial$ on $\bracket{\omega}$ makes it into a H-field, although not
a Liouville closed one because
$\partial:\bracket{\omega}\to\bracket{\omega}$ is not
surjective. There are, however, many subfields of $\bracket{\omega}$
which are Liouville closed, among which $\R((\omega))^{LE}$.

We will show that the formal derivative $\partial f$ of an
omega-series $f\in\bracket{\omega}$ can be interpreted as the
derivative of the function $\hat{f}:\no^{>\R}\to\no$ defined by
$\hat{f}(x)=f\circ x$, namely we have
\[
  \partial f\circ x=\lim_{\eps\to0}\frac{f\circ(x+\eps)-f\circ
    x}{\eps},
\]
where $x$ and $\eps$ range in $\no$ (\prettyref{cor:limit-def}).
Since $\partial f\circ\omega=\partial f$, this shows in particular
that the derivative can be defined in terms of the composition:
$\partial
f=\lim_{\eps\to0}\frac{f\circ(\omega+\eps)-f\circ\omega}{\eps}$.  Other
compatibility conditions then follow, such as the chain rule
$\partial(f\circ g)=(\partial f\circ g)\cdot\partial g$
(\prettyref{cor:chain-rule}).

These results tells us that any omega-series $f\in\bracket{\omega}$,
hence in particular every logarithmic-exponential series, can be
interpreted as a differentiable function $\hat{f}:\no^{>\R}\to\no$ from
positive infinite surreal numbers to surreal numbers. We shall prove
that all such functions are surreal analytic in the following sense.

\begin{namedthm}[\prettyref{thm:analytic}]
  Every $f\in\bracket{\omega}$ is surreal analytic, namely for every
  $x\in\no^{>\R}$ and every sufficiently small $\eps\in\no$ we have
  \[
    f\circ(x+\eps)=\sum_{n\in\N}\frac{1}{n!}(\partial^{n}\!f\circ
    x)\cdot\eps^{n}.
  \]
\end{namedthm}

It is tempting to raise the conjecture that the exponential field
$\no$, enriched with all the functions $\hat{f}:\no^{>\R}\to\no$ for
$f\in\bracket{\omega}$ (possibly restricted to some interval
$(a,+\infty)$) has a good model theory. For instance, the restricted
version could yield an o-minimal structure on $\no$. Indeed, note that
the family of all functions $\hat{f}:\no^{>\R}\to\no$ (for
$f\in\bracket{\omega}$) yields a sort of non-standard Hardy field on
$\no$, namely a field of functions closed under differentiation (it is
also closed under $\exp$, $\log$ and composition).

We do not know up to what extent the above results can be extended
beyond $\bracket{\omega}$, namely whether we can introduce a
composition operator on the whole of $\no$, thus giving a functional
interpretation to all surreal numbers. Concerning this problem, we
have a negative result. Say that a derivation $\partial$ and a
composition $\circ$ are \textbf{compatible} if the function
$x\mapsto f\circ x$ is constant when $\partial f=0$ and strictly
increasing when $\partial f>0$, and if the chain rule
$\partial(f\circ g)=(\partial f\circ g)\cdot\partial g$ holds for all
$f,g\in\no$ (see \prettyref{def:compatible}).

\begin{namedthm}[\prettyref{thm:negative}]
  The simplest derivation $\partial:\no\to\no$ of
  \cite{Berarducci2015} cannot be compatible with a composition on
  $\no$.
\end{namedthm}

We conclude with some questions. The first is to study possible
notions of compositions and compatible derivations on the whole of
$\no$ (see \prettyref{que:compatible}). This is also connected with
the long-standing question of the existence of trans-exponential
o-minimal structures; a good composition on $\no$ may provide a
non-archimedean example. Another related question is to understand
whether $\no$ has non-trivial field automorphisms preserving infinite
sums and the function $\exp$.

\section{Preliminaries}

In this section, we recall a few well known constructions and facts
regarding ordered fields and surreal numbers, and above all, we shall
establish some of the notations that will be used throughout the rest
of the paper. Since surreal numbers form a proper class, we implicitly
work in a set theoretic framework which allows to talk about classes
as first class objects, such as NBG. Therefore, in the following
definitions all objects are allowed to be proper classes, unless
specified otherwise.  Given a class $C$, we shall say that $C$ is
\textbf{small} if it is a set and not a proper class.

\subsection{Hahn fields}

\begin{defn}
  \label{def:Rdomination}Let $K$ be an ordered field, $R\subseteq K$ a
  subfield, and $f,g\in K$. We let:
  \begin{enumerate}
  \item $f\preceq_{R}g$, or $f\in\mathcal{O}_{R}(g)$, if there is
    $c\in R$ such that $|f|\leq c|g|$, and we say that $f$ is
    \textbf{$R$-dominated} by $g$;
  \item $f\prec_{R}g$, or $f\in o_{R}(g)$, if $c|f|<|g|$ for every
    $c\in R$, and we say that $f$ is \textbf{$R$-strictly dominated}
    by $g$;
  \item $f$ is \textbf{$R$-finite} (or $R$\textbf{-bounded}) if
    $f\preceq_{R}1$;
  \item $f$ is\textbf{ $R$}-\textbf{infinitesimal} if $f\prec_{R}1$;
  \item $f\veq_{R}g$ if $f\preceq_{R}g$ and $g\preceq_{R}f$, namely
    $f/g$ is $R$-finite and not $R$-infinitesimal, and we say that $f$
    is \textbf{$R$-comparable} to $g$;
  \item $f\sim_{R}g$ if $f-g\prec_{R}g$, and we say that $f$ is
    \textbf{$R$}-\textbf{asymptotic} to $g$.
  \end{enumerate}
  When $R\subseteq\R$ we suppress the ``$R$''. For instance we write
  $f\preceq g$ if there is $c\in\Q$ such $|f|\leq c|g|$ and we say
  that $f$ is \textbf{dominated} by $g$, or we write
  $f\in\mathcal{O}(1)$ if $f$ is \textbf{finite}, namely $f$ is
  dominated by $1$. We say that $f$ and $g$ are in the \textbf{same
    Archimedean class }if $f\veq g$, namely $f\preceq g$ and
  $g\preceq f$.

  Finally, we say that $\Gamma\subseteq K^{>0}$ is a \textbf{group of
    monomials for $K$} if it is a multiplicative subgroup and for
  every $x\in K$ there is a unique $\m\in\Gamma$ such that
  $x\asymp\m$.  It can be proved that any real closed field admits a
  group of monomials.
\end{defn}

\begin{example}
  The field of Laurent series $\R((x^{\Z}))$ consists of all formal
  series of the form $\sum_{n\geq n_{0}}a_{n}x^{n}$, where
  $a_{n}\in\R$ and $n_{0}\in\Z$, ordered according to the sign of the
  leading coefficient $a_{n_{0}}$. The multiplicative subgroup
  $x^{\Z}:=\{x^{n}\suchthat n\in\Z\}$ is a group of monomials for
  $\R((x^{\Z}))$.
\end{example}

\begin{rem}
  Given two monomials $\m,\n$, we have $\m<\n$ if and only if
  $\m\prec\n$.
\end{rem}

\begin{defn}
  \label{def:terms}Let $(\Gamma,\cdot,<)$ be an ordered abelian group
  written in multiplicative notation. Let $R$ be an ordered field.
  The \textbf{Hahn field} $R((\Gamma))$ consists of all formal sums
  $x=\sum_{\m\in\Gamma}x_{\m}\m$ with \textbf{coefficients}
  $x_{\m}\in R$, whose \textbf{support}
  $\supp(x):=\{\m\in\Gamma\suchthat x_{\m}\neq0\}$ is \textbf{reverse
    well-ordered}, namely every non-empty subset of the support has a
  maximal element. If $x_{\m}\neq0$ we say that $x_{\m}\m$ is a
  \textbf{term} of $x$. We denote by
  $R((\Gamma))_{\mathrm{small}}\subseteq R((\Gamma))$ the subclass of
  all formal sums $x=\sum_{\m\in\Gamma}x_{\m}\m$ whose support is small
  (it coincides with $R((\Gamma))$ when $\Gamma$ is small).

  The addition in $R((\Gamma))$ is defined component-wise and the
  multiplication is given by the usual convolution formula:
  $\left(\sum_{\mathfrak{\m}}x_{\m}\mathfrak{\m}\right)
  \left(\sum_{\n}y_{\n}\mathfrak{\n}\right) =
  \left(\sum_{\o}z_{\o}\o\right)$ where
  $z_{\o}=\sum_{\m\n=\o}x_{\m}y_{\n}\in R$. The fact that the supports are
  reverse well-ordered ensures that the latter sum is finite.

  The \textbf{leading monomial} $\LM(x)$ of $x$ is the maximal
  monomial in $\supp(x)$. The \textbf{leading term} $\LT(x)$ is the
  leading monomial multiplied by its coefficient, and the
  \textbf{leading coefficient }is the coefficient of the leading
  monomial. $R((\Gamma))$ is ordered as follows: $x$ is positive if
  and only if its leading coefficient is positive. We denote by
  $\term(x):=\{x_{\m}\m\suchthat\m\in\supp(x)\}$ the class of the terms
  of $x$.
\end{defn}

\begin{fact}
  Both $R((\Gamma))$ and $R((\Gamma))_{\mathrm{small}}$ are ordered
  fields.
\end{fact}

\begin{rem}
  Note that $\Gamma$ is a multiplicative subgroup of $R((\Gamma))$,
  where we identify $\m\in\Gamma$ with $1\m\in R((\Gamma))$. It
  follows from the definitions that $\Gamma\subseteq R((\Gamma))$
  contains one and only one representative for each equivalence class
  modulo $\veq_{R}$. In particular, taking $R=\R$, we have that
  $\Gamma$ is a group of monomials for $\R((\Gamma))$. The same is
  true for $\R((\Gamma))_{\mathrm{small}}$.
\end{rem}

\subsection{Surreal numbers}

We denote by $\no$ the ordered field of surreal numbers
\cite{Conway1976,Gonshor1986}.  A minimal introduction to $\no$,
containing all the prerequisites for this paper, is contained in
\cite{Berarducci2015}. However, there is no need to assume a prior
knowledge of the surreal numbers (the definition itself will not be
needed), if one is willing to take for granted the following fact.

\begin{fact}
  \label{fact:normal-form}We have:
  \begin{enumerate}
  \item $\no$ is an ordered real closed field equipped with an
    exponential function $\exp:\no\to\no$, $x\mapsto e^{x}:=\exp(x)$,
    making it into an elementary extension of $(\R,<,+,\cdot,\exp)$
    \cite{DriesE2001}; in particular, $\exp:\no\to\no$ is an
    increasing isomorphism from the additive to the positive
    multiplicative group.
  \item $\no$ contains an isomorphic copy of the ordered class $\on$
    of all ordinal numbers (hence $\no$ is a proper class). The
    addition and multiplication restricted to $\on$ coincide with the
    Hessenberg sum and product.
  \item There is a representation of surreal numbers as binary
    sequences of any ordinal length. The relation of being an initial
    segment, called \textbf{simplicity}, is well founded and makes
    $\no$ into a binary tree. This gives us a canonical choice for a
    group $\M\subseteq\no^{>0}$ of monomials: the \textbf{monomials}
    are the simplest positive representatives of the Archimedean
    classes (they form a proper class).
  \item The ordinal $\omega$ belongs to $\M$ (it will later play the
    role of a formal variable with derivative $1$). If
    $1\prec\m\in\M$, then $e^{\m}\in\M$. In particular $e^{\omega}$ and
    $e^{-\omega}$ are monomials, but $e^{1/\omega}$ is not.
  \item There is a canonical isomorphism (written as an
    identification)
    \[
      \no=\R((\M))_{\mathrm{small}}\subset\R((\M)).
    \]
  \item A surreal number $\sum_{m\in\M}x_{\m}\m$ is \textbf{purely
      infinite} if all the monomials $\m$ in its support are infinite,
    namely $\m\succ1$.  Letting $\J\subseteq\no$ be the class of all
    purely infinite surreal numbers, there is a direct sum
    decomposition of $\R$-vector spaces
    \[
      \no=\J\oplus\R\oplus o(1).
    \]
  \item We have
    $\M=\exp(\J)=\left\{ e^{\gamma}\suchthat\gamma\in\J\right\} $, so
    we can write
    \[
      \no=\R((e^{\J}))_{\mathrm{small}}.
    \]
    In other words, every surreal number $x\in\no$ can be uniquely
    written in the form
    \[
      x=\sum_{i<\alpha}r_{i}e^{\gamma_{i}}
    \]
    where $\alpha\in\on$, $r_{i}\in\R^{*}$, and
    $(\gamma_{i})_{i<\alpha}$ is a decreasing sequence in $\J$ indexed
    by an ordinal $\alpha\in\on$.  We call this the \textbf{Ressayre
      normal form} of $x$.
  \item The exponential function on $o(1)$ can be calculated using the
    Taylor series of $\exp$, namely
    \[
      \exp(\varepsilon)=\sum_{n=0}^{\infty}\frac{\varepsilon^{n}}{n!}
    \]
    for all $\varepsilon\in o(1)$ (see \prettyref{subsec:Summability}
    for the meaning of the above infinite sum). Likewise, the inverse
    $\log$ satisfies
    \[
      \log(1+\varepsilon)=\sum_{n=1}^{\infty}(-1)^{n+1}\frac{\varepsilon^{n}}{n}.
    \]
  \end{enumerate}
\end{fact}

\begin{rem}
  For infinite $x$, the equality
  $\exp(x)=\sum_{n=0}^{\infty}\frac{x^{n}}{n!}$ does \emph{not} hold. In
  fact, the right-hand side does not even represent a surreal number
  (see \prettyref{subsec:Summability}). Likewise for $\log(1+x)$.
\end{rem}

\begin{defn}
  By the decomposition $\no=\J\oplus\R\oplus o(1)$, for every surreal
  number $x\in\no$ we can write uniquely
  \[
    x=x^{\big}+x^{=}+x^{\small}
  \]
  where $x^{\big}\in\J$, $x^{=}\in\R$ and $x^{\small}\prec1$. We also
  write $x^{\bigeq}$ for $x^{\big}+x^{=}$.
\end{defn}

\begin{defn}
  \label{def:ell}Thanks to \prettyref{fact:normal-form}(5) we can
  apply to $\no$ the definitions already introduced for Hahn fields
  (support, leading term, etc.). In particular, if
  $x=\sum_{i<\alpha}r_{i}e^{\gamma_{i}}$ is in normal form, its leading
  monomial is $e^{\gamma_{0}}$ and its leading term is
  $r_{0}e^{\gamma_{0}}$; in this case we define
  \[
    \vell(x):=\gamma_{0}.
  \]
\end{defn}

Note that $\vell(x)=\log(x)^{\big}$, as in fact
$\log(x) = \log(r_{0}e^{\gamma_{0}}(1+\varepsilon)) = \gamma_{0} +
\log(r_{0}) + \sum_{n=1}^{\infty}(-1)^{n+1}\frac{\varepsilon^{n}}{n}$ where
$\varepsilon\prec1$. Moreover, $x\vless y$ if and only if
$\vell(x)<\vell(y)$ (so $-\vell$ is a Krull valuation).

\begin{defn}
  \label{def:truncation}If $x=\sum_{i<\alpha}r_{i}e^{\gamma_{i}}$ and
  $\beta\leq\alpha$, the number $\sum_{i<\beta}r_{i}e^{\gamma_{i}}$ is
  called a \textbf{truncation} of $x$. A subclass $A\subseteq\no$ is
  \textbf{truncation closed} if for every $x$ in $A$, all truncations
  of $x$ are also in $A$.
\end{defn}

Note that $x^{\big}$ is a truncation of $x$ and it coincides with the
sum of all the terms $r_{i}e^{\gamma_{i}}$ of $x$ with $\gamma_{i}>0$ (if
there are no such terms, then $x^{\big}=0$).

\begin{notation}
  \label{nota:set-notations}Given $A,B\subseteq\no$ we shall use some
  self-explanatory notations like the following:
  \begin{itemize}
  \item $A^{>0}$ is the set of positive elements of $A$;
  \item $A^{\succ1}$ is the set of elements $a\in A$ satisfying
    $a\succ1$;
  \item $A<B$ means $a<b$ for all $a\in A$ and $b\in B$;
  \item $\exp(A):=\{\exp(x)\suchthat x\in A\}$ and
    $\log(A):=\{\log(x)\suchthat x\in A^{>0}\}$, where
    $\log:\no^{>0}\to\no$ is the inverse of $\exp$.
  \end{itemize}
\end{notation}

\begin{example}
  Since $\M=\exp(\J)$, we have $\M^{\succ1}=\exp(\J^{>0})$ and
  $\M^{\prec1}=\exp(\J^{<0})$.
\end{example}

\subsection{\label{subsec:Summability}Summability}

Any Hahn field, and in particular $\no$ by
\prettyref{fact:normal-form}(5), admits a natural notion of infinite
sum, as follows.

\begin{defn}
  \label{def:summability}Let $I$ be a set (not a proper class) and
  $(x_{i}\suchthat i\in I)$ be an indexed family of elements of $\no$.

  We say that $(x_{i}\suchthat i\in I)$ is \textbf{summable} if
  $\bigcup_{i\in I}\supp(x_{i})$ is reverse well-ordered and for each
  $\m\in\bigcup_{i\in I}\supp(x_{i})$, there are only finitely many
  $i\in I$ such that $\m\in\supp(x_{i})$.  In this case, the
  \textbf{sum} $\sum_{i\in I}x_{i}$ is the unique surreal number
  $y=\sum_{\m}y_{\m}\m$ such that
  $\supp(y)\subseteq\bigcup_{i\in I}\supp(x_{i})$ and, for every
  $\m\in\M$, $y_{\m}=\sum_{i\in I}(x_{i})_{\m}$ (note that there are
  finitely many $i\in I$ with $x_{i}\neq0$ by the hypothesis of
  summability). Similar definitions apply replacing $\no$ with any
  field of the form $R((\Gamma))_{\mathrm{small}}$.

  We shall also say that $\sum_{i\in I}x_{i}$ \textbf{exists }to mean
  that $(x_{i})_{i\in I}$ is summable.
\end{defn}

\begin{rem}
  \label{rem:summability-criterion}A family $(x_{i}\suchthat i\in I)$
  is summable if and only if there are no injective sequences
  $(i_{n})_{n\in\N}$ in $I$ and monomials $\m_{n}\in\supp(x_{i_{n}})$ (not
  necessarily distinct) such that $\m_{n}\preceq\m_{n+1}$ for each
  $n\in\N$ (where $\N$ is the set of non-negative
  integers). Equivalently, for every injective sequence
  $(i_{n})_{n\in\N}$ in $I$ and for any choice of monomials
  $\m_{n}\in\supp(x_{i_{n}})$, there is a subsequence
  $(i_{f(n)})_{n\in\N}$ such that $\m_{i_{f(n)}}\succ\m_{i_{f(n+1)}}$ for
  every $n\in\N$.
\end{rem}

\subsection{Hahn fields embedded in $\protect\no$}

Given a subfield $R$ of $\no$ and a multiplicative subgroup $\Gamma$
of the monomials $\M=e^{\J}$, we will sometimes be interested in the
class of all surreal numbers that can be written as a sum
$\sum r_{\m}\m$ for $r_{\m}\in R$ and $\m\in\Gamma$. Under suitable
assumptions on $R$ and $\Gamma$, this subclass of $\no$ can be
identified with the Hahn field $R((\Gamma))$.

\begin{prop}
  \label{prop:Hahn-subfield-No}Let $\Gamma$ be a small multiplicative
  subgroup of $\M=e^{\J}$ and $R$ be a truncation closed subfield of
  $\no$. If $R<\Gamma^{>1}$, there is a unique field embedding
  $R((\Gamma))\to\no$ sending $r\m$ (as an element of $R((\Gamma))$)
  to $r\m$ (as an element of $\no$) and preserving infinite sums.
\end{prop}

\begin{proof}
  Suppose that $R<\Gamma^{>1}$. It suffices to check that the embedding
  exists. Without loss of generality, we may assume that
  $\R\subseteq R$, as the compositum $\R\cdot R$ is clearly truncation
  closed and it also satisfies $\R\cdot R<\Gamma^{>1}$.

  Let $\sum r_{\m}\m$ be an element of $R((\Gamma))$. We wish to prove
  that $(r_{\m}\m\in\no\,:\,r_{\m}\neq0)$ is summable. Take an injective
  sequence $(r_{\m_{n}}\m_{n})_{n\in\N}$ and a choice of
  $\n_{n}\in\supp(r_{\m_{n}}\m_{n})$.  We can write
  $\n_{n}=\m_{n}\o_{n}$, where $\o_{n}\in\supp(r_{\m_{n}})$.  Note that
  $\o_{n}\in R$, since $R$ contains $\R$ and is closed under
  truncation.

  After extracting a subsequence, we may assume that
  $(\m_{n})_{n\in\N}$ is strictly decreasing. We can now easily check
  that $(\n_{n})_{n\in\N}$ is also strictly decreasing: indeed,
  \[
    \frac{\n_{n}}{\n_{n+1}}=\frac{\m_{n}}{\m_{n+1}}\cdot\frac{\o_{n}}{\o_{n+1}}>1,
  \]
  as $\frac{\o_{n+1}}{\o_{n}}\in R<\Gamma^{>1}$.
\end{proof}

\begin{notation}
  \label{nota:Hahn-subfield-No-1}By \prettyref{prop:Hahn-subfield-No},
  given a small multiplicative group $\Gamma$ of $\M=e^{\J}$ (the class
  of monomials of $\no$) and a truncation closed subfield
  $R\subseteq\no$ such that $R<\Gamma^{>1}$, we can identify the field
  $R((\Gamma))$ with the class of surreal numbers that are of the form
  $\sum r_{\m}\m$ with $r_{\m}\in R$ and $\m\in\Gamma$.
\end{notation}

\begin{lem}
  \label{lem:Hahn-lexicographic}Let $\Gamma_{1}$ and $\Gamma_{2}$ be
  subgroups of a given ordered abelian multiplicative group. Suppose
  $\Gamma_{1}<\Gamma_{2}^{>1}$. Then $\Gamma_{1}\Gamma_{2}$ is naturally
  isomorphic, as an ordered group, to the direct product
  $\Gamma_{1}\times\Gamma_{2}$ with the reverse lexicographic order.
\end{lem}

\begin{proof}
  Clearly, $\Gamma_{1}\cap\Gamma_{2}=\{1\}$, so the map sending
  $ab\in\Gamma_{1}\Gamma_{2}$ to $(a,b)\in\Gamma_{1}\times\Gamma_{2}$ is a
  well-defined isomorphism of abelian groups. We can easily verify
  that it preserves the ordering.  Indeed, let $a,a'\in\Gamma_{1}$ and
  $b,b'\in\Gamma_{2}$ be such that $b<b'$. It suffices to show that
  $ab<a'b'$. This can be rewritten as $a/a'<b'/b$. Since $b'/b>1$, the
  desired result follows by the hypothesis $\Gamma_{1}<\Gamma_{2}^{>1}$.
\end{proof}

Using the above notation, \prettyref{prop:Hahn-subfield-No}, and
\prettyref{lem:Hahn-lexicographic}, we can then deduce the following
well-known result (see for instance \cite[1.4]{DriesMM2001}). However,
note that the result contains an equality rather than just an
isomorphism, thanks to the identifications of
\prettyref{nota:Hahn-subfield-No-1}.

\begin{cor}
  \label{cor:Hahn-lexicographic-1}Let $\Gamma_{1}$, $\Gamma_{2}$ be
  small subgroups of $\M$. If $\Gamma_{1}<\Gamma_{2}^{>1}$, then we have
  $\R((\Gamma_{1}))((\Gamma_{2})) = \R((\Gamma_{1}\Gamma_{2})) \cong
  \R((\Gamma_{1}\times\Gamma_{2}))$.
\end{cor}

\begin{proof}
  We first note that $\R((\Gamma_{1}))<\Gamma_{2}^{>1}$, from which it
  follows at once that
  $\R((\Gamma_{1}))((\Gamma_{2}))\subseteq\R((\Gamma_{1}\Gamma_{2}))$ by
  \prettyref{prop:Hahn-subfield-No}. On the other hand, let
  $x=\sum_{\m\in\Gamma_{1}\Gamma_{2}}r_{\m}\m$ be an element of
  $\R((\Gamma_{1}\Gamma_{2}))$. Since
  $\Gamma_{1}\Gamma_{2}\cong\Gamma_{1}\times\Gamma_{2}$, each
  $\m\in\Gamma_{1}\Gamma_{2}$ decomposes uniquely as a product
  $\m=\n\o$ with $\n\in\Gamma_{1}$ and $\o\in\Gamma_{2}$. But then it is
  easy to verify that
  \[
    x = \sum_{\m}r_{\m}\m = \sum_{\o\in\Gamma_{2}}
    \left(\sum_{\n\in\Gamma_{1}}r_{\n\o}\n\right)\o \in
    \R((\Gamma_{1}))((\Gamma_{2})).
  \]
\end{proof}

\begin{rem}
  If one drops the assumption that $\Gamma$ is small, then the
  conclusion of \ref{prop:Hahn-subfield-No} holds with
  $R((\Gamma))_{\mathrm{small}}$ in place of $R((\Gamma))$. In
  particular, we may canonically identify
  $R((\Gamma))_{\mathrm{small}}$ with a subfield of $\no$, as in
  \prettyref{nota:Hahn-subfield-No-1}.  As a special case, one
  recovers the already mentioned identification
  $\no=\R((\M))_{\mathrm{small}}$ of \prettyref{fact:normal-form}(5).
  The conclusion of \prettyref{cor:Hahn-lexicographic-1} also holds,
  provided one uses $R((\Gamma_{i}))_{\mathrm{small}}$ instead of
  $R((\Gamma_{i}))$ for $i=1,2$.
\end{rem}

\section{Surreal analytic functions}

A real function is analytic at a point in its domain if there is a
neighborhood of the point in which it coincides with the limit of a
power series. Such notion does not generalize directly to surreal
numbers, as $\no$ does not have a good notion of limit for series.
However, we can replace the limit with the natural notion of infinite
sum from \prettyref{def:summability}. This leads to a theory of
``surreal analytic function'' developed in \cite{Alling1987}. In this
section we isolate and extend some of those results in a form suitable
for our goals.

Infinite sum bears some resemblance with the usual notion of absolute
convergence. On the one hand, like absolute convergence, it enjoys
some good algebraic properties, such as being independent on the
``order'' in which we sum the elements of the family. On the other
hand, it is not related to the order topology; for instance, even if a
family $(x_{i})_{i\in I}$ is summable, and $(y_{i})_{i\in I}$ is such that
$|y_{i}|\leq|x_{i}|$, it does not necessarily follow that
$(y_{i})_{i\in I}$ is summable.

\begin{lem}
  Let $(a_{i}\,:\,i\in I)$ be a summable family of surreal numbers.
  Then for any partition $I=\bigsqcup_{j\in J}I_{j}$ of the set $I$,
  each sum $\sum_{i\in I_{j}}a_{i}$ exists, the family
  $(\sum_{i\in I_{j}}a_{i}\,:\,j\in J)$ is summable, and
  \[
    \sum_{j\in J}\sum_{i\in I_{j}}a_{i}=\sum_{i\in I}a_{i}.
  \]
\end{lem}

\begin{proof}
  Clearly, since $(a_{i}\,:\,i\in I)$ is summable, so is each
  $(a_{i}\,:\,i\in I_{j})$ for $j\in J$. Moreover, it also follows
  easily that $(\sum_{i\in I_{j}}a_{i}\,:\,j\in J)$ is summable, as each
  monomial $\m$ in $\supp(\sum_{i\in I_{j}}a_{i})$ must appear in
  $\supp(a_{i})$ for some $i\in I_{j}$. To check that its sum is indeed
  equal to $\sum_{i\in I}a_{i}$, for a given monomial $\m$, let
  $a_{i,\m}$be the coefficient of $\m$ in $a_{i}$. Then the coefficient
  of $\m$ in $\sum_{j\in J}\sum_{i\in I_{j}}a_{i}$ is
  $\sum_{j\in J}\sum_{i\in I_{j}}a_{i,\m}=\sum_{i\in I}a_{i,\m}$, which in
  turn is the coefficient of $\m$ in $\sum_{i\in I}a_{i}$, proving the
  conclusion.
\end{proof}

\begin{cor}
  \label{cor:exchange-of-sums}Let
  $(a_{i,j}\suchthat(i,j)\in I\times J)$ be a summable family of
  surreal numbers. Then both $\sum_{i\in I}\sum_{j\in J}a_{i,j}$ and
  $\sum_{j\in J}\sum_{i\in I}a_{i,j}$ exist and
  \[
    \sum_{i\in I}\sum_{j\in J}a_{i,j}=\sum_{j\in J}\sum_{i\in
      I}a_{i,j}=\sum_{(i,j)\in I\times J}a_{i,j}.
  \]
\end{cor}

\begin{rem}
  The assumption of summability of
  $(a_{i,j}\suchthat(i,j)\in I\times J)$ is necessary, or the equality
  may not hold. For instance, take $a_{i,i}=\omega$,
  $a_{i,i+1}=-\omega$, and $a_{i,j}=0$ otherwise for $i,j\in\N$, which
  is clearly not summable. Then $\sum_{i\in\N}\sum_{j\in\N}a_{i,j}=0$
  while $\sum_{j\in\N}\sum_{i\in\N}a_{i,j}=\omega$. Moreover, one of the
  two sums may not even exists; for instance,
  $\sum_{i\in\N}\sum_{j=0}^{1}(-1)^{j}\omega$ clearly exists and is equal
  to $0$, while $\sum_{j=0}^{1}\sum_{i\in\N}(-1)^{j}\omega$ does not
  exist. It can also happen that the two sums $\sum_{i}\sum_{j}$ and
  $\sum_{j}\sum_{i}$ exists and are equal, but the sum
  $\sum_{(i,j)\in I\times J}$ does not exists: take $a_{i,i}=2\omega$
  and $a_{i+1,i}=a_{i,i+1}=-\omega$, with all other terms $a_{i,j}$ being
  zero.
\end{rem}

\subsection{Products and powers of summable families}

The following is well known.

\begin{rem}
  \label{rem:product}If $(x_{i})_{i\in I}$ and $(y_{j})_{i\in J}$ are
  summable, then so is $(x_{i}y_{j}\suchthat(i,j)\in I\times J)$.  Its
  sum $\sum_{(i,j)\in I\times J}x_{i}y_{j}$ coincides with the product
  $(\sum_{i\in I}x_{i})(\sum_{j\in J}y_{j})$.
\end{rem}

Using \prettyref{rem:product}, one can easily express the $n$-th power
of a sum as follows.

\begin{prop}
  \label{prop:n-th power}Let $(x_{i})_{i\in I}$ be a summable family of
  surreal numbers and let $n\in\N$. Then the family
  $\left(\prod_{m<n}x_{\tau(m)}\suchthat\tau:n\to I\right)$ is summable
  and
  \[
    \left(\sum_{i\in I}x_{i}\right)^{n}=\sum_{\tau:n\to
      I}\text{ }\prod_{m<n}x_{\tau(m)}.
  \]
\end{prop}

\begin{proof}
  By induction on $n\in\N$ based on \prettyref{rem:product}.
\end{proof}

\begin{cor}
  \label{cor:power-of-series}If $(a_{i}\eps^{i})_{i\in\N}$ is summable,
  then for every $n\in\N$,
  \[
    \left(\sum_{i\in\N}a_{i}\eps^{i}\right)^{n} =
    \sum_{k\in\N}\left(\sum_{i_{1}+\ldots+i_{n}=k}a_{i_{1}}a_{i_{2}} \ldots
      a_{i_{n}}\right)\eps^{k}.
  \]
\end{cor}

\begin{proof}
  By \prettyref{prop:n-th power},
  $\left(\sum_{i\in\N}a_{i}\eps^{i}\right)^{n} =
  \sum_{\tau:n\to\N}\prod_{m<n}a_{\tau(m)}\eps^{\tau(m)}$, and the result
  follows by setting $\tau(m)=i_{m}$ and isolating the coefficient of
  $\eps^{k}$ in the second member.
\end{proof}

\subsection{Sums of power series}

We shall now define how to evaluate a surreal power series on a
surreal number, and the corresponding notion of surreal analytic
function.  This is similar to how real analytic functions are extended
to $\no$, with the difference that we now allow power series to have
surreal coefficients.

\begin{defn}
  \label{def:surreal-analytic}Given a surreal power series
  $P(X)=\sum_{i=0}^{\infty}a_{i}X^{i}\in\no[[X]]$, we define
  \[
    P(\varepsilon):=\sum_{i\in\N}a_{i}\varepsilon^{i}
  \]
  for any $\varepsilon\in\no$ such that the sum on the right hand side
  exists.

  Given a function $f:U\to\no$ from an open subset $U$ of $\no$, we
  say that $f$ is \textbf{surreal analytic at $x$} if there are a
  neighborhood $V\subseteq U$ of $x$ and a power series
  $P(X)\in\no[[X]]$ such that $f(y)=P(y-x)$ for all $y\in V$.
\end{defn}

Unlike the case of real analytic functions, in which some power series
are not convergent and thus do not yield analytic functions, we shall
now verify that \emph{every} power series with surreal coefficients
induces a surreal analytic function.

By Neumann's lemma \cite{Neumann1949}, if $(a_{i})_{i\in\N}$ is a
sequence of real coefficients and $\varepsilon\prec1$, then
$(a_{i}\eps^{i})_{i\in\N}$ is summable. Therefore, for every power series
$P(X)\in\R[[X]]$, $P(\varepsilon)$ is well defined for any
$\varepsilon\vless1$. We can easily extend this result to series with
surreal coefficients.  We start with the following variant of
Neumann's lemma. Its proof is an adaptation of a similar argument in
\cite[p.\,52]{Gonshor1986}.

\begin{lem}
  \label{lem:Neumann}Let $R$ be a subfield of $\no$ and
  $\varepsilon\vless_{R}1$.  Let $(\n_{i})_{i\in\N}$,
  $(\m_{i,j})_{i\in\N,j\leq k_{i}}$ be sequences of monomials in
  respectively $R$ and $\supp(\varepsilon)$, where $(k_{i})_{i\in\N}$ is
  a sequence of natural numbers with $\lim_{i\to\infty}k_{i}=\infty$.
  Then the sum $\sum_{i\in\N}\n_{i}\m_{i,0}\ldots\m_{i,k_{i}}$ exists.
\end{lem}

\begin{proof}
  Suppose by contradiction that there are two family as in the
  hypothesis such that $\sum_{i\in\N}\n_{i}\m_{i,0}\ldots\m_{i,k_{i}}$ does
  not exist. By taking a subsequence, we may assume that
  $(\n_{i}\m_{i,0}\ldots\m_{i,k_{i}})_{i\in\N}$ is weakly increasing. We
  may picture $\m_{i,j}$ as the $(i,j)$-entry of an infinite table,
  where $i$ is the row index and $j$ is the column index. Rearranging
  the terms, we can assume that each row is weakly increasing, namely
  $\m_{i,0}\leq\m_{i,1}\leq\ldots\leq\m_{i,k_{i}}$ for all $i\in\N$.

  Taking a subsequence we may further assume that $(k_{i})_{i\in\N}$ is
  strictly increasing, so in particular $k_{i}\geq i$. Choosing a
  further subsequence we can assume that the first column
  $(\m_{i,0})_{i\geq0}$ is weakly decreasing, since all these monomials
  are in the support of $\varepsilon$. Similarly we can assume that
  $(\m_{i,1})_{i\geq1}$ is weakly decreasing. Continuing in this
  fashion, by a diagonalization argument we can assume that, for any
  fixed $k$, the $k$-th column $(\m_{i,k})_{i\in\N}$ becomes weakly
  decreasing after its $k$-th entry, namely
  $\m_{k,k}\geq\m_{k+1,k}\geq\m_{k+2,k}\geq\ldots$. Note that these terms
  exist since $k_{i}\geq k$ for all $i\geq k$.

  Now fix $i\in\N$ and let $j>i$ (so $k_{j}>k_{i}$). By construction,
  $\n_{i}\m_{i,0} \ldots \m_{i,k_{i}} \leq \n_{j}\m_{j,0} \ldots
  \m_{j,k_{i}}\m_{j,k_{i}+1} \ldots \m_{j,k_{j}}$.  Since
  $\m_{j,k_{i}+1}\ldots\m_{j,k_{j}}\prec_{R}1$, we must have
  $\n_{i}>\n_{j}\m_{j,k_{i}+1}\ldots\m_{j,k_{j}}$. It follows that
  $\m_{i,0}\ldots\m_{i,k_{i}}<\m_{j,0}\ldots\m_{j,k_{i}}$, so in particular
  there is some $k\leq k_{i}$ with $\m_{i,k}<\m_{j,k}$. Now recall that
  the $k$-th column is weakly decreasing after its $k$-th entry, hence
  necessarily $i<k$. We have thus proved that for each $i\in\N$ and
  $j>i$ there is some $k$ with $i<k\leq k_{i}$ such that
  $\m_{i,k}<\m_{j,k}$.

  Taking $j=k_{i}$, and recalling that all the rows are weakly
  increasing, we obtain
  $\m_{i,i}\leq\m_{i,k}<\m_{k_{i},k}\leq\m_{k_{i},k_{i}}$ for all
  $i\in\N$. Iterating we obtain an infinite increasing chain of
  elements of the form $\m_{l,l}$, contradicting the fact that
  $\{\m_{i,j}\,:\,i\in\N,j\leq k_{i}\}$ is in $\supp(\varepsilon)$.
\end{proof}

\begin{cor}
  \label{cor:power-series}Let $R$ be a truncation closed subfield of
  $\no$ and $\varepsilon\prec_{R}1$. Let $(a_{i})_{i\in\N}$ be a sequence
  of coefficients in $R$. Then $(a_{i}\varepsilon^{i})_{i\in\N}$ is
  summable.
\end{cor}

\begin{proof}
  Without loss of generality, we may assume that $\R\subseteq
  R$. Indeed, we may replace $R$ with the compositum $\R\cdot R$,
  which is also closed under truncation, as $\varepsilon\vless_{R}1$
  trivially implies $\varepsilon\vless_{\R\cdot R}1$. In particular, we
  may assume that $\supp(a_{i})\subseteq R$ for all $a_{i}\in R$. Note
  that for all $i\in\N$, any monomial in the support of
  $a_{i}\eps^{i}$ has the form $\n_{i}\m_{i,0}\ldots\m_{i,i-1}$ where
  $\n_{i}\in\supp(a_{i})\subseteq R$ and $\m_{i,j}\in\supp(\eps)$ for
  $j\leq i-1$. The conclusion then follows easily from
  \prettyref{lem:Neumann}.
\end{proof}

\begin{cor}
  \label{cor:power-series-are-surreal-analytic}For every power series
  $P(X)\in\no[[X]]$, the partial function
  $\varepsilon\mapsto P(\varepsilon)$ is surreal analytic at $0$.
\end{cor}

\begin{proof}
  Given a power series $P(X)=\sum_{i=0}^{\infty}a_{i}X^{i}$, it suffices
  to apply \prettyref{cor:power-series} with the ring $R$ generated by
  the monomials in the supports $\supp(a_{i})$. The function
  $\varepsilon\mapsto P(\varepsilon)$ is then defined at least on
  $o_{R}(1)$, which is a nonempty convex subclass containing $0$ as
  $R$ is necessarily small.
\end{proof}

\begin{prop}
  \label{prop:surreal-analytic-is-differentiable}Suppose that $f$ is a
  surreal analytic function at some $x\in\no$. Then $f$ is infinitely
  differentiable at $x$ and
  \[
    f(x+\varepsilon)=\sum_{i=0}^{\infty}\frac{f^{(i)}(x)}{i!}\varepsilon^{i}.
  \]
\end{prop}

\begin{proof}
  Let $f$ be surreal analytic at $x$, with power series
  $P(X)=\sum_{i=0}^{\infty}a_{i}X^{i}$.  Then for every sufficiently small
  $\delta$ we have
  \begin{align*}
    f'(x+\delta)= & \lim_{\varepsilon\to0}\frac{f(x+\delta+\varepsilon)-f(x+\delta)}{\varepsilon} = \lim_{\varepsilon\to0}\sum_{i=0}^{\infty}a_{i}\frac{(\delta+\varepsilon)^{i}-\delta^{i}}{\varepsilon}\\
    = & \lim_{\varepsilon\to0}\sum_{i=0}^{\infty}a_{i}\cdot\frac{\delta^{i}+i\delta^{i-1}\varepsilon+\binom{i}{2}\delta^{i-2}\varepsilon^{2}+\dots+i\delta\varepsilon^{i-1}+\varepsilon^{i}-\delta^{i}}{\varepsilon}\\
    = & \sum_{i=1}^{\infty}ia_{i}\delta^{i-1}+\lim_{\varepsilon\to0}\varepsilon\cdot\sum_{i=2}^{\infty}\left(\binom{i}{2}\delta^{i-2}+\dots+\varepsilon^{i}\right)=\sum_{i=1}^{\infty}ia_{i}\delta^{i-1}.
  \end{align*}
  Therefore, $f$ is differentiable at $x$ and its derivative $f'$ is
  surreal analytic at $x$. Moreover, the above equation also shows
  that $f'(x)=a_{1}$. By induction, it follows that $f$ is infinitely
  differentiable, and that $a_{i}=\frac{f^{(i)}(x)}{i!}$, as desired.
\end{proof}

Moreover, we also observe that Neumann's lemma, already in its
original formulation, implies the following statement for power series
with real coefficients, which will prove useful later on.

\begin{cor}
  \label{cor:power-series-summable}Let $(\varepsilon_{i})_{i\in I}$ be a
  summable family such that $\varepsilon_{i}\prec1$ for all $i\in\N$.
  Let $P_{i}(X)=\sum_{n=1}^{\infty}a_{i,n}X^{n}\in\R[[X]]$ be real power
  series for $i\in I$. Then the family
  $(P_{i}(\eps_{i})\suchthat i\in I)$ is summable.
\end{cor}

\begin{proof}
  Suppose by contradiction that there is a weakly increasing sequence
  of monomials $(\m_{n})_{n\in\N}$ such that
  $\m_{n}\in\supp(P_{i_{n}}(\varepsilon_{i_{n}}))$.  Then for all
  $n\in\N$ there is a positive integer $k_{n}$ such that
  $\m_{n}\in\supp(a_{i_{n},k_{n}}\eps_{i_{n}}^{k_{n}})$. After extracting a
  subsequence, we may either assume that
  $\lim_{n\to\infty}k_{n}=\infty$, and we reach a contradiction by
  \prettyref{lem:Neumann}, or we may assume that the sequence
  $(k_{n})_{n\in\N}$ is constant, so that
  $\m_{n}\in\supp(\eps_{i_{n}}^{k})$ for some fixed $k\in\N$ and all
  $n\in\N$.

  In the latter case, write $\m_{n}=\n_{n,1}\cdot\dots\cdot\n_{n,k}$ with
  $\n_{n,j}\in\supp(\varepsilon_{i_{n}})$. Since
  $(\varepsilon_{i})_{i\in I}$ is summable, we may extract a subsequence
  and assume that $(\n_{n,j})_{n\in\N}$ is strictly decreasing for each
  $j=1,\dots,k$. But then $(\m_{n})_{n\in\N}$ is strictly decreasing, a
  contradiction.
\end{proof}

\begin{rem}
  Since $\no$ is totally disconnected, the present notion of surreal
  analyticity does not have a good theory of analytic continuation.
  For instance, one can define a surreal analytic function on all
  finite numbers by choosing a power series $P_{r}(X)\in\R[[X]]$ for
  each $r\in\R$ and defining $f(r+\varepsilon)=P_{r}(\varepsilon)$ for
  each $r\in\R$ and $\varepsilon\vless1$. Moreover, one can choose the
  series $P_{r}$ such that the restriction of $f$ to $\R$ is itself a
  real analytic function, but with yet other Taylor expansions.  It
  would be interesting to develop an analogous of rigid analytic
  geometry for surreal numbers that prevents such pathological
  behavior.
\end{rem}

\subsection{Composition of power series}

By \prettyref{cor:power-series-are-surreal-analytic}, there is a
morphism from $\no[[X]]$ to germs at zero of surreal functions defined
by evaluating a formal power series
$P(x)=\sum_{i\in\N}a_{i}X^{i}\in\no[[X]]$ at $X=\eps$ for any
sufficiently small $\eps\in\no$. As for traditional power series, we
can show that this morphism behaves well with respect to composition
of power series.

\begin{defn}
  \label{def:composition-formal-power-series}Let $R$ be a subfield of
  $\no$. Given two formal power series
  $P(X):=\sum_{n=0}^{\infty}a_{n}X^{n}$ and
  $Q(X):=\sum_{m=1}^{\infty}b_{m}X^{m}$ in $R[[X]]$, where $Q(X)$ has no
  constant term, their composition $(P\circ Q)(X)$ is defined as the
  power series $\sum_{k\in\N}c_{k}X^{k}\in R[[X]]$ where $c_{0}=a_{0}$ and,
  for $k>0$,
  \[
  c_{k}=\sum_{n=1}^{k}a_{n}\sum\limits _{m_{1}+\ldots+m_{n}=k}b_{m_{1}}\cdots
  b_{m_{n}}.
\]
\end{defn}

\begin{lem}
  \label{lem:double-index-sum}Let $R$ be a truncation closed subfield
  of $\no$ and $\varepsilon\prec_{R}1$. Let
  $(a_{i,j}\suchthat(i,j)\in I\times J)$ be a family of surreal numbers
  in $R$ such that, for any fixed $j\in J$, $\sum_{i\in I}a_{i,j}$
  exists. Then $\sum_{(i,j)\in I\times J}a_{i,j}\eps^{j}$ exists.
\end{lem}

\begin{proof}
  As in the proof of \prettyref{cor:power-series}, we may assume that
  $\supp(a_{i,j})\subseteq R$ for all $(i,j)\in I\times J$. For a
  contradiction, suppose that there is an injective sequence of pairs
  $(i_{n},j_{n})_{n\in\N}$ and a weakly increasing sequence of monomials
  $\m_{n}\in\supp(a_{i_{n},j_{n}}\eps^{j_{n}})$. After extracting a
  subsequence, we may assume that either $\lim_{n\in\N}j_{n}=+\infty$,
  in which case we reach a contradiction by
  \prettyref{cor:power-series}, or the sequence $(j_{n})_{n\in\N}$ is
  constant, so that there is some $j\in J$ such that
  $\m_{n}\in a_{i_{n},j}\eps^{j}$ for every $n\in\N$. In this case, it
  follows that $(a_{i,j}\eps^{j})_{i\in I}$ is not summable, which is
  absurd since $\sum_{i\in I}a_{i,j}$ exists, hence so does
  $\eps^{j}(\sum_{i\in I}a_{i,j})=$ $\sum_{i\in I}a_{i,j}\eps^{j}$.
\end{proof}

\begin{prop}
  \label{prop:composition-of-power-series}Let $R$ be a truncation
  closed subfield of $\no$ and $\varepsilon\prec_{R}1$. Let
  $P(X):=\sum_{n=0}^{\infty}a_{n}X^{n}$ and
  $Q(X):=\sum_{m=1}^{\infty}b_{m}X^{m}$ be two power series in
  $R[[X]]$ (where $Q(X)$ has no constant term). Then
  $(P\circ Q)(\eps)=P(Q(\eps))$.
\end{prop}

\begin{proof}
  The three sums $P(\varepsilon)$, $Q(\eps)$ and
  $(P\circ Q)(\varepsilon)$ exist by
  \prettyref{cor:power-series}. Since $Q(\eps)\prec_{R}1$,
  $P(Q(\eps))$ exists as well. Let
  $d_{n,k}=\sum\limits _{m_{1}+\ldots+m_{n}=k}b_{m_{1}}\cdots b_{m_{n}}$ for
  $k\in\N^{*}$. By \prettyref{cor:power-of-series},
  \[
    P(Q(\varepsilon)) = \sum_{n=0}^{\infty}a_{n}\left(\sum_{m=1}^{\infty}b_{m}\eps^{m}\right)^{n} = a_{0}+\sum_{n=1}^{\infty}a_{n}\sum_{k=1}^{\infty}d_{n,k}\varepsilon^{k}.
  \]
  Note that $d_{n,k}=0$ for $k<n$, so the family
  $(a_{n}d_{n,k}\,:\,n\in\N)$ is summable for any $k\in\N^{*}$. By
  \prettyref{lem:double-index-sum}, the family
  $(a_{n}d_{n,k}\varepsilon^{k}\,:\,(n,k)\in\N\times\N^{*})$ is
  summable. Therefore, by \prettyref{cor:exchange-of-sums} we have
  \[
    a_{0} + \sum_{n=1}^{\infty}a_{n}\sum_{k=1}^{\infty}d_{n,k}\varepsilon^{k}
    = a_{0}+\sum_{k=1}^{\infty}\sum_{n=1}^{\infty}a_{n}d_{n,k}\varepsilon^{k}
    = (P\circ Q)(\varepsilon).
  \]
\end{proof}

\section{Transseries\label{sec:Transseries}}

With the help of the surreal numbers we shall attempt a general
definition of ``field of transseries''.

\begin{defn}
  \label{def:transseries}We say that $T$ is a \textbf{transserial
    subfield }of $\no$ if $T$ is a truncation closed subfield of $\no$
  (\prettyref{def:truncation}) containing $\R$ and such that
  $\log(T^{>0})\subseteq T$.

  More generally, let $F$ be an ordered logarithmic field (not
  necessarily included in $\no$) containing $\R$ and endowed with a
  partial operator $\sum$ from small indexed families of elements of
  $F$ to $F$. We say that $F$ is a \textbf{field of transseries} if it
  is isomorphic to a transserial subfield $T$ of $\no$ through a field
  isomorphism $f:F\to T$ preserving $\R$, $\log$ and $\sum$ (the
  latter condition means that $(x_{i}:i\in I)$ is the domain of $\sum$
  if and only if $(f(x_{i}))_{i\in I}$ is summable in $\no$ and
  $\sum_{i\in I}f(x_{i})=f(\sum_{i\in I}x_{i})$).  We shall call $f$ an
  \textbf{isomorphism of transseries}.
\end{defn}

In \cite{Schmeling2001} an axiomatic definition of transseries field
is given. The critical axiom, there called ``T4'', is rather
technical.  One of the main results in \cite{Berarducci2015} is that
$\no$ satisfies T4, hence it is a field of transseries in the sense of
\cite{Schmeling2001}.  More generally, since T4 is inherited by taking
subfields, it follows that a field of transseries in the sense of
\prettyref{def:transseries} is also a field of transseries in the
sense of \cite{Schmeling2001} (we also expect the converse to be true,
but it is beyond the scope of this paper).

\subsection{Log-atomic numbers }

We write $\log_{n}(x)$ for the $n$-fold iterate of $\log(x)$, namely
$\log_{0}(x)=x$, $\log_{n+1}(x)=\log(\log_{n}(x))$. Likewise, we write
$\exp_{0}(x)=x$, $\exp_{n+1}(x)=\exp(\exp_{n}(x))$.

\begin{defn}
  A positive infinite surreal number $x\in\no$ is \textbf{log-atomic}
  if for every $n\in\N$, $\log_{n}(x)$ is an infinite monomial. We call
  $\li$ the class of all log-atomic numbers. Note that
  $\log(\li)=\exp(\li)=\li$.
\end{defn}

A subclass of the log-atomic numbers, the so called $\kappa$-numbers,
was isolated by \cite{Kuhlmann2014}. The ordinal $\omega$ is a
$\kappa$-number, hence in particular it is log-atomic. In
\cite{Berarducci2015} we gave a parametrization
$\{\lambda_{x}\suchthat x\in\no\}$ of $\li$ and we proved that there is
exactly one log-atomic numbers in each ``level'' of $\no$.

\begin{defn}
  Given $x,y>\R$ we write $x\lequal y$, and we say that $x,y$ are in
  the same \textbf{level} if for some $n\in\N$ we have
  $\log_{n}(x)\veq\log_{n}(y)$.
\end{defn}

\begin{rem}
  \label{rem:one-more-log}For all $x,y>\R$, $x\veq y$ implies
  $\log(x)\sim\log(y)$, so in the above definition we can equivalently
  require $\log_{n}(x)\sim\log_{n}(y)$.
\end{rem}

\begin{fact}[\cite{Berarducci2015}]
  We have:
  \begin{enumerate}
  \item for each $x\in\no$ with $x>\R$, there are $n\in\N$ and
    $\lambda\in\li$ such that $\log_{n}(x)\veq\lambda$ \cite[Prop.\
    5.8]{Berarducci2015}; in particular, every level contains a
    log-atomic number;
  \item for each $\lambda,\mu\in\li$, if $\lambda\lequal\mu$, then
    $\lambda=\mu$; in particular, every level contains a \emph{unique}
    log-atomic number;
  \item for every $x>\R$ and every positive $n\in\N$, we have
    $x\veq^{L}x^{n}$, but $x\not\veq^{L}e^{x}$;
  \item in particular, for $\lambda,\mu\in\li$, if $\lambda<\mu$, then
    $\lambda^{n}<\mu$ for every $n\in\N$;
  \item there are log-atomic numbers strictly between $\omega$ and
    $e^{\omega}$; there are also log-atomic numbers smaller than
    $\log_{n}(\omega)$ for every $n\in\N$ or bigger than
    $\exp_{n}(\omega)$ for every $n\in\N$, such as the ordinal
    $\varepsilon_{0}$.
  \end{enumerate}
\end{fact}

\subsection{Omega-series, LE-series, EL-series \label{subsec:LE-series}}

In this section we shall introduce three subfields
$\R((\omega))^{LE} \subset \R((\omega))^{EL} \subset\bracket{\omega}$ of
$\no$. We shall see that first two are naturally isomorphic to the
exponential fields of respectively the LE-series of \cite{Dries1997,
  DriesMM2001} and the EL-series generated by logarithmic words of
\cite{Kuhlmann2000, Kuhlmann2012d}, while the third one is a very big
field properly containing both (the ordinal $\omega$ plays the role of
a formal variable $>\R$).

\begin{defn}
  \label{def:set-operations}Given a subclass $X$ of $\no$, we write
  \textbf{$\sum X$} for the family of all surreal numbers $x\in\no$
  which can be written in the form $x=\sum_{i\in I}y_{i}$ for some
  summable family $(y_{i})_{i\in I}$ of elements of $X$ indexed by a set
  $I$. Note that $\sum$ is a closure operator, as
  $X\subseteq\sum X=\sum\sum X$.
\end{defn}

\begin{defn}
  \label{def:omega-series}We define $\bracket{\omega}$, the field of
  \textbf{omega-series}, as the smallest subfield of $\no$ containing
  $\R\cup\{\omega\}$ and closed under $\sum$, $\exp$ and $\log$.
\end{defn}

We shall prove later that $\bracket{\omega}$ is a proper class.

\begin{defn}
  \label{def:newLE}Let $\R((\omega))^{LE}\subset\bracket{\omega}$ be
  the union $\bigcup_{n\in\N}X_{n}$, where $X_{0}=\R\cup\{\omega\}$ and
  $X_{n+1}=\sum(X_{n}\cup\exp(X_{n})\cup\log(X_{n}))$. In other words, a
  surreal number $x$ belongs to $\R((\omega))^{LE}$ if and only if $x$
  can be obtained in finitely many steps starting from
  $\R\cup\{\omega\}$ and using the set-operations $\sum,\exp,\log$.
\end{defn}

\begin{defn}
  Let $\R((\omega))^{EL}$ be defined as $\R((\omega))^{LE}$ but starting
  with $X_{0}'=\R\cup\{\omega,\log(\omega),\log_{2}(\omega),\ldots\}$
  instead of $X_{0}=\R\cup\{\omega\}$. In other words, a surreal number
  belongs to $\R((\omega))^{EL}$ if and only if it can be obtained in
  finitely many steps from $X_{0}'$ using $\sum,\exp,\log$ (in this
  case it turns out that $\log$ is not actually necessary).
\end{defn}

\begin{rem}
  Unlike $\bracket{\omega}$, the subfields $\R((\omega))^{LE}$ and
  $\R((\omega))^{EL}$ are not closed under $\sum$; for instance
  $\sum_{n\in\N}\log_{n}(\omega)$ belongs to $\bracket{\omega}$ but not
  to $\R((\omega))^{LE}$. Indeed, one needs $k$ steps to generate
  $\log_{k}(\omega)$ starting from $\R\cup\{\omega\}$, so the whole sum
  $\sum_{n\in\N}\log_{n}(\omega)$ cannot be generated in finitely many
  steps. The same example witnesses that the inclusion
  $\R((\omega))^{LE}\subset\R((\omega))^{EL}$ is proper, as the latter
  field does contain $\sum_{n\in\N}\log_{n}(\omega)$.  Finally note that
  $\sum_{n\in\N}1/\exp_{n}(\omega)$ belongs to $\bracket{\omega}$ but
  not to $\R((\omega))^{EL}$.

  Both $\R((\omega))^{LE}$ and $\R((\omega))^{EL}$ are elementary
  extensions of the real exponential field $(\R,+,\cdot,\exp)$, but
  they are no longer elementary equivalent if we add the differential
  operator $\partial$ of \cite{Berarducci2015} to the language (see
  \prettyref{subsec:Transserial-derivations}): indeed in
  $\R((\omega))^{LE}$ (and in $\no$ itself) the derivation $\partial$
  is surjective, while in $\R((\omega))^{EL}$ it is not. For instance
  one can show that $\exp(-\sum_{n\in\N}\log_{n}(\omega))$ is an element
  of $\R((\omega))^{EL}$ without anti-derivative in
  $\R((\omega))^{EL}$, and in fact not even in
  $\bracket{\omega}$. Indeed, for the simplest surreal derivation
  $\partial$ (\cite[Def.\ 6.7]{Berarducci2015}), which has
  anti-derivatives, we have
  $\partial\kappa_{-1}=\exp(-\sum_{n\in\N}\log_{n}(\omega))$, where
  $\kappa_{-1}\in\no$ is the simplest log-atomic number smaller than
  $\log_{n}(\omega)$ for each $n\in\N$. Such a number cannot belong to
  $\bracket{\omega}$, and since $\ker\partial=\R$, there cannot be any
  $x\in\bracket{\omega}$ with
  $\partial x=\exp(-\sum_{n\in\N}\log_{n}(\omega))$.
\end{rem}

There are many interesting subfields between $\R((\omega))^{LE}$ and
$\bracket{\omega}$ whose domain is a set, for instance the series in
$\bracket{\omega}$ with hereditarily countable support.

The definition of $\R((\omega))^{LE}$ as a union
$\bigcup_{n\in\N}X_{n}$ suggests the possibility of prolonging the
sequence $X_{n}$ along the transfinite ordinals, setting
$X_{0}=\R\cup\{\omega\}$,
$X_{\alpha+1}=\sum(X_{\alpha}\cup\exp(X_{\alpha})\cup\log(X_{\alpha}))$
and $X_{\lambda}=\bigcup_{i<\lambda}X_{i}$ for each limit ordinal
$\alpha$. One can verify that the union
$\bigcup_{\alpha\in\on}X_{\alpha}$ along all the ordinals would then
coincide with $\bracket{\omega}$.

\subsection{Isomorphism with classical LE-series}

It is well known that there is a unique embedding of the field of
LE-series into $\no$ sending $x$ to $\omega$ \footnote{In
  \cite{DriesMM2001}, the field of logarithmic-exponential series is
  denoted either by $\R((x^{-1}))^{LE}$ or by $\R((t))^{LE}$, where
  $x>\R$ and $t=x^{-1}$ is infinitesimal. We prefer here to use the
  notation $\R((x))^{LE}$ for the LE-series, with $x>\R$, as in
  \cite{Aschenbrenner2015}, to better match the notation
  $\R((\omega))^{LE}$.}, $\R$ to $\R$, and preserving $\exp$ and
infinite sums (see \cite{Aschenbrenner2015a}).  This subsection will
be devoted to the long, but straightforward proof that
$\R((\omega))^{LE}$ is naturally isomorphic to the field of LE-series,
so in particular it is the image of such embedding. This provides a
simple characterization of the LE-series, which should be compared
with the original definition.

\begin{thm}
  \label{thm:LE}$\R((\omega))^{LE}$ is a field of transseries and it is
  isomorphic to the field of logarithmic-exponential series
  $\R((x))^{LE}$ of \cite{Dries1997,DriesMM2001}; the isomorphism
  sending $\omega$ to $x$ is unique.
\end{thm}

Similarly we have:

\begin{prop}
  \label{prop:characterization-EL}The field $\R((\omega))^{EL}$ is
  naturally isomorphic to the field of EL-series generated by
  logarithmic words \cite[Def.\ 6.2, Example 4.6]{Kuhlmann2012d} (see
  also \prettyref{rem:LE-delta-n}).
\end{prop}

We leave the verification of \prettyref{prop:characterization-EL} to
the reader, but we shall give a detailed proof of \prettyref{thm:LE}.
To this aim we shall first give an equivalent description of
$\R((\omega))^{LE}$ (recall from \prettyref{nota:Hahn-subfield-No-1}
that we are identifying Hahn fields $R((\Gamma))$ with subfields of
$\no$).

\begin{defn}
  \label{def:LE}Let $\lambda\in\li$ (a log-atomic number). We define:

  \begin{tabular*}{10cm}{@{\extracolsep{\fill}}clll}
    (1) & $\M_{0,\lambda}:=\lambda^{\R}$, & $K_{0,\lambda}:=\R((\M_{0,\lambda}))$, & $\J_{0,\lambda}:=\R((\M_{0,\lambda}^{\succ1}))$;\\
    (2) & $\M_{n+1,\lambda}:=e^{\J_{n,\lambda}}$, & $K_{n+1,\lambda}=K_{n,\lambda}((\M_{n+1,\lambda}))$, & $\J_{n+1,\lambda}:=K_{n,\lambda}((\M_{n+1,\lambda}^{\succ1}))$;\\
    (3) & \multicolumn{2}{l}{$\R((\lambda))^{E}:=\bigcup_{n\in\N}K_{n,\lambda}$.} &
  \end{tabular*}
\end{defn}

The next Lemma shows that the above definition is well posed, namely
at each step $\M_{n+1,\lambda}$ is a subgroup of $\M$ and
$K_{n,\lambda}<\M_{n+1,\lambda}^{>1}$, so that under the conventions of
\prettyref{nota:Hahn-subfield-No-1}, each $K_{n+1,\lambda}$ is again in
$\no$; in particular, in clause (2) we are allowed to use the
exponential function of $\no$ to define $e^{\J_{n+1,\lambda}}$. Note
moreover that $\J_{n,\lambda}\subseteq\J$, as we shall verify in a
moment.

\begin{lem}
  \label{lem:Kn}For each $n\in\N$, $\M_{n,\lambda}$ is a well defined
  divisible subgroup of $\M$ and moreover
  $\M_{n+1,\lambda}^{\succ1}>K_{n,\lambda}$.
\end{lem}

\begin{proof}
  We proceed by induction on $n$. Trivially, $\M_{0,\lambda}$ is a well
  defined divisible subgroup of $\M$. Now fix $n$ and assume that
  $\M_{n,\lambda}$ is well defined and that
  $\M_{n,\lambda}^{\vgreater1}>K_{n-1,\lambda}$ (an empty condition if
  $n=0$). Then $\J_{n,\lambda}$ is a well defined subset of $\no$ by
  \prettyref{prop:Hahn-subfield-No}, and in particular it is a
  divisible additive subgroup of $K_{n,\lambda}$.

  We claim that $\J_{n,\lambda}$ is consists only of purely infinite
  numbers. Indeed, let $\m$ be a monomial in the support of
  $\J_{n,\lambda}$.  Then $\m=\n\o$ for some
  $\n\in\M_{n,\lambda}^{\vgreater1}$ and $\o\in K_{n-1,\lambda}$ (with
  $\o=1$ if $n=0$). By inductive hypothesis,
  $\o^{-1}\in K_{n-1,\lambda}<\n$, so $\m>1$, proving the claim. It
  follows that $\M_{n+1,\lambda}$ is a divisible multiplicative
  subgroup of $\M$.

  Finally, let $e^{\gamma}\in\M_{n+1,\lambda}^{\vgreater1}$. We wish to
  prove that $e^{\gamma}>\M_{n,\lambda}$. Let $\m$ be the leading
  monomial of $\gamma$. As before, we can write $\m=\n\o$ for some
  $\n\in\M_{n,\lambda}^{\vgreater1}$ and $\o\in K_{n-1,\lambda}$ (with
  $\o=1$ if $n=0$). By inductive hypothesis, we also know that
  $\n^{\frac{1}{2}}>K_{n-1,\lambda}$.  Since
  $\gamma>\n^{\frac{1}{2}}$, it follows that
  $\m=e^{\gamma}>e^{K_{n-1,\lambda}}$, so in particular,
  $\m>e^{\J_{n-1,\lambda}}=\M_{n,\lambda}$, as desired.
\end{proof}

\begin{rem}
  By \prettyref{cor:Hahn-lexicographic-1} we have
  $K_{n,\lambda}=\R((\M_{0,\lambda}\M_{1,\lambda}\dots\M_{n,\lambda}))$.
\end{rem}

\begin{lem}
  \label{lem:LE-exp-log-inclusions}For all $n\in\N$ we have:
  \begin{enumerate}
  \item $\exp(K_{n,\lambda})\subseteq K_{n+1,\lambda}$;
  \item $K_{n,\lambda}\subseteq K_{n+1,\log(\lambda)}$;
  \item $\log(K_{n,\lambda}^{>0})\subseteq K_{n+1,\log(\lambda)}$.
  \end{enumerate}
  In particular, $\R((\lambda))^{E}$ is closed under $\exp$ and
  $\log(\R((\lambda))^{E})\subseteq\R((\log(\lambda)))^{E}$.
\end{lem}

\begin{proof}
  We work by induction on $n$.

  For (1), let $x\in K_{n,\lambda}$. We can write uniquely
  $x=\gamma+r+\varepsilon$ where $\gamma\in\J_{n,\lambda}$, and if
  $n>0$, $r\in K_{n-1,\lambda}$ and
  $\varepsilon\vless_{K_{n-1,\lambda}}1$, otherwise simply $r\in\R$ and
  $\varepsilon\prec1$. In any case,
  $e^{x}=e^{\gamma}\cdot
  e^{r}\cdot\sum_{i=0}^{\infty}\frac{\varepsilon^{i}}{i!}$.  But then it
  suffices to note that
  $e^{\gamma}\in\M_{n+1,\lambda}\subseteq K_{n+1,\lambda}$ by definition,
  while $e^{r}$ is either already in $\R$ or in $K_{n,\lambda}$ by
  inductive hypothesis, and the remaining sum is in $K_{n,\lambda}$
  because $K_{n,\lambda}$ is a Hahn field. Therefore,
  $e^{x}\in K_{n+1,\lambda}$, as desired.

  Concerning (2), note that
  $\M_{0,\lambda}=\lambda^{\R}=e^{\R\log(\lambda)}\subseteq
  e^{\J_{0,\log(\lambda)}}=\M_{1,\log(\lambda)}$.  It follows that
  $\J_{0,\lambda}\subseteq\J_{1,\log(\lambda)}$ and
  $K_{0,\lambda}\subseteq K_{1,\log(\lambda)}$. By a straightforward
  induction, it follows that
  $\M_{n,\lambda}\subseteq\M_{n+1,\log(\lambda)}$,
  $\J_{n,\lambda}\subseteq\J_{n+1,\log(\lambda)}$ and
  $K_{n,\lambda}\subseteq K_{n+1,\log(\lambda)}$, proving the desired
  conclusion.

  Finally, for (3), let $x\in K_{n,\lambda}^{>0}$. We can write uniquely
  $x=\m\cdot r\cdot(1+\varepsilon)$ where $\m\in\M_{n,\lambda}$, and if
  $n>0$, $r\in K_{n-1,\lambda}^{>0}$ and
  $\varepsilon\vless_{K_{n-1,\lambda}}1$, otherwise simply $r\in\R$ and
  $\varepsilon\prec1$. We have
  $\log(x)=\log(\m)+\log(r)+\sum_{i=1}^{\infty}(-1)^{i+1}\frac{\varepsilon^{i}}{i}$.
  Since $K_{n,\lambda}$ is a Hahn field, the rightmost sum is in
  $K_{n,\lambda}$, which is contained in $K_{n+1,\log(\lambda)}$ by
  (2), while $\log(r)$ is either already in $\R$ or in
  $K_{n,\log(\lambda)}$ by inductive hypothesis. For $\log(\m)$, we
  simply note that if $n=0$, then
  $\log(\m)=s\cdot\log(\lambda)\in K_{0,\log(\lambda)}$ for some
  $s\in\R$, otherwise $\log(\m)\in K_{n-1,\lambda}$, which is
  contained in $K_{n,\log(\lambda)}$ by (2). Therefore,
  $\log(x)\in K_{n+1,\log(\lambda)}$, as desired.
\end{proof}

\begin{prop}
  \label{prop:E-is-E}For each $\lambda\in\li$, $\R((\lambda))^{E}$ is
  (uniquely) isomorphic to the exponential field $\R((x))^{E}$ defined
  in \cite{Dries1997,DriesMM2001} through an isomorphism sending
  $\lambda$ to $x$ and preserving $\exp$, $\sum$ and $\R$.
\end{prop}

\begin{proof}
  It suffices to note that \prettyref{def:LE} is formally identical to
  the definition of $\R((x))^{E}$, except that in our case the various
  Hahn fields are identified with subfields of $\no$
  (\prettyref{nota:Hahn-subfield-No-1}) and the role of the formal
  variable is taken by $\lambda$. The uniqueness follows trivially.
\end{proof}

\begin{prop}
  \label{prop:LE-is-LE}For each $\lambda\in\li$,
  $\bigcup_{k\in\N}\R((\log_{k}(\lambda)))^{E}$ is (uniquely) isomorphic
  to the exponential field $\R((x))^{LE}$ defined in
  \cite{Dries1997,DriesMM2001} through an isomorphism sending
  $\lambda$ to $x$ and preserving $\exp$, $\sum$ and $\R$.
\end{prop}

\begin{proof}
  In \cite{DriesMM2001}, $\R((x))^{LE}$ is defined as a direct limit of
  a suitable system of self-embeddings
  $\Phi_{k}:\R((x))^{E}\to\R((x))^{E}$.  The embedding $\Phi_{k}$ sends
  $x$ to $\exp_{k}(x)$. In turn, when composed with the isomorphism
  $\R((x))^{E}\cong\R((\log_{k}(\lambda)))^{E}$ of
  \prettyref{prop:E-is-E}, it gives the embedding of $\R((x))^{E}$ into
  $\R((\log_{k}(\lambda)))^{E}$ sending $x$ to $\lambda$. Therefore, the
  image of such direct limit is the directed union
  $\bigcup_{k\in\N}\R((\log_{k}(\lambda)))^{E}$, as desired. The
  uniqueness follows trivially.
\end{proof}

\begin{prop}
  \label{prop:LE}$\bigcup_{k\in\N}\R((\log_{k}(\omega)))^{E}$ is equal to
  $\R((\omega))^{LE}$. In particular, there is a unique isomorphism of
  transseries from $\bigcup_{k\in\N}\R((\log_{k}(\omega)))^{E}$ to
  $\R((x))^{LE}$ sending $\omega$ to $x$.
\end{prop}

\begin{proof}
  Note that each $K_{n,\lambda}$ is closed under infinite sums, while
  by \prettyref{lem:LE-exp-log-inclusions},
  $\exp(K_{n,\lambda})\subseteq K_{n+1,\lambda}$ and
  $\log(K_{n,\lambda}^{>0})\subseteq K_{n+1,\log(\lambda)}$. Since
  $X_{0}\subseteq K_{0,\omega}$, it follows at once that
  $X_{n}\subseteq K_{n,\log_{n}(\omega)}$ for all $n\in\N$, so in
  particular
  $\R((\omega))^{LE}\subseteq\bigcup_{k\in\N}\R((\log_{k}(\omega)))^{E}$.

  Conversely, it is clear that each element of
  $K_{n,\log_{n}(\omega)}$ can be obtained from
  $X_{0}=\R\cup\{\omega\}$ by finitely many applications of $\exp$,
  $\log$ and infinite sums. It follows at once that
  $\bigcup_{k\in\N}\R((\log_{k}(\omega)))^{E}\subseteq\R((\omega))^{LE}$.
\end{proof}

\prettyref{thm:LE} then follows at once by Propositions \ref{prop:LE}
and \ref{prop:LE-is-LE}.

\begin{rem}
  If we modify \prettyref{def:LE} putting
  $\M_{0,\lambda}:=\lambda^{\Z}$ instead of $\lambda^{\R}$, the union
  $\bigcup_{k\in\N}\R((\log_{k}(\omega)))^{E}$ will be the same, since
  $\lambda^{\R}=\exp(\R\log\lambda)$. So in the definition of the
  LE-series in \cite{DriesMM2001} one may start with $x^{\Z}$ instead
  of $x^{\R}$.
\end{rem}

\subsection{Adding more log-atomic numbers\label{subsec:L-series}}

\begin{defn}
  \label{def:Delta-series}Consider the class $\li\subseteq\no$ of
  log-atomic numbers and let $\bracket{\li}$ be the smallest subfield
  of $\no$ containing $\R\cup\li$ and closed under $\exp$, $\log$ and
  $\sum$ (in the sense of \prettyref{def:set-operations}).
\end{defn}

In \cite[Thm.\ 8.6]{Berarducci2015} we showed that $\bracket{\li}$ is
the largest subfield of transseries satisfying axiom ELT4 of
\cite[Def.\ 5.1]{Kuhlmann2014}.  We also showed that $\no$ itself does
not satisfy ELT4, hence $\bracket{\li}\neq\no$ \cite[Thm.\
8.7]{Berarducci2015}. The derivative $\partial:\no\to\no$ introduced
in \cite[Def.\ 6.21]{Berarducci2015} can be restricted to
$\bracket{\li}$ and remains surjective on this subfield. We thus have
the inclusions
\[
  \R((\omega))^{LE} \subset \R((\omega))^{EL} \subset \bracket{\omega}
  \subset \bracket{\li} \subset \no
\]
with $\R((\omega))^{LE}$, $\bracket{\li}$ and $\no$ having a surjective
derivation, while the derivation on $\R((\omega))^{EL}$ and
$\bracket{\omega}$ is not surjective. It would be interesting to study
the complete first order theories of these structures, both as
differential fields, and as differential fields with an
exponentiation. The only known result so far is that $\no$ and
$\R((\omega))^{LE}$ are elementary equivalent as differential fields
\cite{Aschenbrenner2015a}, and probably the same proof can be used to
deduce that $\bracket{\li}$ has the same first order theory as well.

\subsection{Inductive generation of transseries fields and associated ranks}

For the purposes of \prettyref{sec:Substitutions}, it is useful to
inductively construct $\bracket{\omega}$ and other subfields of
$\bracket{\li}$ with a limited use of the $\log$ function, and to
introduce a rank function reflecting the stages of the inductive
construction. We need the following definition.

\begin{defn}
  Let $\Delta\subseteq\li$ be a subclass with
  $\log(\Delta)\subseteq\Delta$ and let $\bracket{\Delta}$ be the
  smallest subfield of $\no$ containing $\R\cup\Delta$ and closed
  under $\sum$, $\exp$ and $\log$.
\end{defn}

As we shall see \prettyref{cor:bracket-delta-is-closure},
$\bracket{\Delta}$ coincides with the smallest subclass of $\no$
containing $\R\cup\Delta$ and closed under $\sum$ and $\exp$ (or even
just $\exp_{\restriction\J}$); the closure under $\log$ can be
automatically deduced. Taking $\Delta=\li$, we obtain the field
$\bracket{\li}$ seen in \prettyref{subsec:L-series}.  On the other
hand, when $\Delta=\{\log_{n}(\omega)\suchthat n\in\N\}$, we obtain
$\bracket{\omega}$ (\prettyref{def:omega-series}).

\begin{notation}
  Given a subclass $A\subseteq\M$, we denote by
  $\R((A))_{\mathrm{small}}$ (or just $\R((A))$ if $A$ is a set) the
  class of all surreal numbers with support contained in $A$. Notice
  that if $A$ is a group, $\R((A))_{\sml}$ is a field, but we
  occasionally use the notation without assuming that $A$ is a group.
\end{notation}

\begin{defn}
  \label{def:erank}Let $\log(\Delta)\subseteq\Delta\subseteq\li$.  We
  define by induction on the ordinal $\alpha\in\on$ a subclass
  $\Delta_{\alpha}\subseteq\no$ as follows: $\Delta_{0}=\emptyset$,
  $\Delta_{1}=\Delta\cup\{0\}$;
  $\Delta_{\alpha+1}=\R((e^{\Delta_{\alpha}\cap\J}))_{\sml}$ for
  $\alpha\geq1$;
  $\Delta_{\lambda}=\bigcup_{\alpha<\lambda}\Delta_{\alpha}$ for
  $\lambda$ a limit ordinal. Given
  $x\in\bigcup_{\alpha\in\on}\Delta_{\alpha}$, we define the
  \textbf{(exponential) rank} $\erank(x)$ as the least ordinal $\beta$
  such that $x\in\Delta_{\beta+1}$.
\end{defn}

\begin{rem}
  \label{rem:Delta3-field}Note that $\Delta_{1}$ is not an additive
  group. For $\alpha\geq2$, $\Delta_{\alpha}$ is an $\R$-linear
  subspace of $\no$ (and it is closed under $\sum$); for
  $\alpha\geq3$, $\Delta_{\alpha}$ is a field, and a Hahn field when
  $\alpha$ is a successor ordinal.  Moreover, all the classes
  $\Delta_{\alpha}$ are truncation closed.
\end{rem}

\begin{prop}
  \label{prop:Delta-alpha-beta}For all $\alpha<\beta$ we have
  $\Delta_{\alpha}\subseteq\Delta_{\beta}$.
\end{prop}

\begin{proof}
  It suffices to prove that
  $\Delta_{\alpha}\subseteq\Delta_{\alpha+1}$ for all
  $\alpha\in\on$. This is clear for $\alpha=0$. Since
  $\log(\Delta)\subseteq\Delta$, we have
  $\Delta\subseteq e^{\Delta}\subseteq\R((e^{\Delta}))_{\sml}$, thus
  $\Delta_{1}\subseteq\Delta_{2}$, proving the case $\alpha=1$.  We then
  proceed by induction. If $\alpha=\beta+1$, then
  $\Delta_{\beta}\subseteq\Delta_{\beta+1}$ holds by inductive
  hypothesis, so
  $\Delta_{\alpha} = \R((e^{\Delta_{\beta}\cap\J}))_{\sml} \subseteq
  \R((e^{\Delta_{\beta+1}\cap\J}))_{\sml} = \Delta_{\alpha+1}$.  If
  $\alpha$ is a limit ordinal, take some $x\in\Delta_{\alpha}$.  By
  definition of $\Delta_{\alpha}$, there is some $\beta<\alpha$ such
  that $x\in\Delta_{\beta}$, so by inductive hypothesis,
  $x\in\Delta_{\beta+1} = \R((e^{\Delta_{\beta}\cap\J}))_{\sml} \subseteq
  \R((e^{\Delta_{\alpha}\cap\J}))_{\sml} = \Delta_{\alpha+1}$.  Since
  $x$ is arbitrary, we obtain
  $\Delta_{\alpha} \subseteq \Delta_{\alpha+1}$, as desired.
\end{proof}

The following corollary provides an equivalent definition of the rank.
Its proof is easy and left to the reader.

\begin{cor}
  \label{cor:erank}For $x\in\bigcup_{\alpha\in\on}\Delta_{\alpha}$ we
  have
  \begin{enumerate}
  \item if $x\in\Delta\cup\{0\}$, then $\erank(x)=0$;
  \item otherwise,
    $\erank(x)=\sup\{\erank(\gamma)+1\suchthat
    e^{\gamma}\in\supp(x)\}$.
  \end{enumerate}
  Moreover, $x\in\Delta_{\beta}$ if and only if $\erank(x)<\beta$.
\end{cor}

\begin{prop}
  \label{prop:unionDelta-exp-log-closed}We have:
  \begin{enumerate}
  \item for all $\alpha\geq1$,
    $\sum\Delta_{\alpha+1}\subseteq\Delta_{\alpha+1}$ (in particular,
    $\sum\Delta_{\alpha}\subseteq\Delta_{\alpha+2}$ for all $\alpha$);
  \item for all $\alpha\geq3$,
    $\log(\Delta_{\alpha}^{>0})\subseteq\Delta_{\alpha}$;
  \item for all $\alpha\in\on$,
    $e^{\Delta_{\alpha}}\subseteq\Delta_{\alpha+1}$ (in particular,
    $e^{\Delta_{\alpha}}\subseteq\Delta_{\alpha}$ for all limit
    $\alpha$).
  \end{enumerate}
  In particular, $\Delta_{\alpha}$ is a transserial subfield of $\no$
  for all $\alpha\geq3$, and $\bigcup_{\alpha\in\on}\Delta_{\alpha}$ is
  closed under $\exp$, $\log$ and infinite sums.
\end{prop}

\begin{proof}
  (1) Trivial, since by definition
  $\Delta_{\alpha+1}=\R((e^{\Delta_{\alpha}\cap\J}))_{\sml}$ for
  $\alpha\geq1$.

  (2) Without loss of generality, we may assume that $\alpha$ is of
  the form $\beta+1$ with $\beta\geq2$, so that $\Delta_{\alpha}$ is a
  Hahn field (see \prettyref{rem:Delta3-field}). Take any
  $x\in\Delta_{\alpha}^{>0}$.  We can write uniquely
  $x=re^{\gamma}(1+\varepsilon)$, where $r\in\R^{>0}$,
  $\gamma\in\Delta_{\beta}\cap\J$ and
  $\varepsilon\in\Delta_{\alpha}\cap o(1)$.  Then
  $\log(x) = \gamma + \log(r) +
  \sum_{n=1}^{\infty}(-1)^{n}\frac{\varepsilon^{n}}{n}\in\Delta_{\beta} +
  \Delta_{\alpha} = \Delta_{\alpha}$ by
  \prettyref{prop:Delta-alpha-beta}. It follows that
  $\log(\Delta_{\alpha}^{>0})\subseteq\Delta_{\alpha}$, as desired.

  (3) Note that the conclusion is trivially true for $\alpha=0,1$, so
  we may assume that $\alpha\geq2$. Take any $x\in\Delta_{\alpha}.$
  Since $\Delta_{\alpha}$ is closed under truncation (see again
  \prettyref{rem:Delta3-field}), we can write uniquely
  $x=\gamma+r+\varepsilon$, with $\gamma\in\Delta_{\alpha}\cap\J$,
  $r\in\R$ and $\varepsilon\in\Delta_{\alpha}\cap o(1)$. Since
  $\Delta_{\alpha+1}$ is a Hahn field, we have
  $e^{x} = e^{\gamma} \cdot e^{r} \cdot
  \sum_{n=0}^{\infty}\frac{\varepsilon^{n}}{n!} \in
  \R((e^{\Delta_{\alpha}\cap\J}))_{\sml} = \Delta_{\alpha+1}$, as desired.
\end{proof}

\begin{cor}
  $\bracket{\Delta}=\bigcup_{\alpha\in\on}\Delta_{\alpha}$.
\end{cor}

\begin{proof}
  By \prettyref{prop:unionDelta-exp-log-closed},
  $\bigcup_{\alpha\in\on}\Delta_{\alpha}$ contains $\R$ and $\Delta$ (as
  both are contained in $\Delta_{2}$) and it is closed under $\exp$,
  $\log$ and infinite sums. It follows that
  $\bracket{\Delta}\subseteq\bigcup_{\alpha\in\on}\Delta_{\alpha}$.  On
  the other hand, one can easily verify by induction that
  $\Delta_{\alpha}\subseteq\bracket{\Delta}$ for all $\alpha\in\on$,
  and the conclusion follows.
\end{proof}

\begin{cor}
  \label{cor:bracket-delta-is-closure}$\bigcup_{\alpha\in\on} =
  \bracket{\Delta}$ is the smallest class containing $\Delta\cup\{0\}$
  and such that whenever the exponents $\gamma_{i}\in\J$ of
  $x=\sum_{i<\alpha}r_{i}e^{\gamma_{i}}$ are in the class, then also
  $x$ is in the class. The ordinal $\erank(x)$ measures the number of
  steps needed to obtain $x$ with this inductive construction.
\end{cor}

\begin{cor}
  $\bracket{\Delta}$ is truncation closed, so it is a field of
  transseries in the sense of \prettyref{def:transseries}.
\end{cor}

\begin{proof}
  Immediate from the equality
  $\bracket{\Delta} = \bigcup_{\alpha\in\on}\Delta_{\alpha}$.
\end{proof}

\begin{cor}
  $\bracket{\Delta}$ is a proper class. In particular,
  $\bracket{\omega}$ is a proper class.
\end{cor}

\begin{proof}
  Let $\Gamma=\M\cap\bracket{\Delta}$ be the class of monomials of
  $\bracket{\Delta}$. Since $\bracket{\Delta}$ is closed under $\sum$
  and truncations, we have $\bracket{\Delta}=\R((\Gamma))_{\sml}$.  If
  for a contradiction $\bracket{\Delta}$ were a set, then
  $\bracket{\Delta}=\R((\Gamma))$.  Since on the other hand
  $\bracket{\Delta}$ is an exponential subfield of $\no$,
  $\R((\Gamma))$ would then carry a compatible exponential function,
  contradicting \cite{Kuhlmann1997}.
\end{proof}

The following remark is implicit in our previous observations, but it
is worth to record it:

\begin{rem}
  \label{rem:LE-delta-n}Let
  $\Delta=\{\log_{n}(\omega)\suchthat n\in\N\}$.  Then
  $\R((\omega))^{EL}=\bigcup_{n\in\N}\Delta_{n}=\Delta_{\omega}$.
\end{rem}

\section{Substitutions\label{sec:Substitutions}}

Before defining the full notion of composition, we first define
\emph{substitutions} (also called right-compositions in
\cite{Schmeling2001}).

\begin{defn}
  Let $T$ be a field of transseries. We say that $f:T\to\no$ is
  \textbf{strongly additive} if for every summable sequence
  $(x_{i}\suchthat i\in I)$ in $T$, the sequence
  $(f(x_{i})\suchthat i\in I)$ in $\no$ is summable and
  $f(\sum_{i\in I}x_{i})=\sum_{i\in I}f(x_{i})$.
\end{defn}

\begin{defn}
  Let $T$ a field of transseries. A \textbf{substitution} $c:T\to\no$
  is a strongly additive map which is the identity on $\R$ and
  preserves $\log$, namely $c(\log(x))=\log(c(x))$ for all $x\in T$.
\end{defn}

It is fairly easy to check that the substitutions are well behaved
functions.

\begin{prop}
  \label{prop:substitution-is-hom}Let $c:T\to\no$ be a substitution.
  Then $c$ is an ordered field isomorphism fixing $\R$. In particular,
  for all $x,y\in T$ we have $x<y\to c(x)<c(y)$ and therefore
  $x\prec y\to c(x)\prec c(y)$.
\end{prop}

\begin{proof}
  Fix some $x,y\in T$. Clearly, $c$ is additive. Moreover, if $x>0$,
  then $\log(x)\in T$, so $c(\log(x))=\log(c(x))$, so $c(x)>0$, and in
  particular, $c$ preserves the ordering. If $x,y>0$, then
  $c(xy)=c(e^{\log(xy)})=e^{c(\log(x)+\log(y))}=e^{c(\log(x))}\cdot
  e^{c(\log(y))}=c(x)c(y)$, and it follows easily that $c$ is
  multiplicative. Therefore, $c$ is an ordered field isomorphism which
  by definition fixes $\R$. In particular, if $x<y$, then
  $c(x)<c(y)$. Moreover, if $x\vless y$, then $r|x|<|y|$ for all
  $r\in\R$, so $r|c(x)|<|c(y)|$ for all $r\in\R$, so
  $c(x)\vless c(y)$.
\end{proof}

In this section, we show how to construct inductively substitutions on
fields of the form $\bracket{\Delta}$ starting from their values on
some subclass $\Delta\subseteq\li$. The proof that the construction is
well defined is fairly complicated and technical; for the sake of
readability, the proof of one of the intermediate statement, the
``summability lemma'' \ref{lem:summability}, will be postponed to
\prettyref{sec:induction}.

\subsection{Pre-substitutions}

To build a substitution on $\bracket{\Delta}$, we start with a certain
assignment of values to each element of $\Delta$ satisfying some
suitable compatibility conditions. We call such assignment a
pre-substitution.

\begin{defn}
  \label{def:pre-composition}A map $c_{0}:\Delta\to\no$ is a
  \textbf{pre-substitution} if
  \begin{enumerate}
  \item the domain $\Delta$ is a subclass of $\li$ closed under
    $\log$;
  \item $c_{0}(\lambda)>0$ and
    $c_{0}(\log(\lambda))=\log(c_{0}(\lambda))$ for all
    $\lambda\in\Delta$;
  \item for any decreasing sequence $(\lambda_{i}\in\Delta)_{i\in\N}$,
    the family $\left(c_{0}(\lambda_{i})\right)_{i\in\N}$ is summable;
  \item for any increasing sequence $(\lambda_{i}\in\Delta)_{i\in\N}$,
    the family $\left(c_{0}(\lambda_{i})^{-1}\right)_{i\in\N}$ is
    summable;
  \item for all $\lambda,\mu\in\Delta$, if $\lambda<\mu$, then
    $c_{0}(\lambda)\prec c_{0}(\mu)$.
  \end{enumerate}
\end{defn}

\begin{rem}
  By (1) and (2) it follows by induction on $n\in\N$ that
  $c_{0}(\lambda)>\exp_{n}(0)$ for every $\lambda\in\Delta$, and
  therefore for all $\lambda\in\Delta$ we have
  $1\vless c_{0}(\lambda)$. Moreover, if $\lambda<\mu$, then
  $c_{0}(\lambda)<c_{0}(\mu)$ and
  $c_{0}(\lambda)^{n}\vless c_{0}(\mu)$ for all $n\in\N$ (since
  $\log(c_{0}(\lambda))=c_{0}(\log(\lambda))\prec
  c_{0}(\log(\mu))=\log(c_{0}(\mu))$).
\end{rem}

Clearly, if $\Delta\subseteq\li$ is a class closed under $\log$ and
$c:\bracket{\Delta}\to\no$ is a substitution, then
$c_{\restriction\Delta}$ is a pre-substitution. We shall prove that the
converse holds, namely that every pre-substitution with domain
$\Delta$ extends to a (unique) substitution with domain
$\bracket{\Delta}$ (\prettyref{thm:pre-comp-extends}), and as a
corollary we shall deduce the existence of substitutions on
$\bracket{\omega}$ (\prettyref{cor:substitution-omega-series}).  We
first give an explicit example of pre-substitution on
$\Delta=\{\log_{i}(\omega)\suchthat i\in\N\}$.

\begin{prop}
  \label{prop:sum-of-logs}Let $x\in\no^{>\N}$. Then the sequence
  $(\log_{i}(x))_{i\in\N}$ is summable.
\end{prop}

\begin{proof}
  By \cite[Prop.\ 5.8]{Berarducci2015}, there is an integer $k\in\N$
  and some log-atomic number $\mu\in\li$ such that
  $\log_{k}(x)=\mu+\varepsilon$ for some $\varepsilon\vless1$. Thus, it
  suffices to show that the sequence
  $(\log_{i}(\mu+\varepsilon))_{i\in\N}$ is summable. Let $P(y)$ be the
  Taylor series of $\log(1+y)$, namely
  $P(y):=\sum_{n=1}^{\infty}\frac{(-1)^{n}}{n}y^{n}$.  Then
  \begin{align*}
    \log(\mu+\varepsilon) & = \log(\mu) + \log\left(1+\frac{\varepsilon}{\mu}\right) = \log(\mu)+P\left(\frac{\varepsilon}{\mu}\right)\\
                          & = \mu_{1} + \varepsilon_{1}
  \end{align*}
  where $\mu_{1}:=\log(\mu)\in\li$ and
  $\varepsilon_{1}=P\left(\frac{\eps}{\mu}\right)\prec1$.  We define
  inductively $\mu_{0}:=\mu$, $\mu_{i+1}:=\log(\mu_{i})$,
  $\eps_{0}:=\eps$ and
  $\varepsilon_{i+1}:=P\left(\frac{\eps_{i}}{\mu_{i}}\right)$.  By
  construction, $\varepsilon_{i}\prec1$ and
  $\log_{i}(\mu+\eps)=\mu_{i}+\eps_{i}$ for all $i\in\N$. Since
  $(\mu_{i})_{i\in\N}$ is a decreasing sequence of monomials,
  $\sum_{i}\mu_{i}$ exists. To finish the proof it suffices to show that
  $\sum_{i}\eps_{i}$ exists. Let $\m$ be a monomial in the support of
  $\varepsilon_{i+1}=P\left(\frac{\eps_{i}}{\mu_{i}}\right)$.  Then there
  is an integer $m\geq1$ such that
  $\m \in \supp
  \left(\left(\frac{\varepsilon_{i}}{\mu_{i}}\right)^{m}\right) \subseteq
  \frac{1}{\mu_{i}^{m}}\supp(\eps_{i})^{m}$.  By an easy induction it
  follows that
  \[
    \m=\frac{1}{\mu_{0}^{n_{0}}\cdot\dots\cdot\mu_{i}^{n_{i}}}\cdot\o
  \]
  where $n_{0}\geq\ldots\geq n_{i}\geq1$ and $\o$ is a product of
  finitely many elements of $\supp(\eps)$. Note that $\o$ varies in
  the set $\bigcup_{m=1}^{\infty}\supp(\varepsilon)^{m}$, which is
  reverse well-ordered by \prettyref{lem:Neumann}. Therefore, it
  suffices to prove that the family
  $\left(\frac{1}{\mu_{0}^{n_{0}}\cdot\dots\cdot\mu_{i}^{n_{i}}}\suchthat
    i\in\N\right)$ is summable.

  Letting $\delta=\sum_{i\in\N}\frac{1}{\mu_{0}\mu_{1}\cdots\mu_{i}}$, we
  have that $\frac{1}{\mu_{0}^{n_{0}}\cdot\dots\cdot\mu_{i}^{n_{i}}}$ is in
  the support of $\delta^{n_{0}}$. Since $\delta\prec1$, by
  \prettyref{cor:power-series}, $(\delta^{n}\suchthat n\in\N)$ is
  summable, so
  $\left(\frac{1}{\mu_{0}^{n_{0}}\cdot\dots\cdot\mu_{i}^{n_{i}}}\suchthat
    i\in\N\right)$ is summable, hence $(\eps_{i})_{i\in\N}$ is summable,
  as desired.
\end{proof}

\begin{cor}
  \label{cor:basic-pre-substitution}Let $x\in\no^{>\R}$, let
  $\Delta=\left\{ \log_{i}(\omega)\suchthat i\in\N\right\} $ and let
  $c_{0}^{x}:\Delta\to\no$ be the map that sends $\log_{i}(\omega)$ to
  $\log_{i}(x)$. Then $c_{0}^{x}$ is a pre-substitution.
\end{cor}

\subsection{Trees}

We now aim at extending each pre-substitution $c_{0}:\Delta\to\no$ to a
substitution $c:\bracket{\Delta}\to\no$. For this, we introduce the
notion of \emph{tree}, whose aim is to keep track of the monomials
that may appear in the support of $c(x)$ by expressing $c(x)$ in terms
of the values of $c_{0}$. To justify the definition of tree, consider
the following heuristic argument.

Suppose we wish to calculate $c(x)$ for some $x\in\bracket{\Delta}$.
If $\erank(x)=0$, then we simply use the equations
$c(\lambda)=c_{0}(\lambda)=\sum_{t\in\term(c_{0}(\lambda))}t$ and
$c(0)=0$. Now assume $\erank(x)>0$ and write
$x=\sum_{i<\alpha}r_{i}e^{\gamma_{i}}$ in normal form. First, we observe
that we must have $c(x)=\sum_{i<\alpha}c(r_{i}e^{\gamma_{i}})$, so our
problem reduces to calculating $c(r_{i}e^{\gamma_{i}})$ for each
$i<\alpha$. Fix one $\gamma=\gamma_{i}$ and consider the following
equation:
\[
  c(re^{\gamma}) = re^{c(\gamma)} = re^{c(\gamma)^{\bigeq}} \cdot
  \exp(c(\gamma)^{\small}) = re^{c(\gamma)^{\bigeq}} \cdot
  \sum_{n=0}^{\infty}\frac{\left(c(\gamma)^{\small}\right)^{n}}{n!}.
\]

Note that $\erank(\gamma)<\erank(x)$, so we may assume to already have
obtained $c(\gamma)$, and that $c(\gamma)^{\small}$ is presented as a
sum $c(\gamma)^{\downarrow}=\sum_{j\in J}t_{j}$ for some family
$(t_{j})_{j\in J}$ of terms (i.e. elements of $\R^{*}\M$), where
$J=J_{i}$ is some index set. Using \prettyref{prop:n-th power}, we get
\[
  c(re^{\gamma}) = re^{c(\gamma)^{\bigeq}} \cdot
  \sum_{n=0}^{\infty}\sum_{\tau:n\to J}\prod_{m<n}t_{\tau(m)} =
  \sum_{n=0}^{\infty}\sum_{\tau:n\to J}re^{c(\gamma)^{\bigeq}} \cdot
  \prod_{m<n}t_{\tau(m)}.
\]

Note that the right-hand side can be seen as a sum of \emph{terms}.
We use the above equation to present $c(re^{\gamma})$ as a sum of terms
indexed by the set $\{(n,\tau)\suchthat n\in\N,\tau:n\to J\}$.  By
taking the sum over all terms $r_{i}e^{\gamma_{i}}$, we obtain a
presentation of $c(x)$ as a sum of terms indexed by the set
$\{(r_{i}e^{\gamma_{i}},n,\tau)\suchthat i<\alpha,n\in\N,\tau:n\to
J_{i}\}$.

We then proceed inductively and assume that the index sets $J_{i}$ are
themselves constructed in the same way (unless $\erank(\gamma_{i})=0$,
in which case we use the equations
$c(\lambda)=c_{0}(\lambda)=\sum_{t\in\term(c_{0}(\lambda))}t$ and
$c(0)=0$). One can then picture the index
$(r_{i}e^{\gamma_{i}},n,\tau)$ as a tree with root
$r_{i}e^{\gamma_{i}}$ and children $\tau(0),\dots,\tau(n-1)$, as in the
following definition.

\begin{defn}
  \label{def:trees}Fix a pre-substitution $c_{0}:\Delta\to\no$. We
  define inductively the class of trees as follows. A \textbf{tree} is
  an ordered triple $T=\tree{\treeroot(T)}n{\tau}$ where
  $\treeroot(T)\in\bracket{\Delta}\cap\R^{*}\M$ is a term, called the
  \textbf{root} of $T$, and $n,\tau$ are defined as follows:
  \begin{enumerate}
  \item if $\treeroot(T)=\lambda\in\Delta$, then $n=0$ and $\tau$ is a
    term of $c_{0}(\lambda)$, so in this case $T=\tree{\lambda}0t$ with
    $t\in\term(c_{0}(\lambda))$;
  \item if $\treeroot(T)=re^{\gamma}\nin\Delta$, then $n\in\N$ and
    $\tau$ is a function with domain $n=\{0,1,\ldots,n-1\}$ such that
    $\tau(0),\ldots,\tau(n-1)$ are trees, called the \textbf{children}
    of $T$ ($n$ can be zero, in which case $T$ has no children); we
    also require that, for each $i<n$, the root $\treeroot(\tau(i))$
    of $\tau(i)$ is a term of $\gamma=\vell(\treeroot(T))$ (where
    $\vell$ is as in \prettyref{def:ell}).
  \end{enumerate}

  The \textbf{descendants} of $T$ are $T$ itself, its children, and
  the descendants of its children. The \textbf{proper descendants }are
  the descendants different from $T$ itself. The \textbf{leaves} of
  $T$ are the descendants $U$ of $T$ without children (for instance
  the descendants with root in $\Delta$).
\end{defn}

\begin{figure}[b]
  \begin{centering}
    \bigskip{}
    \begin{tikzpicture}[scale=0.8, transform shape]
      \draw (-3.5,1) node(U0) {$\sigma(0) = (\lambda, 0, t_0)$}; \draw
      (-1.5,1) node(UI) {$\vphantom{\tau(2) = }\ldots$}; \draw (1.2,1)
      node(UN) {$\sigma(m-1) = (\lambda, 0, t_{m-1})$}; \draw (-1.2, 0)
      node(U) {$\tau(0) = (se^\lambda, m, \sigma)$}; \draw (1,-1)
      node(T) {$T = (re^\gamma, n, \tau)$}; \draw (2.8,0) node(T1)
      {$\tau(1) = \ldots$}; \draw (6,0) node(TI)
      {$\vphantom{\tau(2) = }\ldots$}; \draw (5,1) node(E1)
      {$\vphantom{\tau(2) = }\ldots$}; \draw (7.5,1) node(E2)
      {$\vphantom{\tau(2) = }\ldots$}; \draw (T) -- (U); \draw (T) --
      (T1); \draw (T) -- (TI) -- (E2); \draw (U) -- (U0); \draw (U) --
      (UI); \draw (U) -- (UN); \draw (TI) -- (E1);
    \end{tikzpicture}
  \end{centering}
  \caption{An example of tree with root $R(T)=re^{\gamma}$, where
    $se^{\lambda}$ is a term of $\gamma$, $\lambda\in\Delta$,
    $t_{0},\dots,t_{m-1}$ are terms of $c_{0}(\lambda)$, and the
    contribution $\protect\ctree(T)$ of $T$ is
    \[
      \scriptscriptstyle \mkern-135mu\protect\begin{aligned}{\textstyle \protect\ctree(T)} & ={\textstyle re^{c(\gamma)^{\protect\bigeq}}\frac{1}{n!}\protect\ctree(\tau(0))\protect\ctree(\tau(1))\ldots}\protect\\
        & ={\textstyle re^{c(\gamma)^{\protect\bigeq}}\frac{1}{n!}se^{c_{0}(\lambda)^{\protect\bigeq}}\frac{1}{m!}\protect\ctree(\sigma(0))\ldots\protect\ctree(\sigma(m-1))\protect\ctree(\tau(1))\ldots}\protect\\
        & ={\textstyle
          re^{c(\gamma)^{\protect\bigeq}}\frac{1}{n!}se^{c_{0}(\lambda)^{\protect\bigeq}}\frac{1}{m!}t_{0}\dots
          t_{m-1}\protect\ctree(\tau(1))\ldots} \protect\end{aligned}
    \]
  }
\end{figure}

Note that by induction on $\erank$, the above definition of tree is
well founded.

\begin{defn}
  Let $T=\tree{\treeroot(T)}n{\tau}$ be a tree. We define
  $\size(T)\in\N$ as the number of descendants of $T$, namely:
  \begin{enumerate}
  \item $\size(T):=1$ if $T$ has no children, namely $n=0$;
  \item $\size(T):=1+\sum_{i<n}\size(\tau(i))$ otherwise.
  \end{enumerate}
\end{defn}

\subsection{\label{sub:extension}Extending pre-substitutions to substitutions}

Fix a pre-substitution $c_{0}:\Delta\to\no$. We shall now define a
substitution $c:\bracket{\Delta}\to\no$ extending the given
pre-substitution $c_{0}$. To this aim, we shall define simultaneously
by induction on $\alpha\in\on$ the following objects:
\begin{itemize}
\item the set of \textbf{admissible trees} $\admtree(x)$ of each
  $x\in\Delta_{\alpha}$ (which are trees in the sense of
  \prettyref{def:trees} with root $\treeroot(T)\in\term(x)$ and some
  further requirements);
\item the \textbf{contribution} $\ctree(T)\in\R^{*}\M$ of each
  $T\in\admtree(x)$;
\item the extension $c:\Delta_{\alpha}\to\no$ (which is obtained by
  summing the contributions of the admissible trees in $\admtree(x)$,
  that is $c(x)=\sum_{T\in\admtree(x)}\ctree(T)$).
\end{itemize}
The main difficulty will be in proving that each family
$(\ctree(T)\suchthat T\in\admtree(x))$ is summable, which is needed to
show that $c(x)=\sum_{T\in\admtree(x)}\ctree(T)$ is well defined
(\prettyref{lem:summability}).

\begin{defn}
  \label{def:inductive}Let $\alpha\in\on$ be given. Let $I(\alpha)$ be
  the hypothesis
  \begin{center}
    \emph{For all $x\in\Delta_{\alpha}$,
      $(\ctree(T)\suchthat T\in\admtree(x))$ is summable}
  \end{center}
  where $\admtree(x)$ and $\ctree(T)$ for $T\in\admtree(x)$ are
  inductively defined as in \prettyref{def:everything} (assuming
  $I(\beta)$ for $\beta<\alpha$).
\end{defn}

\begin{defn}
  \label{def:everything}First, we let $\admtree(0):=\emptyset$, and
  for $\lambda\in\Delta$, we define:
  \begin{enumerate}
  \item
    $\admtree(\lambda):=\left\{ \tree{\lambda}0t\suchthat
      t\in\term(c_{0}(\lambda))\right\} $ (namely every tree with root
    in $\Delta$ is admissible);
  \item $\ctree(\tree{\lambda}0t):=t$ (the value of a tree with root
    in $\Delta$ is its third component);
  \item $c(\lambda):=\sum_{T\in\admtree(\lambda)}\ctree(T)$.
  \end{enumerate}
  This defines $\admtree(x)$, $\ctree(T)$ and $c(x)$ for all
  $x\in\Delta_{1}$ and $T\in\admtree(x)$.

  Now let $\alpha>1$ and assume $I(\beta)$ for all $\beta<\alpha$.
  For general $x\in\Delta_{\alpha}$ we define:
  \begin{enumerate}[resume]
  \item $\admtree(x):=\bigcup_{t\in\term(x)}\admtree(t)$;
  \item $\admtreeinf(x):=\{T\in\admtree(x)\suchthat\ctree(T)\prec1\}$.
  \end{enumerate}
  When $x=t=re^{\gamma}$ is a term in
  $\Delta_{\alpha}\setminus\Delta_{1}$, let $\beta<\alpha$ be such that
  $re^{\gamma}\in\Delta_{\beta+1}\setminus\Delta_{\beta}$.  We observe
  that $\gamma\in\Delta_{\beta}$, and we define:
  \begin{enumerate}[resume]
  \item
    $\admtree(re^{\gamma}):=\left\{
      \tree{re^{\gamma}}n{\tau}\,:\,n\in\N,\tau:n\to\admtreeinf(\gamma)\right\}
    $;
  \item for $T=\tree{re^{\gamma}}n{\tau}\in\admtree(re^{\gamma})$,
    \[
      \ctree(T):=re^{c(\gamma)^{\bigeq}}\cdot\frac{1}{n!}\prod_{i<n}\ctree(\tau(i)).
    \]
  \end{enumerate}
  Finally, for any $x\in\Delta_{\alpha}$, if $I(\alpha)$ holds, we
  define:
  \begin{enumerate}[resume]
  \item $c(x):=\sum_{T\in\admtree(x)}\ctree(T)$.
  \end{enumerate}
\end{defn}

\begin{rem}
  It is important to note that points (1)-(7) only require $I(\beta)$
  for $\beta<\alpha$, while (8) does require $I(\alpha)$. The
  inductive hypothesis $I(\alpha)$ itself is defined by induction on
  $\alpha$!  We also remark that $I(0),I(1)$ are trivially true.
\end{rem}

\begin{defn}
  \label{def:substitution}Assuming that $I(\alpha)$ holds for every
  $\alpha\in\on$, we define $c:\bracket{\Delta}\to\no$ as the union of
  the functions $c:\Delta_{\alpha}\to\no$ for $\alpha\in\on$.
\end{defn}

\begin{rem}
  The present notion of tree should be compared with the similar
  notion of labeled trees in \cite{Schmeling2001}. In this comparison,
  the admissible trees play the same role as the well-labeled trees.
\end{rem}

We shall now prove that $c:\bracket{\Delta}\to\no$ is well defined and
that it is the unique substitution on $\bracket{\Delta}$ extending
$c_{0}$. The most technical and difficult part will be proving that if
$I(\alpha)$ holds, then $(\ctree(T)\suchthat T\in\admtree(x))$ is
summable for all $x\in\Delta_{\alpha+1}\setminus\Delta_{\alpha}$
(\prettyref{lem:summability}). As anticipated, the proof of this fact
is postponed to \prettyref{sec:induction}.

First, we check that $c$ is indeed an extension of $c_{0}$, and that it
fixes $\R$.

\begin{prop}
  \label{prop:clambda}For all $\lambda\in\Delta$,
  $(\ctree(T)\suchthat T\in\admtree(\lambda))$ is summable and
  \[
    c(\lambda)=\sum_{T\in\admtree(\lambda)}\ctree(T)=c_{0}(\lambda).
  \]
  In particular, $I(0)$ and $I(1)$ hold, and $c$ extends $c_{0}$.
\end{prop}

\begin{proof}
  For any $\lambda\in\Delta$ and $T\in\admtree(\lambda)$, we have
  $T=\tree{\lambda}0t$ for some $t\in\term(\lambda)$ and
  $\ctree(T)=t$.  Moreover,
  $(\ctree(T)\suchthat T\in\admtree(\lambda))$ coincides with
  $(t\suchthat t\in\term(c_{0}(\lambda)))$, hence it is summable and by
  definition
  \[
    c(\lambda) = \sum_{T\in\admtree(\lambda)} \bar{c}(T) = \sum\limits
    _{t\in\term(c_{0}(\lambda))}t = c_{0}(\lambda).
  \]
\end{proof}

\begin{prop}
  \label{prop:cR}If $r\in\R$, then
  $(\ctree(T)\suchthat T\in\admtree(r))$ is summable and $c(r)=r$.
\end{prop}

\begin{proof}
  Note first that $I(1)$ holds by \ref{prop:clambda}, and that
  $\R\subseteq\Delta_{2}$, so $\admtree(r)$ is well defined for each
  $r\in\R$. Now observe that $\admtree(0)=\emptyset$, so
  $c(0)=\sum_{T\in\admtree(0)}\ctree(T)$ is an empty sum (equal to
  zero) and we get $c(0)=0$. For $r\neq0$ the only admissible tree
  $T\in\admtree(r)$ is given by $T=\tree{re^{0}}0{\emptyset}$.  By
  definition, $\ctree(T)=re^{c(0)}=re^{0}=r$, hence $c(r)=r$.
\end{proof}

We now prove that assuming $I(\alpha)$, the extension
$c:\Delta_{\alpha}\to\no$ preserves $\log$ and infinite sums. For
$\alpha\geq3$, since $\Delta_{\alpha}$ is a field of transseries, this
says that $c:\Delta_{\alpha}\to\no$ is a substitution. Note that in the
following statement the hypothesis is $I(\alpha)$, but the conclusion
is about terms in $\Delta_{\alpha+1}$.

\begin{prop}
  \label{prop:cexp}Assume $I(\alpha)$. Let
  $re^{\gamma}\in\Delta_{\alpha+1}$ be a term. Then
  $(\ctree(T)\suchthat T\in\admtree(re^{\gamma}))$ is summable, so
  $c(re^{\gamma})=\sum_{T\in\admtree(re^{\gamma})}\ctree(T)$ is well
  defined and
  \[
    c(re^{\gamma})=re^{c(\gamma)}.
  \]
\end{prop}

\begin{proof}
  The result is clear if $re^{\gamma}\in\Delta_{1}=\Delta\cup\{0\}$, for
  in that case $c$ coincides with $c_{0}$ by \prettyref{prop:clambda},
  so we can assume $re^{\gamma}\nin\Delta_{1}$. Then
  $\erank(\gamma)<\erank(re^{\gamma})$, and by the inductive
  hypothesis,
  $c(\gamma)=\sum\limits _{T'\in\admtree(\gamma)}\overline{c}(T')$.  By
  definition of $\admtreeinf(\gamma)$ we have
  \[
    c(\gamma)^{\small}=\sum\limits
    _{T'\in\admtreeinf(\gamma)}\overline{c}(T').
  \]
  Unraveling the definitions we have:
  \begin{align*}
    re^{c(\gamma)} & = re^{c(\gamma)^{\bigeq}}e^{c(\gamma)^{\small}}\\
                  & = \sum_{n\in\N}\,re^{c(\gamma)^{\bigeq}}\frac{1}{n!}\left(c(\gamma)^{\small}\right)^{n}\\
                  & = \sum_{n\in\N}\,re^{c(\gamma)^{\bigeq}}\frac{1}{n!}\left(\sum_{T'\in\admtreeinf(\gamma)}\bar{c}(T')\right)^{n}\\
                  & = \sum_{n\in\N}\,re^{c(\gamma)^{\bigeq}}\frac{1}{n!}\sum_{\tau:n\to\admtreeinf(\gamma)}\,\prod_{i<n}\overline{c}(\tau(i))\\
                  & = \sum_{T\in\admtree(re^{\gamma})}\,\bar{c}(T)=c(re^{\gamma}),
  \end{align*}
  where in the fourth line we used \prettyref{prop:n-th power} (which
  also shows the summability of the relevant sequences) and in the
  fifth line we used the definition of $\ctree(T)$ for
  $T\in\admtree(re^{\gamma})$.
\end{proof}

\begin{prop}
  \label{prop:csum}Assume $I(\alpha)$. Let
  $\sum_{i<\beta}r_{i}e^{\gamma_{i}}\in\Delta_{\alpha}$.  Then
  \[
    c\left(\sum_{i<\beta}r_{i}e^{\gamma_{i}}\right)=\sum_{i<\beta}r_{i}e^{c(\gamma_{i})}.
  \]
\end{prop}

\begin{proof}
  It follows at once from \prettyref{prop:cexp} and the equality
  $\admtree(\sum_{i<\beta}r_{i}e^{\gamma_{i}})=\bigcup_{i<\beta}\admtree(r_{i}e^{\gamma_{i}})$.
\end{proof}

\begin{cor}
  \label{cor:c-alpha-3-substitution}Assume $I(\alpha)$, with
  $\alpha\geq3$.  Then $c:\Delta_{\alpha}\to\no$ is a substitution.
\end{cor}

\begin{proof}
  Since $\alpha\geq3$, $\Delta_{\alpha}$ is a transserial subfield of
  $\no$ by \prettyref{prop:unionDelta-exp-log-closed}. By
  \prettyref{prop:cR}, $c$ fixes $\R$, and by \prettyref{prop:csum},
  it is strongly additive.  Moreover, $c$ preserves $\log$. Indeed,
  let $x=re^{\gamma}(1+\varepsilon)\in\Delta_{\alpha}$, where
  $r\in\R$, $\gamma\in\J$ and $\varepsilon\prec1$. We have
  \[
    c(\log(x)) = c\left(\gamma + \log(r) +
      \sum_{i=1}^{\infty}(-1)^{n}\frac{\varepsilon^{n}}{n}\right) =
    c(\gamma)+\log(r) +
    \sum_{i=1}^{\infty}(-1)^{n}\frac{c(\varepsilon)^{n}}{n}.
  \]
  By \prettyref{prop:cexp}, $c(\gamma)=\log(c(e^{\gamma}))$, so the
  right hand side is $\log(c(x))$, as desired.
\end{proof}

\begin{cor}
  \label{cor:csubstitution}Assume $I(\alpha)$ for all $\alpha\in\on$.
  Then $c:\bracket{\Delta}\to\no$ is a substitution extending $c_{0}$.
\end{cor}

Finally, we need to prove inductively that $I(\alpha)$ holds for all
$\alpha\in\on$. The main difficulty is in proving the successor stage,
namely that $I(\alpha)$ implies $I(\alpha+1)$. This is contained in
the following lemma, the proof of which is postponed to
\prettyref{sec:induction}.

\begin{lem}[Summability]
  \label{lem:summability}Assume $I(\alpha)$. Then
  $(\ctree(T)\suchthat T\in\admtree(x))$ is summable for all
  $x\in\Delta_{\alpha+1}\setminus\Delta_{\alpha}$.  In particular,
  $I(\alpha)$ implies $I(\alpha+1)$.
\end{lem}

\begin{proof}
  Postponed to \prettyref{sec:induction}.
\end{proof}

\begin{thm}
  \label{thm:pre-comp-extends}Any pre-substitution
  $c_{0}:\Delta\to\no$ extends uniquely to a substitution
  $c:\bracket{\Delta}\to\no$.
\end{thm}

\begin{proof}
  Fix a pre-substitution $c_{0}:\Delta\to\no$. By
  \prettyref{prop:clambda}, $I(0)$ and $I(1)$ hold. It is also clear
  by the definition of $I(\alpha)$ that whenever $\alpha$ is a limit
  ordinal, $I(\alpha)$ is implied by, and in fact equivalent to,
  $\bigwedge_{\beta<\alpha}I(\beta)$.  Moreover, by
  \prettyref{lem:summability}, $I(\alpha)$ implies $I(\alpha+1)$.
  Therefore, $I(\alpha)$ holds for all $\alpha\in\on$. By
  \prettyref{cor:csubstitution}, $c$ is a substitution extending
  $c_{0}$. The uniqueness follows by an easy induction on $\erank$.
\end{proof}

\begin{cor}
  \label{cor:substitution-omega-series}Given $x\in\no^{>0}$, there is a
  unique substitution $c^{x}:\bracket{\omega}\to\no$ sending $\omega$
  to $x$.
\end{cor}

\begin{proof}
  Let $\Delta=\left\{ \log_{i}(\omega)\suchthat i\in\N\right\} $ and
  let $c_{0}^{x}:\Delta\to\no$ be the map that sends
  $\log_{i}(\omega)$ to $\log_{i}(x)$. Then $c_{0}^{x}$ is a
  pre-substitution by \prettyref{cor:basic-pre-substitution}.  By
  \prettyref{thm:pre-comp-extends}, there is a unique substitution
  $c^{x}:\bracket{\Delta}=\bracket{\omega}\to\no$ extending $c_{0}^{x}$.
\end{proof}

\section{Composition}

We prove that omega-series can be composed in a meaningful
way. Intuitively, for $f,g\in\bracket{\omega}$, with $g>\R$,
$f\circ g$ is the result of substituting $g$ for $\omega$ in $f$. For
instance, we will have
\[
  \left(\sum_{i\in\N}\log_{i}(\omega)\right) \circ
  \left(\sum_{i\in\N}\log_{i}(\omega)\right) =
  \sum_{i\in\N}\log_{i}\left(\sum_{j\in\N}\log_{j}(\omega)\right).
\]
Note that the right-hand side exists in $\no$ by the results in
\prettyref{sec:Substitutions} and it is in fact an element of
$\bracket{\omega}$.

\begin{defn}
  Let $T\subseteq\no$ be a transserial subfield containing $\omega$.
  A \textbf{composition} on $T$ is a function
  $\circ:T\times\no^{>\R}\to\no$ which satisfies the following axioms:
  \begin{enumerate}
  \item for all $x\in\no^{>\R}$, the map $f\mapsto f\circ x$ is a
    substitution, namely:
    \begin{enumerate}
    \item for any summable $(f_{i})_{i\in I}$ in $T$, the family
      $(f_{i}\circ x)_{i\in I}$ is summable and
      \[
        \left(\sum_{i\in I}f_{i}\right)\circ x=\sum_{i\in I}(f_{i}\circ
        x);
      \]
    \item $r\circ x=r$ for all $r\in\R$;
    \item $\log(f)\circ x=\log(f\circ x)$ for all $f\in T$;
    \end{enumerate}
  \item $T$ is closed under composition: for all $f\in T$,
    $g\in T^{>\R}$ we have $f\circ g\in T$;
  \item associativity: $(f\circ g)\circ x=f\circ(g\circ x)$ for all
    $f\in T$, $g\in T^{>\R}$, $x\in\no^{>\R}$;
  \item $\omega$ is the identity: for all $x\in\no^{>\R}$ and $f\in T$
    we have $\omega\circ x=x$, $f\circ\omega=f$.
  \end{enumerate}
\end{defn}

The axioms are modeled on the usual composition of real valued
functions, where we interpret $\omega$ as the identity function. The
restriction on the second argument to be positive infinite is
necessary for a composition to exist; for instance we cannot hope to
define $\sum_{n\in\N}\omega^{-n}\circ(1/2)$ in any reasonable way, as
the axioms imply that the result should be $\sum_{n\in\N}2^{n}$. Recall
that by \prettyref{prop:substitution-is-hom}, for all $x\in\no^{>\N}$,
the map $f\circ x$ is increasing and it preserves the dominance
relation $\vleq$.

When $T\subseteq\bracket{\omega}$, the list of axioms can be
shortened. More precisely, we have:

\begin{prop}
  \label{prop:uniqueness-composition}If $T$ is a transserial field
  included in $\bracket{\omega}$, there is at most one function
  $\circ:T\times\no^{>\R}\to\no$ satisfying the following conditions:
  \begin{enumerate}
  \item for all $x\in\no^{>\R}$, the map $f\mapsto f\circ x$ is a
    substitution;
  \item for all $x\in\no^{>\R}$, $\omega\circ x=x$.
  \end{enumerate}
  If any such function $\circ$ exists, it satisfies $f\circ\omega=f$
  for any $f\in T$. If moreover $T$ is closed under $\circ$, then
  $\circ$ is associative, so it is a composition.
\end{prop}

\begin{proof}
  Suppose that $\circ$ is a function satisfying the above properties.
  Let
  $\Delta=\left\{ \log_{i}(\omega)\suchthat i\in\N\right\} \subseteq
  T$, and fix some $x\in\no^{>\R}$. We claim that the values of the
  substitution $f\mapsto f\circ x$ for $f\in\Delta$ are uniquely
  determined by the requirement $\omega\circ x=x$. We shall prove this
  by induction on $\erank(f)$; at the same time, we will also verify
  associativity when $T$ is closed under $\circ$.

  Note first that $\log_{i}(\omega)\circ x=\log_{i}(x)$ by definition of
  substitution. Moreover,
  \[
    \log_{i}(\omega)\circ(g\circ x)=\log_{i}(g\circ x)=\log_{i}(g)\circ
    x=(\log_{i}(\omega)\circ g)\circ x
  \]
  for any $g\in T^{>\N}$, and also
  $\log_{i}(\omega)\circ\omega=\log_{i}(\omega)$.  It now follows by
  induction on $\erank(f)$ that the value of $f\circ x$ is also
  uniquely determined, $f\circ\omega=f$, and if $T$ is closed under
  $\circ$, then $f\circ(g\circ x)=(f\circ g)\circ x$ for any
  $g\in T^{>\R}$. Indeed, if $f=\sum_{i<\alpha}r_{i}e^{\gamma_{i}}$, where
  $\erank(f)>0$, then we must have
  \[
    f\circ x=\sum_{i<\alpha}r_{i}e^{\gamma_{i}\circ x}
  \]
  where $\erank(\gamma_{i})<\erank(f)$. The value of $f\circ x$ is then
  uniquely determined by the values $\gamma_{i}\circ x$, which are
  themselves uniquely determined by inductive hypothesis, and clearly
  $f\circ\omega=f$ as again by induction
  $\gamma_{i}\circ\omega=\gamma_{i}$.  Moreover, if $T$ is closed under
  $\circ$, then
  \[
    f\circ(g\circ x)=\sum_{i<\alpha}r_{i}e^{\gamma_{i}\circ(g\circ
      x)}=\sum_{i<\alpha}r_{i}e^{(\gamma_{i}\circ g)\circ
      x}=\left(\sum_{i<\alpha}r_{i}e^{\gamma_{i}\circ g}\right)\circ
    x=(f\circ g)\circ x.
  \]
  Therefore, $\circ$ is unique, $f\circ\omega=f$ for any $f\in T$, and
  if $T$ is closed under $\circ$, then it is associative, so it is a
  composition.
\end{proof}

\begin{thm}
  \label{thm:composition-omega-series}There is a unique composition
  $\circ:\bracket{\omega}\times\no^{>\R}\to\no$.
\end{thm}

\begin{proof}
  Let $\Delta=\left\{ \log_{i}(\omega)\suchthat i\in\N\right\} $. Fix
  $x\in\no^{>\R}$ and $f\in\bracket{\omega}$. By
  \prettyref{cor:substitution-omega-series}, there exists a unique
  substitution $c^{x}$ on $\bracket{\Delta}=\bracket{\omega}$ such that
  $c^{x}(\log_{i}(\omega))=\log_{i}(x)$ for all $i\in\N$.  We then define
  $f\circ x:=c^{x}(f)$. Clearly, this function is the unique one
  satisfying the hypothesis of
  \prettyref{prop:uniqueness-composition}.  One can easily verify by
  induction on $\erank$ that $\bracket{\omega}$ is closed under
  $\circ$, so it is a composition.
\end{proof}

\section{Taylor expansions}

In this section, let $\circ$ be the unique composition on
$\bracket{\omega}$.  We shall now prove that for every
$f\in\bracket{\omega}$, the function $x\mapsto f\circ x$ is surreal
analytic in the sense of \prettyref{def:surreal-analytic}.  Moreover,
the coefficients will coincide with the iterated derivatives of $f$
divided by $n!$, when using the unique surreal derivation on
$\bracket{\omega}$.

\subsection{\label{subsec:Transserial-derivations}Transserial
  derivations}

Recall the notion of derivation from
\cite{Schmeling2001,Berarducci2015}.

\begin{defn}
  Given a field $T$, we recall that a map $\partial:T\to T$ is a
  \textbf{derivation} if it is additive
  ($\partial(x+y)=\partial x+\partial y$) and satisfies the Leibniz
  rule ($\partial(xy)=x\cdot\partial y+\partial x\cdot y)$.  If $T$ is
  a field of transseries we say that $\partial:T\to T$ is a
  \textbf{transserial derivation} if it is a derivation satisfying the
  following additional properties:
  \begin{enumerate}
  \item $\partial$ is strongly additive;
  \item $\partial e^{x}=e^{x}\cdot\partial x$;
  \item $\partial\omega=1$;
  \item $\partial r=0$ if $r\in\R$.
  \end{enumerate}
  As in \cite{Berarducci2015}, we call \textbf{surreal derivation} a
  transserial derivation with $\ker\partial=\R$.
\end{defn}

In \cite{Berarducci2015}, the authors proved that there exist surreal
derivations on $\no$, and in fact several of them. However, just like
we proved that there is a unique composition on $\bracket{\omega}$, we
can easily verify that there exists a unique transserial derivation on
$\bracket{\omega}$.
\begin{prop}
  \label{prop:standard-derivation}The field of omega-series admits a
  unique transserial derivation
  $\partial:\bracket{\omega}\to\bracket{\omega}$, which is in fact a
  surreal derivation.
\end{prop}

\begin{proof}
  Suppose first that there exists a transserial derivation
  $\partial:\bracket{\omega}\to\bracket{\omega}$.  Since
  $\partial\omega=1$, an easy induction on $\erank$ shows that in fact
  the values of $\partial$ are uniquely determined, and that
  $\ker(\partial)=\R$. Therefore, if there is one such derivation, it
  is unique, and it is a surreal derivation.

  For the existence, let $\partial$ be any surreal derivation, which
  exists by the results of \cite{Berarducci2015}. By the same argument
  as above, since $\partial\omega=1\in\bracket{\omega}$, an easy
  induction on $\erank$ shows that
  $\partial(\bracket{\omega})\subseteq\bracket{\omega}$.  Therefore,
  the restriction of $\partial$ to $\bracket{\omega}$ is the unique
  transserial derivation on $\bracket{\omega}$.
\end{proof}

\begin{rem}
  Unlike the subfield $\R((\omega))^{LE}$, but like
  $\R((\omega))^{EL}$, the field of omega-series $\bracket{\omega}$ is
  not closed under anti-derivatives. For instance, it contains no
  integral for the monomial $\exp(-\sum_{n\in\N}\log_{n}(\omega))$.
\end{rem}

\subsection{A Taylor theorem}

From now on, let $\partial:\bracket{\omega}\to\bracket{\omega}$ be the
unique transserial derivation on $\bracket{\omega}$. Recall that for
any $x\prec1$ we have $\exp(x)=\sum_{n\in\N}\frac{x^{n}}{n!}$.  When
$x\succ1$, the equality does not hold, as the right hand side clearly
does not exist. However, we can still approximate $\exp(x)$ with
Taylor polynomials. In particular we have the following:

\begin{prop}
  Given $x\in\no$, there are $A\in\no$ and
  $\varepsilon_{0}\in\no^{>0}$ (depending on $x$) such that, for every
  $\varepsilon\in\no$ smaller in modulus than $\varepsilon_{0}$, we
  have
  \[
    \exp(x+\eps)=\exp(x)+\exp'(x)\eps+\mathcal{O}(A\eps^{2})
  \]
  where $\exp'(x):=\exp(x)$ and $\mathcal{O}(A\varepsilon^{2})$ is a
  surreal number $\vleq A\varepsilon^{2}$. Similarly, we can write
  \[
    \log(x+\eps)=\log(x)+\log'(x)\eps+\mathcal{O}(A\eps^{2})
  \]
  where $\log'(x):=\frac{1}{x}$.
\end{prop}

\begin{proof}
  Immediate from the fact that $\no$ is an elementary extension of
  $\R_{\exp}$.
\end{proof}

The next theorem extends the above remark to a much larger class of
functions.

\begin{thm}
  \label{thm:taylor}Given $f\in\bracket{\omega}$ and $x\in\no^{>\R}$,
  there are $A\in\no$ and $\eps_{0}\in\text{\ensuremath{\no}}^{>0}$
  (both depending on $f$ and $x$) such that, for every $\eps\in\no$
  smaller in modulus than $\eps_{0}$, we have
  \[
    f\circ(x+\eps)=f\circ x+(\partial f\circ
    x)\cdot\eps+\mathcal{O}(A\eps^{2}),
  \]
  where $\mathcal{O}(A\varepsilon^{2})$ is a surreal number
  $\vleq A\varepsilon^{2}$.
\end{thm}

\begin{proof}
  We reason by induction on the ordinal $\erank(f)$, where
  $\Delta=\{\log_{i}(\omega)\,:\,i\in\N\}$.

  Case 1. The theorem is clear if $f\in\R$ or $f=\omega$, as in this
  case $f\circ(x+\eps)=f\circ x+(\partial(f)\circ x)\eps$ for every
  $\eps$ and we can take $A=0$.

  Case 2. Now consider the case when $f=\log(g)$ where $g>0$, and
  assume that conclusion holds for $g$. Then there are $B\in\no$ and
  $\eps_{1}\in\no^{>0}$ (depending on $g,x$) such that
  \[
    g\circ(x+\eps)=g\circ x+(\partial(g)\circ
    x)\eps+\mathcal{O}(B\eps^{2})
  \]
  whenever $|\eps|\leq|\eps_{1}|$. Taking the $\log$ of both sides, and
  recalling that
  $\log(g\circ(x+\eps))=\log(g)\circ(x+\eps)=f\circ(x+\eps)$, we
  obtain
  \begin{align*}
    f\circ(x+\eps) & =\log(g\circ x+(\partial(g)\circ x)\eps+\mathcal{O}(B\eps^{2})).
  \end{align*}
  Using the second order Taylor expansion of $\log$ at $g\circ x$, we
  can find $A\in\no$, depending on $g$ and $x$, such that, for all
  sufficiently small $\eps$,
  \begin{align*}
    \log(g\circ x+(\partial(g)\circ x)\eps+\mathcal{O}(B\eps^{2})) & =\log(g\circ x)+\frac{1}{g\circ x}(\partial(g)\circ x)\eps+\mathcal{O}(A\eps^{2})\\
                                                                  & =\log(g)\circ x+\left(\frac{\partial(g)}{g}\circ x\right)\eps+\mathcal{O}(A\eps^{2})\\
                                                                  & =f\circ x+\left(\partial(f)\circ x\right)\eps+\mathcal{O}(A\eps^{2}).
  \end{align*}
  Combining the equations we obtain
  $f\circ(x+\eps)=f\circ x+\left(\partial(f)\circ
    x\right)\eps+\mathcal{O}(A\eps^{2})$, as desired.

  Case 3. When $f=\log_{n}(\omega)$ for some $n\in\N$, the desired
  result follows from the previous cases by induction on $n$. We have
  thus established the conclusion when $\erank(f)=0$, namely
  $f\in\Delta_{1}=\Delta\cup\{0\}$.

  Case 4. Consider now the case when $f=\exp(g)$ and assume that the
  conclusion holds for $g$. We can then proceed as in case 2 using the
  second order Taylor expansion of $\exp$ at $g\circ x$.

  Case 5. Consider the case when $f=\sum_{i\in I}f_{i}$ and assume by
  induction that the result holds for each $f_{i}$. By definition
  $f\circ(x+\eps)=\sum_{i\in I}(f_{i}\circ(x+\eps))$. By induction there
  are $\eps_{i,x}\in\no^{>0}$ and $A_{i,x}\in\no$ such that
  \[
    f_{i}\circ(x+\eps)=f_{i}\circ x+(\partial(f_{i})\circ
    x)\eps+\mathcal{O}(A_{i,x}\eps^{2})
  \]
  for all $\eps<\eps_{i,x}$. Now let $\eps_{0}\in\no^{>0}$ be smaller
  than $\eps_{i,x}$ for every $i\in I$ and let $A\succeq A_{i,x}$ for
  every $i\in I$. Then for every $\eps$ smaller in modulus than
  $\eps_{0}$ we have
  $f\circ(x+\eps)=f\circ x+(\partial(f)\circ
  x)\cdot\eps+\mathcal{O}(A\eps^{2})$, as desired.

  Finally, observe that the above cases suffices to establish
  inductively the theorem for every $f\in\bracket{\omega}$.
\end{proof}

\begin{cor}
  \label{cor:limit-def}For every $f\in\bracket{\omega}$ and every
  $x\in\no^{>\R}$ we have
  \[
    \partial f\circ x=\lim_{\eps\to0}\frac{f\circ(x+\eps)-f\circ
      x}{\eps}
  \]
  In particular, taking $x=\omega$, we obtain
  $\partial
  f=\lim_{\eps\to0}\frac{f\circ(\omega+\eps)-f\circ\omega}{\eps}$, so
  the derivative is definable in terms of the composition.
\end{cor}

\begin{cor}
  \label{cor:chain-rule}The unique composition on $\bracket{\omega}$
  satisfies $\partial(f\circ g)=(\partial f\circ g)\cdot\partial g$.
\end{cor}

\begin{proof}
  Thanks to \prettyref{cor:limit-def}, it suffices to show that that
  for all sufficiently small $\eps$ we have
  \[
    (f\circ g)\circ(x+\eps)=(f\circ g)\circ x+\left((\partial f\circ
      g)\cdot\partial g\right)\eps+\mathcal{O}(A\eps^{2})
  \]
  where $A\in\no$ depends on $f$,$g$, $x$ but not on $\eps$. Applying
  \prettyref{thm:taylor} first to $g$ and then to $f$, there are
  $C,D\in\no$, not depending on $\eps$, such that
  \begin{eqnarray*}
    (f\circ g)\circ(x+\eps) & = & f\circ(g\circ(x+\eps))\\
                            & = & f\circ(g\circ x+(\partial g\circ x)\eps+\mathcal{O}(C\cdot\eps^{2}))\\
                            & = & f\circ(g\circ x)+(\partial f\circ(g\circ x))\cdot(\partial g\circ x)\eps+\mathcal{O}(D\cdot\eps^{2}),
  \end{eqnarray*}
  and we conclude by noting that
  $(\partial f\circ(g\circ x))\cdot(\partial g\circ x)=((\partial
  f\circ g)\cdot\partial g)\circ x$.
\end{proof}

\subsection{Surreal analyticity}

We now extend in the obvious way the notion of surreal analyticity of
\prettyref{def:surreal-analytic} to the numbers in $\bracket{\omega}$.

\begin{defn}
  \label{def:analytic}Let $f\in\bracket{\omega}$. We say that $f$ is
  \textbf{surreal analytic} \textbf{at $x\in\no^{>\R}$} if the function
  $y\mapsto f\circ y$ is surreal analytic in a neighborhood of $x$ is
  the sense of \prettyref{def:surreal-analytic}. We say that $f$ is
  \textbf{surreal analytic} if $y\mapsto f\circ y$ is surreal analytic
  at every $x\in\no^{>\R}$.
\end{defn}

For instance, $\exp(\omega)$ and $\log(\omega)$ are surreal analytic.

\begin{prop}
  \label{prop:exp-analytic}Let $x\in\no^{>\R}$. Then for every
  $\eps\prec1$ we have
  $\exp(x+\eps)=\sum_{i=0}^{\infty}\frac{e^{x}}{i!}\eps^{i}$.  In
  particular, $\exp(\omega)$ is surreal analytic.
\end{prop}

\begin{proof}
  Indeed,
  $\exp(x+\varepsilon)=\exp(x)\cdot\exp(\varepsilon)=\exp(x)\cdot\sum_{i=0}^{\infty}\frac{\varepsilon^{i}}{i!}$.
\end{proof}

\begin{prop}
  \label{prop:log-analytic}Let $x\in\no^{>\R}$. Then for every
  $\eps\prec x$ we have
  $\log(x+\eps)=\log(x)+\sum_{i=1}^{\infty}\frac{(-1)^{i+1}}{ix^{i}}\eps^{i}$.
  In particular, $\log(\omega)$ is surreal analytic.
\end{prop}

\begin{proof}
  It suffices to write $x+\eps=x\left(1+\frac{\eps}{x}\right)$, so
  that $\delta:=\frac{\eps}{x}\prec1$, and recall that
  \[
    \log(x+\varepsilon) = \log\left(x\left(1 +
        \frac{\varepsilon}{x}\right)\right) = \log(x) + \log(1+\delta)
    = \log(x) + \sum_{i=1}^{\infty}\frac{(-1)^{i+1}}{i}\delta^{i}.
  \]
\end{proof}

Moreover, surreal analyticity is preserved under compositions.

\begin{lem}
  \label{lem:composition-analytic}If $g\in\bracket{\omega}$ is surreal
  analytic at $x\in\no^{>\R}$ and $f\in\bracket{\omega}$ is surreal
  analytic at $y:=g\circ x$, then $f\circ g$ is surreal analytic at
  $x$.
\end{lem}

\begin{proof}
  Fix $f,g,x,y$ as in the hypothesis. By assumption there are two
  sequences $(a_{i})_{i\in\N}$ and $(b_{j})_{j\in\N}$ in $\no$ such that,
  for every sufficiently small $\eps,\delta$ we have
  \[
    g\circ(x+\eps)=g\circ x+\sum_{j=1}^{\infty}b_{j}\eps^{j}
  \]
  and
  \[
    f\circ(y+\delta)=\sum_{i\in\N}a_{i}\delta^{i}.
  \]
  Note that
  $(f\circ g)\circ(x+\eps) = f\circ(y+\sum_{j=1}^{\infty}b_{j}\eps^{j}) =
  \sum_{i\in\N}a_{i}(\sum_{j=1}^{\infty}b_{j}\eps^{j})^{i}$ for every
  sufficiently small $\eps$. To finish the proof it suffices to
  observe that, by \prettyref{prop:composition-of-power-series}, there
  is a sequence $(c_{m})_{m\in\N}$ in $\no$ such that, for every
  sufficiently small $\eps$, we have
  \[
    \sum_{k\in\N}a_{k}(\sum_{n=1}^{\infty}b_{n}\eps^{n})^{k}=\sum_{m\in\N}c_{m}\eps^{m}.
  \]
\end{proof}

\begin{cor}
  \label{cor:log-i-analytic}For all $i\in\N$, $\log_{i}(\omega)$ is
  surreal analytic.
\end{cor}

We can also verify that if $f\in\bracket{\omega}$ is surreal analytic,
the coefficients of its Taylor expansions can be calculated using the
derivation $\partial$ just like with classical analytic functions.

\begin{prop}
  \label{prop:analytic}If $f\in\bracket{\omega}$ is surreal analytic
  at $x\in\no^{>\R}$, then for every sufficiently small $\eps\in\no$ we
  have
  \[
    f\circ(x+\eps) = \sum_{n\in\N}\frac{1}{n!}(\partial^{n}\!f\circ
    x)\cdot\eps^{n}
  \]
  where $\partial^{0}\!f=f$ and
  $\partial^{n+1}\!f=\partial(\partial^{n}\!f)$.
\end{prop}

\begin{proof}
  Let $f\in\bracket{\omega}$ be analytic at $x\in\no^{>\R}$. Let
  $\hat{f}$ be associated function
  $x+\varepsilon\mapsto f\circ(x+\varepsilon)$, which by assumption is
  also surreal analytic (in the sense of
  \prettyref{def:surreal-analytic}).  By
  \prettyref{prop:surreal-analytic-is-differentiable}, we know that
  \[
    f\circ(x+\varepsilon) = \hat{f}(x+\varepsilon) =
    \sum_{i=0}^{\infty}\frac{\hat{f}^{(i)}(x)}{i!}\varepsilon^{i}.
  \]
  By \prettyref{cor:limit-def}, it follows by induction on $i$ that in
  fact $\hat{f}^{(i)}(x)=\partial^{i}\!f\circ x$, proving the desired
  conclusion.
\end{proof}

We can then conclude that \emph{every} omega-series is surreal
analytic.

\begin{thm}
  \label{thm:analytic}Every $f\in\bracket{\omega}$ is surreal
  analytic, and for every $x\in\no^{>\R}$ and every sufficiently small
  $\eps\in\no$ we have
  \[
    f\circ(x+\eps) = \sum_{i\in\N}\frac{1}{i!}(\partial^{i}\!f\circ
    x)\cdot\eps^{i}.
  \]
\end{thm}

\begin{proof}
  Let $f\in\bracket{\omega}$. We reason by induction on $\erank(f)$,
  where $\Delta=\{\log_{i}(\omega)\suchthat i\in\N\}$.

  The case $f=0$ is trivial, while the case $f=\log_{n}(\omega)$
  follows from \prettyref{cor:log-i-analytic} and
  \prettyref{prop:analytic}.  This shows the conclusion for
  $\erank(f)=0$, namely for $f\in\Delta_{1}=\Delta\cup\{0\}$.

  Now suppose $\erank(f)>0$. Write
  $f=\sum_{j<\alpha}r_{j}e^{\gamma_{j}}$, and recall that by definition
  $\erank(\gamma_{j})<\erank(f)$ for all $j<\alpha$. Therefore, by
  inductive hypothesis, we can assume that $\gamma_{j}$ is surreal
  analytic for every $j<\alpha$. Since $\exp(\omega)$ is surreal
  analytic by \prettyref{prop:exp-analytic}, it follows that
  $\exp(\omega)\circ\gamma_{j}=\exp(\gamma_{j})$ is surreal analytic by
  \prettyref{lem:composition-analytic}, hence so is
  $f_{j}:=r_{j}e^{\gamma_{j}}$. This means that for each $x$, there is
  some $\eps_{j}>0$ such that for all $\eps$ smaller than $\eps_{j}$ in
  absolute value, we have
  $f_{j}\circ(x+\eps)=\sum_{i\in\N}\frac{1}{i!}(\partial^{i}\!f_{j}\circ
  x)\cdot\eps^{i}$.

  Since $\partial$ is strongly additive, and
  $(f_{j}\suchthat j<\alpha)$ is summable, the family
  $(\partial f_{j}\suchthat j<\alpha)$ is also summable and
  $\sum_{j}\partial f_{j}=\partial\left(\sum_{j}f_{j}\right)$.  In turn,
  $(\partial f_{j}\circ x\suchthat j<\alpha)$ must be summable, and
  $\sum_{j}(\partial f_{j}\circ x)=(\sum_{j}\partial f_{j})\circ
  x=\partial\left(\sum_{j}f_{j}\right)\circ x=\partial f\circ x$.
  Similarly, by induction on $i\in\N$,
  $(\partial^{i}\!f_{j}\circ x\suchthat j<\alpha)$ is summable and
  $\sum_{j}(\partial^{i}\!f_{j}\circ x)=\partial^{i}\!f\circ x$.  By
  \prettyref{lem:double-index-sum}, for every sufficiently small
  $\eps$,
  $\left((\partial^{i}\!f_{j}\circ
    x)\cdot\eps^{i}\suchthat(i,j)\in\N\times\alpha\right)$ is summable
  and therefore, by \prettyref{cor:exchange-of-sums}, we have
  \[
    \sum_{j}\sum_{i}\frac{1}{i!}(\partial^{i}\!f_{j}\circ
    x)\cdot\eps^{i}=\sum_{i}\sum_{j}\frac{1}{i!}(\partial^{i}\!f_{j}\circ
    x)\cdot\eps^{i}=\sum_{i}\frac{1}{i!}\left(\partial^{i}\!f\circ
      x\right)\cdot\varepsilon^{i}.
  \]
  Recalling that
  $f_{j}\circ(x+\eps) = \sum_{i\in\N}\frac{1}{i!}(\partial^{i}\!f_{j}\circ
  x)\cdot\eps^{i}$, it follows that
  \[
    f\circ(x+\eps) =
    \sum_{j}(f_{j}\circ(x+\eps))=\sum_{j}\sum_{i}\frac{1}{i!}(\partial^{i}\!f_{j}\circ
    x)\cdot\eps^{i} = \sum_{n}\frac{1}{i!}(\partial^{i}\!f\circ
    x)\cdot\eps^{n}
  \]
  thus proving that $f$ is surreal analytic.
\end{proof}

\begin{rem}
  \label{rem:convergence-radius-bracket}When $f\in\R((\omega))^{LE}$
  and $x\in\R((\omega))^{LE}$, one can verify that there exists an
  $n\in\N$ such that the equation of \prettyref{thm:analytic} holds
  for any $\varepsilon\vleq e^{-\exp_{n}(\omega)}$. Indeed, note that
  the subfields $K_{m,\log_{i}(\omega)}$ (see \prettyref{def:LE}) are
  closed under the derivation $\partial$, and that there is some
  $k\in\N$ such that $g\circ x\in K_{m+k,\log_{i+k}(\omega)}$ for any
  $g\in K_{m,\log_{i}(\omega)}$. Then all the coefficients
  $\partial^{i}\!f\circ x/i!$ live in some fixed
  $K_{n,\log_{n}(\omega)}$, and it suffices to apply
  \prettyref{cor:power-series} to get the desired conclusion. In
  particular, one can give a meaningful definition of analyticity for
  LE-series by staying inside the field of LE-series, without
  resorting to $\no$.

  In full generality, \prettyref{cor:power-series} guarantees that the
  equation of \prettyref{thm:analytic} holds for any $\varepsilon$
  that is infinitesimal with respect to any non-zero omega-series
  $g\in\bracket{\omega}$.  In some cases, this is the best we can
  do. Take for instance $f=\sum_{n=0}^{\infty}e^{-\exp_{n}(\omega)}$.
  Then one can easily verify that
  $(\partial^{i}\!f\circ\omega)_{i\in\N}=(\partial^{i}\!f)_{i\in\N}$ is
  not summable, and in fact that
  $(\partial^{i}\!f\cdot\varepsilon)_{i\in\N}$ is not summable for any
  $\varepsilon$ such that $\varepsilon\vgeq e^{-\exp_{n}(\omega)}$ for
  some $n\in\N$, and in particular for any
  $\varepsilon\in\bracket{\omega}^{*}$.  Therefore, the expansion of
  $f\circ(\omega+\varepsilon)$ given by \prettyref{thm:analytic} only
  exists for the numbers $\varepsilon$ with absolute value smaller
  than any omega-series.
\end{rem}

\begin{cor}
  \label{cor:taylor-refined}Given $f\in\bracket{\omega}$ and
  $x\in\no^{>\R}$, we have
  \[
    f\circ(x+\eps)=f\circ x+(\partial f\circ
    x)\cdot\eps+\mathcal{O}((\partial^{2}f\circ x)\cdot\eps^{2})
  \]
  whenever $\varepsilon\in\no$ satisfies
  $(\partial^{i+2}\!f\circ x)\cdot\varepsilon^{i}\vleq\partial^{2}f\circ
  x$ for all $i\in\N$.
\end{cor}

\section{A negative result}

The interaction between the unique composition $\circ$ on
$\bracket{\omega}$ and the unique transserial derivation on
$\bracket{\omega}$ suggests looking for compositions that are
compatible with a transserial derivation.

\begin{defn}
  \label{def:compatible}Given a transserial subfield $T\subseteq\no$,
  a transserial derivation $\partial:T\to T$, and a composition
  $\circ\suchthat T\times\no^{>\R}\to\no$, we say that $\partial$ and
  $\circ$ are \textbf{compatible} if the following holds:
  \begin{enumerate}
  \item if $\partial f=0$, then $f\circ x=f$ for every $x$;
  \item $\partial f>0\implies f\circ x<f\circ y$ whenever $x<y$;
  \item $\partial(f\circ g)=(\partial f\circ g)\cdot\partial g$.
  \end{enumerate}
\end{defn}

\begin{thm}
  The unique surreal derivation $\partial$ on $\bracket{\omega}$ is
  compatible with the unique composition on $\bracket{\omega}$.
\end{thm}

\begin{proof}
  Condition (1) follow at once from $\ker(\partial)=\R$.

  For condition (2), let $f\in\bracket{\omega}$. We reason by
  induction on $\erank(f)$, where
  $\Delta=\{\log_{i}(\omega)\suchthat i\in\N\}$.  If $\erank(f)=0$,
  then the conclusion is easy: for instance if $f=\log_{i}(\omega)$,
  then $f\circ g=\log_{i}(g)$ and the chain rule in (3) can be verified
  as in the classical case, recalling also \prettyref{cor:limit-def}.
  Now suppose that $\erank(f)>0$. Write
  $f=\sum_{i<\alpha}r_{i}e^{\gamma_{i}}$, where
  $\erank(\gamma_{i})<\erank(f)$ for all $i<\alpha$. Suppose that
  $f\circ x\geq f\circ y$ for some $x<y$. Since the maps
  $g\mapsto g\circ x$, $g\mapsto g\circ y$ are substitutions, they
  preserve the relation $\vleq$
  (\prettyref{prop:substitution-is-hom}), so we must have
  $(r_{0}e^{\gamma_{0}})\circ x\geq(r_{0}e^{\gamma_{0}})\circ y$, so
  $r_{0}e^{\gamma_{0}\circ x}\geq r_{0}e^{\gamma_{0}\circ y}$.

  Without loss of generality, we may assume that $\gamma_{0}\neq0$ (by
  replacing $f$ with $f-r_{0}$) and that $r_{0}>0$ (by replacing $f$
  with $-f$). Under these assumptions, we must have
  $\gamma_{0}\circ x\geq\gamma_{0}\circ y$, so by inductive hypothesis
  $\partial\gamma_{0}\leq0$. Note moreover that since
  $\gamma_{0}\in\J^{\neq0}$, we must have $\partial\gamma_{0}\neq0$.  In
  turn, since $\partial f\sim r_{0}e^{\gamma_{0}}\partial\gamma_{0}$, it
  follows that $\partial f\geq0$, as desired.

  Point (3) is \prettyref{cor:chain-rule}.
\end{proof}

\begin{question}
  \label{que:compatible}We do not know whether there is a composition
  and a compatible transserial derivation (possibly with
  $\ker(\partial)$ bigger than $\R$) on the whole of $\no$.
\end{question}

Note that the present notion of compatibility is rather weak, and for
instance it does not require the conclusion of
\prettyref{thm:analytic} to hold, or even just
\prettyref{thm:taylor}. However, even such a weak notion does not
allow the ``simplest'' derivation $\partial:\no\to\no$ of
\cite{Berarducci2015} to be compatible with a composition.

\begin{thm}
  \label{thm:negative}The ``simplest'' surreal derivation
  $\partial:\no\to\no$ in \cite{Berarducci2015} cannot be compatible
  with a composition $\circ:\no\times\no^{>\R}\to\no$.
\end{thm}

\begin{proof}
  Let $y\in\no$, and observe that the rules of transserial derivations
  yield $\partial(\log_{n}(y))=\frac{1}{\prod_{i<n}\log_{i}(y)}$. Taking
  $y=\omega$ we obtain
  $\partial(\lambda_{-n})=\frac{1}{\prod_{i<n}\lambda_{-i}}$, where
  $\lambda_{-n}=\log_{n}(\omega)$. Now let $\partial:\no\to\no$ be the
  ``simplest derivation'' in \cite{Berarducci2015}. In that paper we
  showed that $\partial$ is surjective, so in particular there is an
  anti-derivative of $\frac{1}{\prod_{n\in\N}\lambda_{-n}}$.  In fact we
  proved that there is a log-atomic number $\lambda_{-\omega}\in\li$
  such that
  $\partial(\lambda_{-\omega})=\frac{1}{\prod_{n\in\N}\lambda_{-n}}$.
  With a suggestive notation $\lambda_{-\omega}$ is denoted
  $\log_{\omega}(\omega)$ in \cite{Aschenbrenner2015a}, suggesting that
  it should be considered as an infinite compositional iterate of
  $\log(\omega)$. In \cite{Berarducci2015} we showed that, if
  $\lambda$ is a log-atomic number bigger than $\exp_{n}(\omega)$ for
  every $n\in\N$, then
  \[
    \partial(\lambda)=\prod_{n\in\N}\log_{n}(\lambda).
  \]
  Note that there is a proper class of log-atomic numbers $\lambda$
  satisfying $\lambda>\exp_{n}(\omega)$ for all $n\in\N$, so the above
  differential equation has a proper class of solutions. Now fix such
  a solution $\lambda$ and suppose for a contradiction that $\partial$
  is compatible with a composition on the whole of $\no$. By the rules
  for $\partial$ and $\circ$ we obtain
  \begin{align*}
    \partial(\lambda_{-\omega}\circ\lambda) & =(\text{\ensuremath{\partial}}(\lambda_{-\omega})\circ\lambda)\cdot\partial(\lambda)\\
                                           & =\left(\frac{1}{\prod_{n}\log_{n}(\omega)}\circ\lambda\right)\cdot\partial(\lambda)\\
                                           & =\left(\frac{1}{\prod_{n}\log_{n}(\lambda)}\right)\cdot\partial(\lambda)=1.
  \end{align*}
  Since $\partial(\lambda_{-\omega})>0$, by the compatibility
  conditions the function $x\mapsto\lambda_{-\omega}\circ x$ is
  strictly increasing, so there is a proper class of elements of the
  form $\lambda_{-\omega}\circ x$ with derivative $1$. This however
  contradicts the fact that $\ker(\partial)=\R$ is a set.
\end{proof}

\begin{rem}
  The above result can be interpreted in different ways. The first is
  that there could be no reasonable composition on the whole of $\no$.
  The second is that, despite the positive results in
  \cite{Aschenbrenner2015a,Berarducci2015}, the simplest derivation in
  \cite{Berarducci2015} may have some shortcomings.  It is conceivable
  that, in order to be able to give positive solution to
  \prettyref{que:compatible}, we should allow a proper class as the
  kernel of $\partial$.
\end{rem}

\section{\label{sec:induction}Proof of the summability lemma
  (\prettyref{lem:summability})}

We will now give a proof of \prettyref{lem:summability}. We work under
the notations of \prettyref{sec:Substitutions}. Suppose that
$c_{0}:\Delta\to\no$ is a given pre-substitution. Then we wish to prove
the following:

\begin{namedthm}[\prettyref{lem:summability}]
  Assume $I(\alpha)$. Then $(\ctree(T)\suchthat T\in\admtree(x))$ is
  summable for all $x\in\Delta_{\alpha+1}\setminus\Delta_{\alpha}$.  In
  particular, $I(\alpha)$ implies $I(\alpha+1)$.
\end{namedthm}

For the rest of this section, \emph{let $c_{0}:\Delta\to\no$ be a
  pre-substitution, and assume that the inductive hypothesis
  $I(\alpha)$ holds.} Then $c:\Delta_{\alpha}\to\no$ is well defined,
and the objects $\admtree(x)$, $\ctree(T)$ and $\admtreeinf(x)$ are
clearly well defined for all $x\in\Delta_{\alpha+1}$ and all
$T\in\admtree(x)$.  Moreover, recall that by \prettyref{prop:cexp},
$c(t)$ is also well defined for all terms
$t\in\Delta_{\alpha+1}\cap\T$.

\subsection{A property of pre-substitutions}

We start by observing a rather technical, but crucial fact on
pre-substitutions.

\begin{lem}
  \label{lem:supp-inverse}Let $x\in\no$ and $\m$ be the leading
  monomial of $x$. Then
  \[
    \supp(x)\subseteq\bigcup_{n=0}^{\infty}\m^{n+1}\cdot\supp(x^{-1})^{n}.
  \]
\end{lem}

\begin{proof}
  Let $t=r\m$ be the leading term of $x$. Write
  $x^{-1}=t^{-1}(1+\varepsilon)$, where $\varepsilon\vless1$. Then
  \[
    x=\frac{t}{(1+\varepsilon)}=t\cdot\sum_{n=0}^{\infty}(-1)^{n}\varepsilon^{n},
  \]
  hence every element in the support of $x$ has the form
  $\m\cdot\n_{1}\cdot\ldots\cdot\n_{n}$ with $n\geq0$ and
  $\n_{i}\in\supp(\eps)$. On the other hand, since
  $\varepsilon=tx^{-1}-1=r\m x^{-1}-1$ and $\eps\prec1$, we have
  $\supp(\eps)\subseteq\m\cdot\supp(x^{-1})$, and the conclusion
  follows.
\end{proof}

\begin{lem}
  \label{lem:finer-large-gaps}Let $c_{0}:\Delta\to\no$ be a
  pre-substitution.  Let $(\lambda_{i})_{i\in\N}$, $(\m_{i})_{i\in\N}$ be
  two sequences such that $\lambda_{i}\in\Delta$ and
  $\m_{i}\in\supp(c_{0}(\lambda_{i}))$ for all $i\in\N$. Then there is an
  increasing sequence of indexes $(i_{j})_{j\in\N}$ such that one of the
  following holds:
  \begin{enumerate}
  \item the subsequence $(\lambda_{i_{j}})_{j\in\N}$ is decreasing and
    for all $j\in\N$
    \[
      \frac{\m_{i_{j+1}}}{\m_{i_{j}}} \vless
      \frac{c_{0}(\lambda_{i_{j+1}})}{c_{0}(\lambda_{i_{j}})} \vless 1;
    \]
  \item the subsequence $(\lambda_{i_{j}})_{j\in\N}$ is increasing and
    for all $j\in\N$
    \[
      \frac{\m_{i_{j+1}}}{\m_{i_{j}}} \vless c_{0}(\lambda_{i_{j+1}})^{2};
    \]
  \item the subsequence $(\lambda_{i_{j}})_{j\in\N}$ is constant and for
    all $j\in\N$
    \[
      \frac{\m_{i_{j+1}}}{\m_{i_{j}}} \vleq 1.
    \]
  \end{enumerate}
  Note that in all three cases we have
  $\frac{\m_{i_{j+1}}}{\m_{i_{j}}} \vless c_{0}(\lambda_{i_{j+1}})^{2}$.
\end{lem}

\begin{proof}
  Let $\lambda_{i}=:e^{\mu_{i}}$. Note that $\mu_{i}\in\Delta$. We have
  \[
    c_{0}(\lambda_{i}) = e^{c_{0}(\mu_{i})} =
    e^{c_{0}(\mu_{i})^{\uparrow=}}e^{c_{0}(\mu_{i})^{\downarrow}}.
  \]
  Thus
  $\n_{i} := \m_{i}e^{-c_{0}(\mu_{i})^{\uparrow}} \in
  \supp\left(\exp(c_{0}(\mu_{i})^{\small=})\right) =
  \supp(\exp(c_{0}(\mu_{i})^{\small})$, and therefore there is some
  $n_{i}\in\N$ such that
  $\n_{i}\in\supp((c_{0}(\mu_{i})^{\small})^{n_{i}})$.  After extracting a
  subsequence we may assume that $(\lambda_{i})_{i\in\N}$ is monotone,
  so either increasing, decreasing, or constant.

  (1) Suppose that $(\lambda_{i})_{i\in\N}$ is decreasing. Then
  $(\mu_{i})_{i\in\N}$ is also decreasing, hence the family
  $(c_{0}(\mu_{i})\suchthat i\in\N)$ is summable. In particular,
  $(c_{0}(\mu_{i})^{\downarrow}\suchthat i\in\N)$ is summable, and by
  \prettyref{cor:power-series-summable},
  $(\exp(c_{0}(\mu_{i})^{\downarrow})\suchthat i\in\N)$ is summable. We
  may therefore extract a subsequence and assume that
  $(\n_{i})_{i\in\N}$ is decreasing, so that
  \[
    \m_{i+1}e^{-c_{0}(\mu_{i+1})^{\uparrow}} \vless
    \m_{i}e^{-c_{0}(\mu_{i})^{\uparrow}}.
  \]
  Since $c_{0}(\lambda_{i})=e^{c_{0}(\mu_{i})}$, it follows that
  \[
    \frac{\m_{i+1}}{\m_{i}} \vless
    \frac{e^{c_{0}(\mu_{i+1})^{\uparrow}}}{e^{c_{0}(\mu_{i})^{\uparrow}}} \veq
    \frac{c_{0}(\lambda_{i+1})}{c_{0}(\lambda_{i})} \vless 1.
  \]

  (2) Consider now the case when $(\lambda_{i})_{i\in\N}$ is increasing.
  Let $\o_{i}:=\LM(c_{0}(\mu_{i}))$. By \prettyref{lem:supp-inverse},
  applied with $x=c_{0}(\mu_{i})$, we deduce that
  \[
    \supp(c_{0}(\mu_{i})) \subseteq \bigcup_{m=0}^{\infty}\o_{i}^{m+1} \cdot
    \supp(c_{0}(\mu_{i})^{-1})^{m}.
  \]
  Since
  $\n_{i} = \frac{\m_{i}}{e^{c_{0}(\mu_{i})^{\big}}} \in
  \supp((c_{0}(\mu_{i})^{\small})^{n_{i}})$, it follows that there is an
  $m_{i}\in\N$ such that
  \[
    \frac{\m_{i}}{e^{c_{0}(\mu_{i})^{\big}}} \in \o_{i}^{n_{i}(m_{i}+1)} \cdot
    \supp(c_{0}(\mu_{i})^{-1})^{n_{i}m_{i}}
  \]
  and therefore
  $\m_{i}\cdot e^{-c_{0}(\mu_{i})^{\big}} \cdot \o_{i}^{-n_{i}(m_{i}+1)} \in
  \supp(c_{0}(\mu_{i})^{-1})^{n_{i}m_{i}}$.

  Now observe that $c_{0}(\mu_{i})^{-1}\vless1$ and that the family
  $(c_{0}(\mu_{i})^{-1}\suchthat i\in\N)$ is summable because
  $(\mu_{i}^{-1})_{i\in\N}$ is decreasing. By
  \prettyref{cor:power-series-summable}, applied with
  $\eps_{i}=c_{0}(\mu_{i})^{-1}$, the family
  $\left(\m_{i}\cdot
    e^{-c_{0}(\mu_{i})^{\big}}\cdot\o_{i}^{-n_{i}(m_{i}+1)}\suchthat
    i\in\N\right)$ is summable. We may therefore extract a subsequence
  and assume that
  \[
    \frac{\m_{i}}{e^{c_{0}(\mu_{i})^{\big}}\cdot\o_{i}^{n_{i}(m_{i}+1)}}
    \vgreater
    \frac{\m_{i+1}}{e^{c_{0}(\mu_{i+1})^{\big}}\cdot\o_{i+1}^{n_{i+1}(m_{i+1}+1)}}.
  \]
  Since $c_{0}(\mu_{i})$ is positive infinite,
  $e^{c_{0}(\mu_{i})^{\big}}\vgreater c_{0}(\mu_{i})^{n}\veq\o_{i}^{n}$ for any
  $n\in\N$, so
  \[
    \frac{\m_{i+1}}{e^{2c_{0}(\mu_{i+1})^{\big}}} \prec
    \frac{\m_{i+1}}{e^{c_{0}(\mu_{i+1})^{\big}} \cdot
      \o_{i+1}^{n_{i+1}(m_{i+1}+1)}} \prec
    \frac{\m_{i}}{e^{c_{0}(\mu_{i})^{\big}}}.
  \]
  Therefore,
  \[
    \frac{\m_{i+1}}{\m_{i}} \prec
    \frac{e^{2c_{0}(\mu_{i+1})^{\big}}}{e^{c_{0}(\mu_{i})^{\big}}} \veq
    \frac{c_{0}(\lambda_{i+1})^{2}}{c_{0}(\lambda_{i})} \preceq
    c_{0}(\lambda_{i+1})^{2}.
  \]

  (3) Finally, suppose that there is a $\lambda\in\Delta$ such that
  $\lambda_{i}=\lambda$ for all $i\in\N$. In this case all the
  monomials $\m_{i}$ are in the support of $c_{0}(\lambda)\in\no$, hence
  obviously we may extract a subsequence and assume that
  $\m_{i+1}\vleq\m_{i}$ for all $i\in\N$.
\end{proof}

\subsection{Further properties of the extensions}

Recall that $I(\alpha)$ implies that $c:\Delta_{\alpha}\to\no$ is a
substitution when $\alpha\geq3$ (\prettyref{cor:csubstitution}).  In
particular, $c$ preserves the ordering and the dominance relation
$\vless$ by \prettyref{prop:substitution-is-hom}. We observe that
$I(\alpha)$ implies similar monotonicity properties for $\alpha<3$,
and also for terms in $\Delta_{\alpha+1}$.

\begin{prop}
  \label{prop:cmonotone}For all $x,y\in\Delta_{\alpha}$, and for all
  $x,y\in\Delta_{\alpha+1}\cap\R^{*}\M$, we have $x<y\to c(x)<c(y)$ and
  $x\prec y\to c(x)\vless c(y)$.
\end{prop}

\begin{proof}
  If $\alpha$ is $0$ or $1$, then for all $x,y\in\Delta_{\alpha}$ we
  have $x<y\to c(x)<c(y)$ and $x\prec y\to c(x)\vless c(y)$ by
  definition of pre-substitution. The same conclusion holds for
  $\alpha\geq3$ by \prettyref{cor:c-alpha-3-substitution} and
  \prettyref{prop:substitution-is-hom}.  For $\alpha=2$, note that by
  \prettyref{prop:csum}, if we expand some $x\in\Delta_{2}\setminus\R$
  as
  $x=r_{0}e^{\lambda_{0}}+\sum_{1\leq i<\beta}r_{i}e^{\lambda_{i}}+s$ (where
  $r_{i},s\in\R$, $\lambda_{i}\in\Delta$, and
  $\lambda_{i}>\lambda_{j}$ for all $i\leq j<\beta$), we have
  \[
    c(x) = r_{0}e^{c_{0}(\lambda_{0})} + \sum_{1\leq
      i<\beta}r_{i}e^{c_{0}(\lambda_{i})} + s,
  \]
  while $c(r)=r$ for all $r\in\R$ by \prettyref{prop:cR}. By
  definition of pre-substitution, it follows at once that
  $c(x)\sim r_{0}e^{c_{0}(\lambda_{0})}$, and in turn, that $c(x)>0$ if
  and only if $x>0$ (and obviously $c(r)>0$ if and only if
  $r>0$). Since $\Delta_{2}$ is an additive group, we have
  $x<y\to c(x)<c(y)$ for all $x,y\in\Delta_{2}$. By the same argument,
  it also follows that $x\prec y\to c(x)\vless c(y)$ for all
  $x,y\in\Delta_{2}$.

  Now take some $x,y\in\Delta_{\alpha+1}\cap\R^{*}\M$. Write
  $x=re^{\gamma},y=se^{\delta}$, with $r,s\in\R^{*}$ and
  $\gamma,\delta\in\J$. By \prettyref{prop:cexp}, $c(re^{\gamma})$ and
  $c(se^{\delta})$ are well defined and equal to respectively
  $re^{c(\gamma)}$, $se^{c(\delta)}$. We observe that if
  $\gamma<\delta$, then $c(\gamma)<c(\delta)$, and if moreover
  $0<\gamma$, then $0<c(\gamma)$ and $\gamma\prec\delta$, so
  $c(\gamma)\prec c(\delta)$.  This easily implies that
  $x<y\to c(x)<c(y)$ and $x\vless y\to c(x)\vless c(y)$.
\end{proof}

We also need the following properties of admissible trees.

\begin{lem}
  \label{lem:children}Let $x\in\Delta_{\alpha+1}$ and
  $T=\tree{re^{\gamma}}n{\tau}\in\admtree(x)$.  We have:
  \begin{enumerate}
  \item
    $re^{c(\gamma)^{\bigeq}}\veq
    re^{c(\gamma)}=c(re^{\gamma})=c(\treeroot(T))$;
  \item if $re^{\gamma}=\treeroot(T)\nin\Delta$, then
    $\ctree(T)\veq c(\treeroot(T))\cdot\prod_{i<n}\ctree(\tau(i))$;
  \item if $U=\tau(i)$ is a child of $T$, then $\ctree(U)$ is
    infinitesimal;
  \item if $U$ is a proper descendant of $T$, then $\ctree(U)$ is
    infinitesimal;
  \item $\ctree(T)\preceq c(\treeroot(T))$;
  \item if $\size(T)>1$, then all the leaves of $T$ have root in
    $\Delta$.
  \end{enumerate}
\end{lem}

\begin{proof}
  (1) follows from \prettyref{prop:cexp}.

  (2), (3), (4) follow at once from the definitions and (1).

  (5) If $\lambda=\treeroot(T)\in\Delta$, then
  $\ctree(T)\in\term(c_{0}(\lambda))$, so
  $\ctree(T)\preceq c_{0}(\lambda)=c(\treeroot(T))$ as desired.  If
  $\treeroot(T)\nin\Delta$, then
  $\ctree(T)\veq c(\treeroot(T))\cdot\prod_{i<n}\ctree(\tau(i))$ by
  (2), and since $\ctree(\tau(i))\prec1$ for each $i<n$ by (3), we
  reach the same conclusion.

  (6) Assume $\size(T)>1$, and let $L$ be a leaf of $T$. Then $L$ is a
  leaf of some child of $T$. Reasoning by induction, we may directly
  assume, without loss of generality, that $L$ is a child of $T$.
  Write $L=\tree{se^{\delta}}0{\sigma}$. Note that $se^{\delta}$ is a
  term of $\gamma=\vell(\treeroot(T))\in\J$, so
  $\treeroot(L)=se^{\delta}\succ1$.  By \prettyref{prop:cmonotone}, it
  follows that $c(\treeroot(L))\succ1$.  Now suppose by contradiction
  that $se^{\delta}\notin\Delta$. Then (2) implies that
  $\ctree(L)\veq c(\treeroot(T))\vgreater1$, but by (4) we must have
  $\ctree(L)\prec1$. Therefore, $se^{\delta}\in\Delta$, as desired.
\end{proof}

\subsection{Bad sequences}

In order to prove that the family
$(\ctree(T)\suchthat T\in\admtree(x))$ is summable for any
$x\in\Delta_{\alpha+1}$, by \prettyref{rem:summability-criterion}, one
could try to verify that there is no injective sequence
$(T_{i})_{i\in\N}$ of trees in $\admtree(x)$ such that
$\ctree(T_{i})\vleq\ctree(T_{i+1})$ for all $i\in\N$. However, we will
actually prove the stronger statement that there are \emph{no bad
  sequences}, which are defined as follows:

\begin{defn}
  Let $x\in\Delta_{\alpha+1}$ and let $(T_{i})_{i\in\N}$ be a sequence of
  trees in $\admtree(x)$. We say that the sequence is \textbf{bad} if
  it is injective, $\treeroot(T_{i})\succeq\treeroot(T_{i+1})$ for each
  $i\in\N$, and
  \[
    \left(\frac{\ctree(T_{i})}{\ctree(T_{i+1})}\right)^{n} \vleq
    \frac{c(\treeroot(T_{i}))}{c(\treeroot(T_{i+1}))}
  \]
  for all $i,n\in\N$.
\end{defn}

For instance, \prettyref{lem:finer-large-gaps}(1) and (3) immediately
imply that there are no bad sequences in $\admtree(x)$ for any
$x\in\Delta_{2}$.  The non-existence of bad sequences in a given
$\admtree(x)$ quickly implies the desired summability.

\begin{prop}
  \label{prop:no-bad-implies-summability}Let
  $x \in \Delta_{\alpha+1}$.  If there are no bad sequences in
  $\admtree(x)$, then $(\ctree(T)\suchthat T\in\admtree(x))$ is
  summable.
\end{prop}

\begin{proof}
  Suppose that $(\ctree(T)\suchthat T\in\admtree(x))$ is not summable.
  Then there is an injective sequence of trees $(T_{i})_{i\in\N}$ in
  $\admtree(x)$ such that
  \[
    \ctree(T_{i})\vleq\ctree(T_{i+1})
  \]
  for all $i\in\N$. After extracting a subsequence, we may assume that
  $\treeroot(T_{i})\succeq\treeroot(T_{i+1})$ for every $i\in\N$, as all
  these roots are terms of $x$. Therefore,
  $c(\treeroot(T_{i}))\succeq c(\treeroot(T_{i+1}))$ for all $i\in\N$ by
  \prettyref{prop:cmonotone}. It follows that for all $i,n\in\N$ we
  have
  \[
    \left(\frac{\ctree(T_{i})}{\ctree(T_{i+1})}\right)^{n}\vleq1\vleq\frac{c(\treeroot(T_{i}))}{c(\treeroot(T_{i+1}))},
  \]
  so the sequence $(T_{i})_{i\in\N}$ is bad.
\end{proof}

\begin{rem}
  If $(T_{i})_{i\in\N}$ is a bad sequence, then all its subsequences are
  bad. This follows from the fact that for all $i,k,n\in\N$ we have
  \[
    \left(\frac{\ctree(T_{i})}{\ctree(T_{i+k+1})}\right)^{n} =
    \left(\prod_{j=0}^{k}\frac{\ctree(T_{i+j})}{\ctree(T_{i+j+1})}\right)^{n}
    \vleq
    \prod_{j=0}^{k}\frac{c(\treeroot(T_{i+j}))}{c(\treeroot(T_{i+j+1}))} =
    \frac{c(\treeroot(T_{i}))}{c(\treeroot(T_{i+k+1}))}.
  \]
\end{rem}

We start with a few special cases in which it is easy to prove that
sequences of trees are not bad.

\begin{prop}
  \label{prop:root-in-delta}Let $x\in\Delta_{\alpha+1}$. Let
  $(T_{i})_{i\in\N}$ be a sequence of distinct trees in
  $\admtree(x)$. If $\treeroot(T_{i})\in\Delta$ for all $i\in\N$, then
  $(T_{i})_{i\in\N}$ is not bad.
\end{prop}

\begin{proof}
  Write $T_{i}=\tree{\lambda_{i}}0{t_{i}}$, where
  $t_{i}=\ctree(T_{i})$ is a term of $c_{0}(\lambda_{i})$. Since
  $\lambda_{i}\in\term(x)$ for each $i\in\N$, after extracting a
  subsequence, we may assume that $(\lambda_{i}\suchthat i\in\N)$ is
  either constant or decreasing.  In the former case, all the
  contributions $\ctree(T_{i})$ are distinct elements of
  $\term(c_{0}(\lambda))$ for some fixed $\lambda\in\Delta$, so after
  extracting a subsequence we may assume
  $\ctree(T_{i})\succ\ctree(T_{i+1})$ for all $i\in\N$, so the sequence
  is not bad. In the latter case, by \prettyref{lem:finer-large-gaps},
  we may extract a further subsequence and assume that
  \[
    \frac{\ctree(T_{i})}{\ctree(T_{i+1})} = \frac{t_{i}}{t_{i+1}} \succ
    \frac{c_{0}(\lambda_{i})}{c_{0}(\lambda_{i+1})} =
    \frac{c(\treeroot(T_{i}))}{c(\treeroot(T_{i+1}))}.
  \]
  Therefore, $(T_{i})_{i\in\N}$ is not bad.
\end{proof}

\begin{prop}
  \label{prop:cterm}Let $t$ be a term in $\Delta_{\alpha+1}$. Then
  there are no bad sequences in $\admtree(t)$.
\end{prop}

\begin{proof}
  Let $(T_{i})_{i\in\N}$ be a sequence of distinct trees in
  $\admtree(t)$.  We want to prove that $(T_{i})_{i\in\N}$ is not
  bad. Since $t$ is a term, by \prettyref{prop:cexp}
  $(\ctree(T)\suchthat T\in\admtree(t))$ is summable. Thus, extracting
  a subsequence, we can assume that
  $\ctree(T_{i})\succ\ctree(T_{i+1})$ for every $i\in\N$. Observing that
  $\treeroot(T_{i})=t$ for every $i\in\N$, it follows that
  $\frac{\overline{c}(T_{i})}{\overline{c}(T_{i+1})} \succ 1 =
  \frac{c(\treeroot(T_{i}))}{c(\treeroot(T_{i+1}))}$, and therefore
  $(T_{i})_{i\in\N}$ is not bad.
\end{proof}

\subsection{Two types of sequences of trees}

We now distinguish two special types of sequences of trees, and verify
that every injective sequences of trees in some given $\admtree(x)$
has at least one subsequence of one of the two types.

\begin{defn}
  Let $x\in\Delta_{\alpha+1}$ and let
  $T_{i}=\tree{r_{i}e^{\gamma_{i}}}{n_{i}}{\tau_{i}}\in\admtree(x)$ be
  distinct trees for $i\in\N$ such that $(\gamma_{i})_{i\in\N}$ is
  weakly decreasing.

  We say that the sequence $(T_{i})_{i\in\N}$ has type:
  \begin{itemize}
  \item[\textbf{(A)}] if
    $\treeroot(\tau_{i}(j))\vgreater\gamma_{0}-\gamma_{i}$ for all
    $i\in\N$, $j<n_{i}$;
  \item[\textbf{(B)}] if $n_{0}\geq1$ and for all $i\in\N^{>0}$ there is
    $k<n_{i}$ such that
    $\treeroot(\tau_{i}(k))\vleq\gamma_{i-1}-\gamma_{i}$.
  \end{itemize}
\end{defn}

Note that a sequence $(T_{i})_{i\in\N}$ may be of neither type. A
sequence with $n_{i}=0$ for all $i\in\N$, or with
$(\gamma_{i})_{i\in\N}$ constant, is vacuously of type (A). Moreover,
for a sequence of type (B), $(\gamma_{i})_{i\in\N}$ is necessarily
strictly decreasing and $n_{i}\geq1$ for all $i\in\N$.

\begin{lem}
  \label{lem:AB-stable}If $(T_{i})_{i\in\N}$ is a sequence of type (A)
  or (B), then all its subsequences have type (A) or (B) respectively.
\end{lem}

\begin{proof}
  Suppose $(T_{i})_{i\in\N}$ is of type (A) and let
  $(T_{i_{j}})_{j\in\N}$ be a subsequence. Since
  $(\gamma_{i})_{i\in\N}$ is weakly decreasing, for all $k<n_{i_{j}}$ we
  have
  \[
    \treeroot(\tau_{i_{j}}(k)) \vgreater \gamma_{0}-\gamma_{i_{j}} \vgeq
    \gamma_{i_{0}}-\gamma_{i_{j}},
  \]
  so the subsequence is of type (A).

  Now let $(T_{i})_{i\in\N}$ be a sequence of type (B). Write
  $T_{i}=\tree{r_{i}e^{\gamma_{i}}}{n_{i}}{\tau_{i}}$.  Using again the fact
  that $(\gamma_{i})_{i\in\N}$ is weakly decreasing, if $k$ is such that
  $\treeroot(\tau_{i}(k))\vleq\gamma_{i-1}-\gamma_{i}$, then
  $\treeroot(\tau_{i}(k))\vleq\gamma_{j}-\gamma_{i}$ for all $j<i$, so
  any any subsequence of $(T_{i})_{i\in\N}$ is of type (B).
\end{proof}

\begin{prop}
  \label{prop:type-A-or-B}Let $x\in\Delta_{\alpha+1}$ and let
  $T_{i}=\tree{r_{i}e^{\gamma_{i}}}{n_{i}}{\tau_{i}}\in\admtree(x)$ be
  distinct trees for $i\in\N$. Then $(T_{i})_{i\in\N}$ has a subsequence
  of type (A) or (B).
\end{prop}

\begin{proof}
  After extracting a subsequence, we may assume that
  $(\gamma_{i})_{i\in\N}$ is weakly decreasing. If $n_{i}=0$ for every
  $i\in\N$, then $(T_{i})_{i\in\N}$ is of type (A) and we are done. We
  can therefore suppose without loss of generality that $n_{0}\geq1$.

  We proceed by trying to construct a subsequence
  $(T_{i_{j}})_{j\in\N}$ of type (B), and check that when the
  construction fail we find a subsequence of type (A). We define
  $T_{i_{j}}$ by induction on $j\in\N$. For $j=0$, we let
  $T_{i_{j}}=T_{i_{0}}:=T_{0}$.

  Assuming that $T_{i_{j}}$ has been defined, we have two cases. If
  $\treeroot(\tau_{i}(k))\vgreater\gamma_{i_{j}}-\gamma_{i}$ for all
  $i>i_{j}$, $k<n_{i}$, then the sequence
  $T_{i_{j}},T_{i_{j}+1},T_{i_{j}+2},\ldots$, has type (A), and we are
  done. Otherwise, we let $i_{j+1}$ be the minimum $i$ for which there
  exists $k$ such that
  $\treeroot(\tau_{i}(k))\vleq\gamma_{i_{j}}-\gamma_{i}$.

  Clearly, either the procedure fails after a finite number of steps,
  and we find a subsequence of type (A), or it defines a subsequence
  $(T_{i_{j}})_{j\in\N}$ of type (B), as desired.
\end{proof}

\subsection{No bad sequences of type (A)}

As a start, it is fairly easy to see that bad sequences of type (A) do
not exist.

\begin{prop}
  \label{prop:type-A-not-bad}Let $x\in\Delta_{\alpha+1}$. Then
  $\admtree(x)$ contains no bad sequences of type (A).
\end{prop}

\begin{proof}
  For a contradiction let $(T_{i})_{i\in\N}$ be a bad sequence in
  $\admtree(x)$ of type (A). By \prettyref{prop:cterm} the sequence of
  terms $(\treeroot(T_{i})\suchthat i\in\N)$ cannot be constant, so by
  taking a subsequence we can assume that the terms $\treeroot(T_{i})$
  are distinct, and since they are all terms of $x$, we may also
  assume (taking another subsequence) that
  $\treeroot(T_{0})\succ\treeroot(T_{1})\succ\treeroot(T_{2})\succ\ldots$.
  By \prettyref{prop:cmonotone} it then follows that
  $c(\treeroot(T_{n}))\succ c(\treeroot(T_{n+1}))$ for every $n\in\N$.

  Let $i\in\N$ and write
  $T_{i}=\tree{r_{i}e^{\gamma_{i}}}{n_{i}}{\tau_{i}}$.  By assumption, for
  any child $U=\tau_{i}(j)$ of $T_{i}$ we have
  $\treeroot(U)\vgreater\gamma_{0}-\gamma_{i}$ (this holds vacuously if
  $T_{i}$ has no children). We claim that for any such $U$ we must have
  $\treeroot(U)\in\term(\gamma_{0})$. Indeed by construction
  $\treeroot(U)\in\term(\gamma_{i})$; therefore, if
  $\treeroot(U)\nin\term(\gamma_{0})$, then $\treeroot(U)$ would be a
  term of the difference $\gamma_{0}-\gamma_{i}$, contradicting the
  assumption $\treeroot(U)\vgreater\gamma_{0}-\gamma_{i}$.

  We have thus proved that all the roots of the children of the trees
  $T_{i}$ are terms of $\gamma_{0}=\vell(\treeroot(T_{0}))$; hence, we
  can replace the root of each $T_{i}$ with $e^{\gamma_{0}}$ obtaining a
  new sequence $T_{i}':=\tree{e^{\gamma_{0}}}{n_{i}}{\tau_{i}}$ in
  $\admtree(e^{\gamma_{0}})$. Since $T_{i}$ and $T_{i}'$ have the same
  children, by \prettyref{lem:children}(2) we have:
  \[
    \frac{\ctree(T_{i})}{\ctree(T_{i}')} \veq
    \frac{c(\treeroot(T_{i}))}{c(\treeroot(T_{i}'))} \veq
    \frac{c(e^{\gamma_{i}})}{c(e^{\gamma_{0}})}.
  \]
  By \prettyref{prop:cexp}, the family
  $(\ctree(T')\suchthat T'\in\admtree(e^{\gamma_{0}}))$ is
  summable. Therefore, after extracting a subsequence we may assume
  that $\frac{\ctree(T_{i}')}{\ctree(T_{i+1}')}\succeq1$ (note that the
  inequality is not necessarily strict, because the trees $T_{i}'$
  might not be distinct). It follows that
  \[
    \frac{\ctree(T_{i})}{\ctree(T_{i+1})} \veq
    \frac{c(\treeroot(T_{i}))}{c(\treeroot(T_{i+1}))} \cdot
    \frac{\ctree(T_{i}')}{\ctree(T_{i+1}')} \succeq
    \frac{c(\treeroot(T_{i}))}{c(\treeroot(T_{i+1}))} \succ 1.
  \]
  Therefore, $(T_{i})_{i\in\N}$ is not bad.
\end{proof}

\subsection{Pruning trees}

In the sequel we consider trees in $\admtree(x)$ for some
$x\in\Delta_{\alpha+1}$.  We establish a procedure to ``prune'' a tree
$T$, that is, to remove some descendants, in such a way that its
contribution $\ctree(T)$ changes only by a small amount.

\begin{defn}
  Let $T=\tree{re^{\gamma}}n{\tau}$ be an admissible tree
  (i.e.~$T\in\admtree(re^{\gamma})$), $U$ be a child of $T$
  (necessarily admissible), and $U'$ be an admissible tree with the
  same root as $U$. Let $j$ be the \emph{minimum} integer such that
  $\tau(j)=U$.
  \begin{enumerate}
  \item We define $T[U'/U]$ as $T$ with $U$ replaced by $U'$. More
    precisely,
    \[
      T[U'/U]:=\tree{re^{\gamma}}n{\tau^{*}}
    \]
    where $\tau^{*}(i):=\tau(i)$ for $i\neq j$ and $\tau^{*}(j):=U'$.
    Note that if $\ctree(U')\prec1$, then $\replace T{U'}U$ is again
    an admissible tree.
  \item We define $\remove TU$ as the admissible tree obtained from
    $T$ by removing the child $U$. More precisely,
    \[
      \remove TU:=\tree{re^{\gamma}}{n-1}{\tau^{*}}
    \]
    where $\tau^{*}(i):=\tau(i)$ for $i<j$ and
    $\tau^{*}(i):=\tau(i+1)$ for $i\geq j$.
  \end{enumerate}
\end{defn}

\begin{defn}
  Let $T=\langle re^{\gamma},n,\tau\rangle\in\admtree(x)$ with
  $\size(T)>1$.  If $L$ is a leaf of $T$, we define the
  \textbf{minimal child of $T$ with leaf $L$} to be the child
  $U=\tau(j)$ of $T$ such that:
  \begin{enumerate}
  \item $L$ is a leaf of $U$ (possibly $L=U$);
  \item among such children, $\treeroot(U)$ is minimal with respect to
    $\preceq$;
  \item among such children, $j$ is minimal.
  \end{enumerate}
\end{defn}

\begin{defn}
  Let $T=\tree{re^{\gamma}}n{\tau}\in\admtree(x)$ with $\size(T)>1$ and
  let $L$ be a leaf of $T$. We define $\prune TL$ by induction on
  $\size(T)$ as follows. Let $U$ be the minimal child of $T$ with leaf
  $L$. We define:
  \begin{enumerate}
  \item if $\size(U)=1$ (namely $L=U$), let $\prune TL:=\remove TL$;
  \item if $\size(U)>1$ and $\overline{c}(\prune UL)\prec1$, let
    $\prune TL:=T[\prune UL/U]$;
  \item if $\size(U)>1$ and $\ctree(\prune UL)\succeq1$, let
    $\prune TL:=\remove TU$.
  \end{enumerate}
\end{defn}

\begin{rem}
  \label{rem:prune-admissible}Note that in all three cases,
  $\prune TL$ is still an admissible tree; in particular, in (2) this
  is guaranteed by the condition $\ctree(\prune UL)\prec1$, as for all
  children $S$ of an admissible tree the contribution $\ctree(S)$ must
  be infinitesimal.
\end{rem}

\begin{lem}
  \label{lem:prune}Let $L$ be a leaf in $T\in\admtree(x)$, with
  $\size(T)>1$, and let $U$ be the minimal child of $T$ with leaf
  $L$. We have:
  \begin{enumerate}
  \item $\size(\prune TL)<\size(T)$ and
    $\treeroot(\prune TL)=\treeroot(T)$;
  \item $\prune TL\in\admtree(x)$;
  \item if $\prune TL=\remove TU$, then
    $\ctree(T)\veq\ctree(\prune TL)\cdot\ctree(U)$;
  \item if $\prune TL:=T[\prune UL/U]$, then
    $\ctree(T)=\ctree(\prune TL)\cdot\frac{\ctree(U)}{\ctree(\prune
      UL)}$;
  \item $\ctree(\prune TL)\vgreater\ctree(T)$;
  \end{enumerate}
\end{lem}

\begin{proof}
  We work by induction on $\size(T)$. Point (1) is straightforward and
  point (2) is \prettyref{rem:prune-admissible}.

  For (3), let $T=:\tree{re^{\gamma}}n{\tau}$ and let $j<n$ be minimal
  such that $U=\tau(j)$. By definition we have
  \[
    \ctree(T) = re^{c(\gamma)^{\bigeq}} \cdot \ctree(U) \cdot \frac{1}{n!} \prod_{\substack{i<n\\
        i\neq j } }\ctree(\tau(i))
  \]
  while
  \[
    \ctree(\remove TU) = re^{c(\gamma)^{\bigeq}} \cdot \frac{1}{(n-1)!} \cdot \prod_{\substack{i<n\\
        i\neq j } }\ctree(\tau(i))
  \]
  Thus clearly $\ctree(\remove T{U)\veq\frac{\ctree(T)}{\ctree(U)}}$
  and (3) follows.

  A similar argument shows that if $\prune TL=\replace T{\prune UL}U$,
  then
  $\ctree(\prune TL)=\ctree(T)\cdot\frac{\ctree(\prune
    UL)}{\ctree(U)}$ and we obtain (4).

  For (5), just note that if $\prune TL=T\setminus U$, then
  $\ctree(T)\veq\ctree(\prune TL)\ctree(U)$, and since
  $\ctree(U)\prec1$ we obtain $\ctree(T)\prec\ctree(\prune TL)$; if
  instead $\prune TL=T[\prune UL/U],$ by induction we have
  $\ctree(U)\prec\ctree(\prune UL)$ and we reach the same conclusion
  using (4).
\end{proof}

\begin{lem}
  \label{lem:RUless}Let $T$ be an admissible tree and $U$ be a proper
  descendant of $T$. Then $\treeroot(U)\succ1$, and if $U'$ is a
  proper descendant of $U$ we have
  $1\prec\treeroot(U')^{n}\prec\treeroot(U)$ for every $n\in\N$.
\end{lem}

\begin{proof}
  Suppose first that $U$ is a child of $T$. Write
  $\treeroot(T)=re^{\gamma}$, so that $\treeroot(U)$ is a term of
  $\gamma=\vell(\treeroot(T))$.  Since $\gamma\in\J,$ $\treeroot(U)$
  is of the form $se^{\delta}$ with $0<\delta\in\J$, so
  $\treeroot(U)\vgreater1$, proving the first conclusion. Moreover, it
  follows that $\delta^{n}\prec e^{\delta}\veq\treeroot(U)$ for all
  $n\in\N$. If now $U'$ is a child of $U$, then $\treeroot(U')$ is a
  term of $\delta$, so
  $\treeroot(U')^{n}\vleq\delta^{n}\vless\treeroot(U)$, while by the
  previous argument $\treeroot(U')\vgreater1$. The general conclusion
  with $U$ a descendant of $T$ and $U'$ a descendant of $U$ now
  follows by transitivity of $\vleq$.
\end{proof}

\begin{prop}
  \label{prop:prune-branch}Let $L$ be a leaf in a tree $T$ of size
  $>1$ and let $U$ be the minimal child of $T$ with leaf $L$. Then
  \[
    \ctree(T)\veq\ctree(\prune TL)\cdot\ctree(L)\cdot
    t\quad\textrm{where}\quad 1 \vleq t\preceq c(\treeroot(U))^{2}.
  \]
\end{prop}

\begin{proof}
  We work by induction on $\size(T)$.

  Case 1. If $\size(U)=1$ (namely $U=L$), then $\prune TL=\remove TL$
  and $\ctree(T)\veq\ctree(\prune TL)\cdot\ctree(L)$, so it suffices
  to take $t=1$.

  Case 2. Assume $\size(U)>1$ and $\ctree(\prune UL)\succeq1$. Then
  $\prune TL=\remove TU$, and therefore
  $\ctree(T)\veq\ctree(\prune TL)\cdot\ctree(U)$.  We may assume by
  induction that
  $\ctree(U)\veq\ctree(\prune UL)\cdot\ctree(L)\cdot u$, where
  $1\vleq u\preceq c(\treeroot(U'))^{2}$ and $U'$ is the minimal child
  of $U$ with leaf $L$. Substituting we obtain
  \[
    \ctree(T)\veq\ctree(\prune TL)\cdot\ctree(L)\cdot\ctree(\prune
    UL)\cdot u.
  \]
  By \prettyref{lem:children} we have
  $\ctree(\prune UL)\preceq c(\treeroot(\prune UL))=c(\treeroot(U))$,
  and by \prettyref{lem:RUless}
  $u\preceq c(\treeroot(U'))^{2}\prec c(\treeroot(U))$, hence we can
  take $t:=\ctree(\prune UL)\cdot u$.

  Case 3. Finally, assume $\size(U)>1$ and $\ctree(\prune UL)\prec1$.
  Then $\prune TL=\replace T{\prune UL}U$, and by
  \prettyref{lem:prune} we have
  $\ctree(T)=\ctree(\prune TL)\cdot\frac{\ctree(U)}{\ctree(\prune
    UL)}$.  By inductive hypothesis we have
  $\ctree(U)\veq\ctree(\prune UL)\cdot\ctree(L)\cdot u$, where
  reasoning as above we have $1\vleq u\vless c(\treeroot(U))$.
  Substituting we get
  \[
    \ctree(T)\veq\ctree(\prune TL)\cdot\ctree(L)\cdot u,
  \]
  hence we can take $t=u$.
\end{proof}

\subsection{No bad sequences}

We can finally prove that there are no bad sequences at all in any
$\admtree(x)$.

\begin{prop}
  \label{prop:no-bad-seq-1}Let $x\in\Delta_{\alpha+1}$. If
  $(T_{i})_{i\in\N}$ is a bad sequence in $\admtree(x)$, then there are
  a bad sequence $(S_{j})_{j\in\N}$ in $\admtree(x)$ and some $k\in\N$
  such that $\size(S_{0})<\size(T_{k})$,
  $\treeroot(S_{0})=\treeroot(T_{k})$ and
  $\ctree(S_{0})\succ\ctree(T_{k})$.
\end{prop}

\begin{proof}
  By \prettyref{prop:type-A-or-B} and \prettyref{prop:type-A-not-bad},
  there is a subsequence $(P_{j})_{j\in\N}$ of $(T_{i})_{i\in\N}$ of type
  (B). Recall that by definition of type (B), $\size(P_{j})>0$ for all
  $j\in\N$.

  Write $P_{j}=\tree{r_{j}e^{\gamma_{j}}}{n_{j}}{\tau_{j}}$. Let
  $L_{0}$ be a leaf of $P_{0}$. For $j\geq1$, let $U_{j}$ be a child of
  $P_{j}$ with $\treeroot(U_{j})\vleq\gamma_{j-1}-\gamma_{j}$, which
  exists by definition of type (B), and let $L_{j}$ be a leaf of
  $U_{j}$.  We may then assume that $U_{j}$ is the \emph{minimal} child
  with leaf $L_{j}$ (if not, just replace $U_{j}$ with the minimal child
  $U$ with leaf $L_{j}$, and observe that the condition
  $\treeroot(U)\preceq\gamma_{j-1}-\gamma_{j}$ is still satisfied
  because $\treeroot(U)\vleq\treeroot(U_{j})$).

  We can write $L_{j}=\tree{\lambda_{j}}0{s_{j}}$, where
  $\lambda_{j}\in\Delta$ and
  $s_{j}=\ctree(L_{j})\in\term(c_{0}(\lambda_{j}))$. By
  \prettyref{lem:RUless} we have $\lambda_{j}\vleq\treeroot(U_{j})$;
  therefore, since $c$ preserves $\preceq$ by
  \prettyref{prop:cmonotone},
  \[
    c(\lambda_{j})\preceq c(\treeroot(U_{j}))\preceq
    c(\gamma_{j-1}-\gamma_{j})
  \]
  for all $j\geq1$.

  By \prettyref{lem:finer-large-gaps}, we may extract a further
  subsequence of $(P_{j})_{j\in\N}$ and assume that for all $j\in\N$ we
  have $\left(\frac{s_{j+1}}{s_{j}}\right)\prec c(\lambda_{j+1})^{2}$, so
  \[
    \left(\frac{s_{j+1}}{s_{j}}\right)\preceq
    c(\gamma_{j}-\gamma_{j+1})^{2}.
  \]
  Now let $S_{j}:=\prune{P_{j}}{L_{j}}$, which is well defined since
  $\size(P_{j})>0$ for all $j\in\N$. We shall prove that
  $(S_{j})_{j\in\N}$ has the desired properties.

  By \prettyref{prop:prune-branch}, for all $j\in\N$ we have
  \[
    \ctree(P_{j})=\ctree(\prune{P_{j}}{Lj})\cdot\ctree(L_{j})\cdot
    t_{j}=\ctree(\prune{P_{j}}{L_{j}})\cdot s_{j}\cdot t_{j}
  \]
  where $1\vleq t_{j}\vleq c(\treeroot(U_{j}))^{2}$ for all $j\in\N$.  In
  particular,
  $\frac{t_{j+1}}{t_{j}}\preceq t_{j+1}\vleq c(\treeroot(U_{j+1}))^{2}$, so
  \[
    \frac{t_{j+1}}{t_{j}}\preceq c(\gamma_{j}-\gamma_{j+1})^{2}.
  \]
  It follows that
  \[
    \frac{\ctree(\prune{P_{j}}{L_{j}})}{\ctree(\prune{P_{j+1}}{L_{j+1}})}
    =
    \frac{\ctree(P_{j})}{\ctree(P_{j+1})}\cdot\frac{s_{j+1}}{s_{j}}\cdot\frac{t_{j+1}}{t_{j}}
    \vleq c(\gamma_{j}-\gamma_{j+1})^{4} \cdot
    \frac{\ctree(P_{j})}{\ctree(P_{j+1})}.
  \]
  Since $(P_{j})_{j\in\N}$ is bad, for all $j,n\in\N$ we have
  \[
    \left(\frac{\ctree(P_{j})}{\ctree(P_{j+1})}\right)^{n} \preceq
    \frac{c(\treeroot(P_{j}))}{c(\treeroot(P_{j+1}))}.
  \]
  Likewise, for all $j,n\in\N$ we also have
  \[
    \left(c(\gamma_{j}-\gamma_{j+1})\right)^{n} \preceq
    e^{c(\gamma_{j}-\gamma_{j+1})}\veq\frac{c(\treeroot(P_{j}))}{c(R(P_{j+1}))}
  \]
  using \prettyref{lem:children}, \prettyref{prop:cmonotone} and the
  fact that $\gamma_{j}-\gamma_{j+1}\vgreater1$. It follows that for all
  $j,n\in\N$ we have
  \[
    \left(\frac{\ctree(\prune{P_{j}}{L_{j}})}{\ctree(\prune{P_{j+1}}{L_{j+1}})}\right)^{n}
    \vleq \frac{c(\treeroot(P_{j}))}{c(\treeroot(P_{j+1}))}.
  \]
  Recalling that $\treeroot(\prune{P_{j}}{L_{j}})=\treeroot(P_{j})$ for
  all $j\in\N$, it follows that
  $(S_{j})_{j\in\N}=(\prune{P_{j}}{L_{j}})_{j\in\N}$ is another bad
  sequence in $\admtree(x)$.

  To conclude, let $k\in\N$ be such that $T_{k}=P_{0}$. By construction,
  $\size(S_{0})=\size(\prune{P_{0}}{L_{0}})<\size(P_{0})=\size(T_{k})$, and
  by \prettyref{lem:prune},
  $\ctree(S_{0})=\ctree(\prune{P_{0}}{L_{0}})\succ\ctree(P_{0})=\ctree(T_{k})$,
  as desired.
\end{proof}

\begin{prop}
  \label{prop:no-bad-seq}Let $x\in\Delta_{\alpha+1}$. Then
  $\admtree(x)$ contains no bad sequences.
\end{prop}

\begin{proof}
  Suppose by contradiction that there is a bad sequence of trees in
  $\admtree(x)$. Among all such bad sequences, let $(T_{i})_{i\in\N}$ be
  the one such that $\size(T_{0})$ is minimal, and fixed $T_{0}$,
  $\size(T_{1})$ is minimal, and so on. By
  \prettyref{prop:no-bad-seq-1}, there is another bad sequence
  $(S_{j})_{j\in\N}$ in $\admtree(x)$ and some $k\in\N$ such that
  $\size(S_{0})<\size(T_{k})$, $\treeroot(S_{0})=\treeroot(T_{k})$ and
  $\ctree(S_{0})\vgreater\ctree(T_{k})$.

  We observe that
  \[
    T_{0},T_{1},\dots,T_{k-1},S_{0},S_{1},\dots
  \]
  is again a bad sequence in $\admtree(x)$. Indeed, it suffices to
  note that for all $n\in\N$ we have
  \[
    \left(\frac{\ctree(T_{k-1})}{\ctree(S_{0})}\right)^{n} \vless
    \left(\frac{\ctree(T_{k-1})}{\ctree(T_{k})}\right)^{n} \vleq
    \frac{c(\treeroot(T_{k-1}))}{c(\treeroot(T_{k}))} =
    \frac{c(\treeroot(T_{k-1}))}{c(\treeroot(S_{0}))}.
  \]
  However, since $\size(S_{0})<\size(T_{k})$, this contradicts our
  minimality assumption. Therefore, there are no bad sequences in
  $\admtree(x)$, as desired.
\end{proof}

By \prettyref{prop:no-bad-implies-summability}, this completes the
proof of \prettyref{lem:summability}, as desired.

\subsection*{Acknowledgments}

We thank the anonymous referee for the very careful report.

\end{document}